\documentclass[12pt,final]{siamltex}
\usepackage{amsmath}
\usepackage{amssymb}
\usepackage[normalem]{ulem}
\usepackage{multirow}

\DeclareMathAlphabet\mathbfcal{OMS}{cmsy}{b}{n}
\usepackage{color}
\definecolor{grb}{rgb}   {0.0,   0.99,   0.5 } 
\definecolor{mg}{rgb}   {0.85,  0.,    0.85} 
\newcommand{\Bk}{\color{black}}
\newcommand{\Rd}{\color{red}}

\newcommand{\Question}[1]{\marginpar{}}
\renewcommand{\Question}[1]{%
            \marginpar{\flushleft\scriptsize\bfseries\upshape#1}}
\usepackage{amsfonts,graphicx,psfrag}

\newcommand{\bfn}{{\bf n}}

\newcommand{\real}{{\rm I\!R}}

\newcommand{\Om}{{\Omega}}

\def\Aa{{\boldsymbol a}}

\def\n{{\boldsymbol n}}

\def\x {{\boldsymbol x}}

\newcommand{\qed}{\rule{0.08in}{0.08in}}

\newenvironment{theorem*}{{\bf Theorem}\em}{\rm\mbox{}}

\newtheorem{remark}[theorem]{Remark}

\begin{document}

\title{Convergence of a finite element method for degenerate two-phase flow in porous media}
%Enriched Galerkin for Stokes}

\author{
 Vivette Girault \thanks{Paris VI, visiting professor at Rice}\and
  Beatrice M.  Rivi\`ere\thanks{Department of Computational and Applied
    Mathematics, Rice University, Houston, TX 77005. Supported in part
    by NSF-DMS } \and Loic Cappanera}  

\date{\today}
\maketitle

\begin{abstract}
A finite element method with mass-lumping and flux upwinding, is formulated for solving the immiscible two-phase flow problem in porous media.  The method approximates directly the wetting phase pressure and saturation, which are the primary unknowns. The discrete saturation satisfies a maximum principle. Theoretical convergence  is proved via a compactness argument. The proof is convoluted because of the degeneracy of the phase mobilities and the unboundedness of the capillary pressure.
\end{abstract}
%\begin{keywords} 
%\end{keywords}

\setcounter{section}{0}

\section{Introduction}

\label{sec:intro}

This work discretizes on a suitable mesh a degenerate two-phase flow system set in a polyhedral domain by a new finite element scheme that directly approximates the wetting phase pressure and saturation, similar to the formulation proposed in \cite{Forsyth1991}. Mass lumping is used to compute the integrals and a suitable upwinding is used to compute the flux, guaranteeing that the discrete saturation satisfies a maximum principle. The resulting system of discrete equations is a finite element analogue of the finite volume scheme introduced and analyzed by Eymard {\it et al} in the seminal work~\cite{Eymard2003}. From the point of view of implementation, the advantage of finite elements is that they only use nodal values and a single simplicial mesh. In particular, no orthogonality property is required between the faces and the lines joining  the centers of control volumes. From a theoretical point of view, owing that the finite element scheme is based on functions, some steps in its convergence analysis are simpler, but nevertheless the major difficulty in the analysis consists in proving sufficient a priori estimates in spite of the degeneracy. By following closely~\cite{Eymard2003}, the degeneracy is remediated by reintroducing in the proofs discrete artificial pressures. From there, convergence of the numerical solutions is shown  via a compactness argument.  

Incompressible two-phase flow is a popular and important multiphase flow model in reservoirs for the oil and gas industry.  Based on conservation laws at the continuum scale, the model assumes the existence of a representative elementary volume. Each wetting phase and non-wetting phase saturation satisfies a mass balance equation and each phase velocity follows the generalized Darcy law~\cite{peaceman2000fundamentals,azizpetroleum}. The equations of the mathematical model read
\begin{equation}
\label{eq:orignalpb}
\begin{split}
\partial_t (\varphi s_w) - \nabla \cdot (\eta_w \nabla p_w)  &= f_w(s_\mathrm{in}) \bar{q} - f_w(s_w) \underline{q},\\
\partial_t (\varphi s_o) - \nabla \cdot (\eta_o \nabla p_o)  &= f_o(s_\mathrm{in}) \bar{q} - f_o(s_w) \underline{q},\\
p_c(s_w) = p_o - p_w,&\quad s_w+ s_o = 1,
\end{split}
\end{equation}
complemented by initial and boundary conditions.
Here $p_w, s_w,\eta_w, f_w $, (respectively, $p_o, s_o, \eta_o, f_o$), are the pressure, saturation, mobility and fractional flow of the wetting (respectively non-wetting) phase, $\varphi$ is the 
porosity, $s_\mathrm{in}$ is a given input saturation,  and $\bar{q}, \underline{q}$ are given flow rates.  The capillary pressure, $p_c$, is a given function that depends nonlinearly on the saturation. Because the phase mobilities are degenerate when they are evaluated at the residual phase saturations and the derivative of the capillary pressure is unbounded, this system of two coupled nonlinear partial differential equations has coefficients that vanish in parts of the domain; this degeneracy makes the numerical analysis challenging.

At the continuous level, this problem has several equivalent formulations, see~\cite{chavent1986mathematical}. They are linked to the choice of primary unknowns selected among wetting phase and non-wetting phase pressure and saturation, or capillary pressure \cite{Helmig,Bastian2014}. A good state of the art can be found in the reference~\cite{arbogast92}. Up to our knowledge, the mathematical analysis of the system of equations was first done in~\cite{KroenerLuckaus,alt1985nonsteady}.  An equivalent formulation of the model, based on Chavent's global pressure that removes the degeneracy, was analyzed in~\cite{Chen2001,ChenEwing}. Since then, the global pressure formulation has been discretized and analyzed in many references \cite{ohlberger1997convergence,chen2001error,Michel03}, but unfortunately, this formulation is not used in engineering practice because the global pressure is not a physical unknown. Otherwise, with one exception, the numerical analysis of the discrete version of \eqref{eq:orignalpb},
has always been done under unrealistic assumptions that cannot be checked at the discrete level
\cite{douglas1983finite,EpshteynRiviere2009}.  Related to this line of work, the discretization of a degenerate
parabolic equation has been studied in the literature \cite{arbogastwheeler,yotov1997,woodward2000,eymard2006}. 
 The only paper that performs the complete numerical analysis of the discrete degenerate two-phase flow system written as above (i.e., in the form used by engineers) is 
the analysis on finite volumes done in reference~\cite{Eymard2003}. This motivates our extension of this work to finite elements. 

The remaining part of this introduction makes precise problem \eqref{eq:orignalpb}. The numerical scheme  is developed in Section \ref{sec:scheme}. Because of the nonlinearity and degeneracy of its equations, existence of a discrete solution requires that the discrete wetting phase saturation satisfies a maximum principle. This is the first object of Section \ref{sec:1staprioribdds}, the second one being basic a priori pressure estimates, after which existence is shown in Section \ref{sec:existence}. The most technical part, done in Section \ref{sec:addpressbdd}, is the derivation of an unconditional bound on an auxiliary pressure, which allows to use a compactness argument. Weak and strong convergences are proved in Section \ref{sec:Convg} and the equations satisfied by the limit are identified in Section \ref{sec:identlim}, thus confirming existence of a solution of the weak formulation \eqref{eq:Var}. Numerical results are presented in Section \ref{sec:Valid}.

%-------------------------- 

\subsection{Model problem}
\label{sec:problem}

Let $\Omega \subset \real^d$, $d=2$ or $3$, be a bounded connected Lipschitz domain with boundary $\partial \Omega$ and unit exterior normal $\n$, and let $T$ be a final time. With the last relation in \eqref{eq:orignalpb}, $s_w$ is the only unknown saturation; so we set $s=  s_w$, and rewrite \eqref{eq:orignalpb} almost everywhere in $\Omega \times ]0,T[$ as
\begin{eqnarray}
\partial_t (\varphi s) - \nabla \cdot (\eta_w \nabla p_w)  = f_w(s_\mathrm{in}) \bar{q} - f_w(s) \underline{q} \label{eq:2ph1}\\
-\partial_t (\varphi s) - \nabla \cdot (\eta_o \nabla p_o)  = f_o(s_\mathrm{in}) \bar{q} - f_o(s) \underline{q}, \label{eq:2ph2}
\end{eqnarray}
complemented by a natural boundary condition almost everywhere on $\partial\Omega \times ]0,T[$
\begin{equation}
\label{eq:bdc}
\eta_w \nabla p_w\cdot \bfn=0, \quad \eta_o \nabla p_o \cdot \bfn = 0,
\end{equation}
and an initial condition almost everywhere in $\Omega$
\begin{equation}
\label{eq:initc}
s_w(\cdot, 0) = s_w^0, \quad 0 \le s_w^0 \le 1.
\end{equation}
%The mobilities $\eta_w, \eta_o$ are defined by:
%\[
%\eta_w = \frac{ k_{rw}K }{\mu_w},\quad
%\eta_o = \frac{ k_{ro}K }{\mu_o}.
%\]
The fractional flows are related to the mobilities by 
\begin{equation}
\label{eq:fluidfract}
f_w = \frac{\eta_w}{\eta_w+\eta_o},\quad f_o = 1 - f_w.
\end{equation}
Recall that the phase saturations sum up to 1 and the phase pressures are related by
\begin{equation}
\label{eq:cap.pressure}
p_c(s_w) = p_o - p_w.
\end{equation}

The first part of this work, up to Section \ref{subsec:Bddgrad.g}, is done under the following basic assumptions:

{\bf Assumptions}
\begin{itemize}
\item The porosity $\varphi$ is piecewise constant in space, independent of time, positive, bounded and uniformly bounded away from zero. 
\item The mobility of the wetting phase  $\eta_w\geq 0$ is continuous and increasing.
The mobility of the non-wetting phase $\eta_o\geq 0$ is continuous and  decreasing. This implies that the function $f_w$ is increasing
and the function $f_o$ is decreasing.
\item There is a positive constant $\eta_\ast$ such that
\begin{equation}\label{eq:lowerboundetas}
\eta_w(s) + \eta_o(s) \geq \eta_\ast, \quad \forall s \in [0,1].
\end{equation}
\item The  capillary pressure $p_c$ is a continuous, strictly decreasing function in $W^{1,1}(0,1)$.
\item  The flow rates at the injection and production wells, $\bar{q}, \underline{q}\in  L^2(\Omega \times ]0,T[)$ satisfy
\begin{equation}\label{eq:flowrate}
\bar{q} \geq 0, \quad \underline{q} \geq 0, \quad \int_\Omega \bar{q} = \int_\Omega \underline{q}.
\end{equation}
\item The prescribed input saturation $s_\mathrm{in}$ satisfies almost everywhere in $\Omega \times ]0,T[$
\begin{equation}\label{eq:inputsat}
0 \le s_\mathrm{in} \le 1 \Bk .
\end{equation} 
\end{itemize}

\Rd Since $p_c$, $\eta_\alpha$, $f_\alpha$, $\alpha = w,o$ are bounded above and below, it is convenient to extend them continuously by constants to $\real$. \Bk

Although the numerical scheme studied below does not discretize the global pressure, following~\cite{Eymard2003}, its convergence proof uses a number of auxiliary functions related to the global pressure. First,
 we  introduce  the primitive $g_c$ of $p_c$,\Bk
\begin{equation}
\label{eq:gc}
\forall x\in [0,1],\quad g_c(x)=\int_x^1 p_c(s) ds.
\end{equation}
Since $p_c$ is a continuous function on $[0,1]$, the function $g_c$ belongs to $\mathcal{C}^1([0,1])$. Next, we introduce the auxiliary pressures  $p_{wg}$, $p_{wo}$,  and $g$, \Bk
\begin{equation}\label{eq:defpwgpog}
\forall x\in [0,1],\quad p_{wg}(x) = \int_0^x f_o(s) p_c'(s) ds, \quad
p_{og}(x) = \int_0^x f_w(s) p_c'(s) ds,
\end{equation}
\begin{equation}
\label{eq:g}
\forall x \in [0,1],\quad g(x)= -\int_0^x \frac{\eta_w(s) \eta_o(s)}{\eta_w(s) + \eta_o(s)}p^\prime_c(s) ds. 
\end{equation}
Owing to \eqref{eq:fluidfract},
\begin{equation}
\label{eq:pwg+pog}
\forall x\in [0,1],\quad p_{wg}(x) + p_{og}(x) = \int_0^x p_c'(s) ds = p_c(x) - p_c(0).
\end{equation}
 Moreover, the derivative of $g$ satisfies formally the identities \Bk
\begin{equation}
\label{eq:pwgprime+gprime}
\forall x \in [0,1],\quad \eta_{ \alpha \Bk}(x) p^\prime_{\alpha  g}(x) +g^\prime(x) = 0, \quad   \alpha =w, o.
\end{equation} 

%-----------------------

\subsection{Weak variational formulation}
\label{subsec:weakvar}

 By multiplying \eqref{eq:2ph1} and \eqref{eq:2ph2} with a smooth function $v$, say 
$v \in {\mathcal C}^1(\Omega \times [0,T])$ that vanishes at $t = T$, applying Green's formula in time and space, and using the boundary and initial conditions \eqref{eq:bdc} and \eqref{eq:initc}, we formally derive a weak variational formulation
\begin{align*}
-\int_0^T \int_\Omega \varphi\, s\, \partial_t v + \int_0^T \int_\Omega \eta_w \nabla\,p_w \cdot \nabla\,v = &\int_\Omega \varphi \, s^0 v(0) + \int_0^T \int_\Omega\big(f_w(s_\mathrm{in})\bar q - f_w(s)\underline{q}\big)v,\\
\int_0^T \int_\Omega \varphi\, s\, \partial_t v + \int_0^T \int_\Omega \eta_o \nabla\,p_o \cdot \nabla\,v = &-\int_\Omega \varphi\, s^0 v(0) + \int_0^T \int_\Omega\big(f_o(s_\mathrm{in})\bar q - f_o(s)\underline{q}\big)v.
\end{align*}
But in general, the pressures are not sufficiently smooth to make 
this formulation meaningful and following~\cite{chavent1986mathematical}, by using \eqref{eq:pwgprime+gprime}, it is rewritten in terms of the artificial pressures, 
\begin{equation}
\label{eq:Var}
\begin{split}
-\int_0^T \int_\Omega \varphi\, s\, \partial_t v + \int_0^T \int_\Omega \big(\eta_w &\nabla (p_w + p_{wg}(s)) + \nabla\,g(s)\big)\cdot \nabla\,v 
= \int_\Omega \varphi \, s^0 v(0)\\
& + \int_0^T \int_\Omega\big(f_w(s_\mathrm{in})\bar q - f_w(s)\underline{q}\big)v,\\
\int_0^T \int_\Omega \varphi\, s\, \partial_t v + \int_0^T \int_\Omega \big(\eta_o &\nabla (p_o - p_{og}(s)) -\nabla\,g(s)\big)\cdot \nabla\,v 
= -\int_\Omega \varphi\, s^0 v(0) \\
&+ \int_0^T \int_\Omega\big(f_o(s_\mathrm{in})\bar q - f_o(s)\underline{q}\big)v.
\end{split}
\end{equation}
The two formulas coincide when the pressures are slightly more regular.
With the above assumptions, problem \eqref{eq:Var} has been analyzed in reference~\cite{alt1985nonsteady}, where it is shown that it has a solution $s$ in $L^\infty(\Omega \times ]0,T[)$ with $g(s)$ in $L^2(0,T;H^1(\Omega))$, $p_\alpha$, $\alpha =w,o$, in $L^2(\Omega \times ]0,T[)$ with both $p_w + p_{wg}(s)$ and $p_o - p_{og}(s)$ in $L^2(0,T;H^1(\Omega))$.

%====================

\section{Scheme}
\label{sec:scheme}

From now on, we assume that $\Omega$ is a polygon ($d=2$) or Lipschitz polyhedron ($d=3$) so it can be entirely meshed.

\subsection{Meshes  and discretization spaces}
\label{subsec:mesh}

The mesh  ${\mathcal T}_h$  is  a  regular family of simplices $K$, with a constraint on the angle  that will be used to enforce the maximum principle: each angle is not larger than $\pi/2$, see \cite{Casado2006}.  This is easily constructed in 2D. In 3D, since we only investigate convergence we can embed the domain in a triangulated box. Moreover, since the porosity $\varphi$ is a piecewise constant, to simplify we also assume that the mesh is such that $\varphi$ is a constant per element. 
The parameter $h$ denotes the mesh size i.e., the maximum diameter of the simplices.
On this mesh, we consider the standard finite element space of order one
\begin{equation}
\label{eq:Xh}
X_h = \{v_h \in {\mathcal C}^0(\bar \Omega)\,;\, \forall K \in {\mathcal T}_h, v_h|_K \in {\mathbb P}_1\}.
\end{equation}
Thus the dimension of $X_h$ is the number of nodes, say $M$, of ${\mathcal T}_h$.
\Bk
Let $\phi_i$ be the Lagrange basis function, that is piecewise linear, and takes the value $1$ at node $i$ and the value 0 elsewhere.  As usual, the Lagrange interpolation operator $I_h \in {\mathcal L}({\mathcal C}^0(\bar \Omega);X_h)$ is defined by
\begin{equation}
\label{eq:Ih}
\forall v \in {\mathcal C}^0(\bar \Omega),\quad I_h(v) = \sum_{i=1}^{M} v_i \phi_i,
\end{equation}
where $v_i$ is the value of $v$ at the node of index $i$. 
It is easy to see that under the mesh condition, we have
\begin{equation}
\label{eq:acute}
\forall K, \quad \int_K \nabla \phi_i \cdot\nabla\phi_j \leq 0, \quad \forall i\neq j.
\end{equation}
For a given node $i$, we denote by $\Delta_i$ the union of elements sharing the node $i$  and by $\mathcal{N}(i)$ the set of indices of all the nodes in $\Delta_i$.
In the spirit of~\cite{GuermondPopov16}, we define
\begin{equation}
\label{eq:cij}
c_{ij} = \int_{\Delta_i \cap \Delta_j} | \nabla \phi_i \cdot\nabla \phi_j|, \quad \forall i,j.
\end{equation}
Recall that the trapezoidal rule on a triangle or a tetrahedron $K$ is
\[
\int_K f \approx \frac{1}{d+1}|K| \sum_{\ell=1}^{d+1} f_{i_\ell},
\]
where $f_{i_\ell}$ is the value of the function $f$ at the $\ell^{th}$ node (vertex), with global number $i_\ell$,  of  $K$.
For any region ${\mathcal O}$, the notation $\vert \mathcal{O}\vert$ means the measure (volume) of $\mathcal{O}$.

We define
\[
m_i = \frac{1}{d+1} \sum_{K\in\Delta_i}  \vert K\vert = \frac{1}{d+1} |\Delta_i|,
\]
 and taking into account the porosity $\varphi$, we  define  more generally 
\[
\tilde{m}_i(\varphi) = \frac{1}{d+1} \sum_{K\in\Delta_i} \varphi|_{K} \vert K\vert,
\]
 so that $m_i = \tilde{m}_i(1)$. It is well-known that the trapezoidal rule defines a norm on $X_h$, $\|\cdot \|_h$, uniformly equivalent to the $L^2$ norm.
Let $U_h \in X_h$ and write 
\[
U_h = \sum_{i=1}^{M} U^i \phi_i .
\]
 The discrete \Bk $L^2$ norm  associated with the trapezoidal rule is
\[
\Vert U_h \Vert_h = \left(\sum_{i=1}^M m_i |U^i|^2\right)^{\frac{1}{2}}.
\]
There exist positive constants $\underline{C}$ and $\overline{C}$, independent of $h$ and $M$, such that
\begin{equation}
\label{eq:equivnorm}
\forall U_h \in X_h,\quad \underline{C}\,\|U_h\|^2_{L^2(\Om)} \le \|U_h\|^2_h \le \overline{C}\,\|U_h\|^2_{L^2(\Om)}.
\end{equation}
 This is also true for other piecewise polynomial functions, but with possibly different constants. \Bk
The scalar product associated with this norm is denoted by $(\cdot,\cdot)_h$,
\begin{equation}
\label{eq:scalprod}
\forall U_h,V_h \in X_h,\quad (U_h,V_h)_h = \sum_{i=1}^M m_i U^i V^i.
\end{equation}
By analogy, we introduce the notation
\begin{equation}
\label{eq:weightscalprod}
\forall U_h,V_h \in X_h,\quad (U_h,V_h)^\varphi_h = \sum_{i=1}^M \tilde m_i(\varphi) U^i V^i.
\end{equation}
The assumptions on the porosity $\varphi$ imply that \eqref{eq:weightscalprod} defines a weighted scalar product associated with the weighted norm $\|\cdot \|_h^\varphi$,
$$\forall U_h \in X_h,\quad \|U_h \|_h^\varphi = \big((U_h,U_h)^\varphi_h\big)^{\frac{1}{2}},
$$
that satisfies the analogue of \eqref{eq:equivnorm}, with the same constants $\underline{C}$ and $\overline{C}$,
\begin{equation}
\label{eq:equivnormphi}
\forall U_h \in X_h,\quad \underline{C}\,(\min_\Omega \varphi)\,\|U_h\|^2_{L^2(\Om)} \le  \big(\|U_h \|_h^\varphi\big)^2 \le 
\overline{C}\,(\max_\Omega \varphi)\,\|U_h\|^2_{L^2(\Om)}.
\end{equation}

%--------------
\subsection{Motivation of the space discretization}
\label{subsec:motiv}

While discretizing the time derivative is fairly straightforward, discretizing the space derivatives  is more delicate because we need a scheme that is consistent and satisfies the maximum principle for the saturation. For the moment, we freeze the time variable and focus on consistency in space. 
First, we recall a standard property of functions of $X_h$ on meshes satisfying \eqref{eq:acute}.

\begin{proposition}
\label{pro:equalityUV}
Under condition  \eqref{eq:acute},  the following identities holds for all $U_h$ and $V_h$ in $X_h$, with $c_{ij}$ defined in \eqref{eq:cij}:
\begin{equation}
\label{eq:nablauv}
\int_\Omega\nabla\,U_h \cdot \nabla\,V_h = -\sum_{i=1}^M U^i \sum_{j\neq i, j\in {\mathcal N}(i)}  c_{ij} \big(V^j - V^i\big)
= \frac{1}{2}\sum_{i=1}^M \sum_{j\neq i, j\in {\mathcal N}(i)} c_{ij}\big(U^j - U^i\big)\big(V^j - V^i\big).
\end{equation}
\end{proposition}

\begin{proof}
The first equality is obtained by using \eqref{eq:acute}, \eqref{eq:cij} and  the fact that
$$\sum_{j=1}^M \phi_j =1,$$
as in \cite{eymard2000finite} (Section 12.1).
%By definition,
%$$
%\int_\Omega \nabla\,U_h \cdot \nabla\,V_h = \int_\Omega \sum_{i,j =1}^M  U^i V^j (\nabla\,\phi_i \cdot \nabla\,\phi_j) = \sum_{i,j=1}^M \int_{\Delta_i \cap \Delta_j} U^i V^j(\nabla\,\phi_i \cdot \nabla\,
%\phi_j) .
%%$$
%Now, for any $i$,
%$$\sum_{j=1}^M \int_{\Delta_i \cap \Delta_j} \nabla\,\phi_i \cdot \nabla\,
%\phi_j =  \int_\Omega \nabla\,\phi_i\cdot\big( \sum_{j=1}^M \nabla\,
%\phi_j\big) =0,
%$$
%because
%$$\sum_{j=1}^M \phi_j =1\quad \mbox{implies} \quad  \sum_{j=1}^M \nabla\,\phi_j =0.
%$$
%Thus, \eqref{eq:cij} and \eqref{eq:acute} imply
%$$\int_{\Delta_i } |\nabla\,\phi_i|^2 = -\sum_{1\le j\le M, j\ne i} \int_{\Delta_i \cap \Delta_j} \nabla\,\phi_i \cdot \nabla\,\phi_j = \sum_{1\le j\le M, j\ne i} c_{ij}.
%$$
%This readily yields the first part of  \eqref{eq:nablauv}. 
For the second part, we use the symmetry of $c_{ij}$ and the anti-symmetry of $V^j - V^i$ to deduce that
$$-\sum_{i=1}^M U^i \sum_{j\neq i, j\in {\mathcal N}(i)}  c_{ij} \big(V^j - V^i\big) =
\frac{1}{2} \sum_{i=1}^M  \sum_{j\neq i, j\in {\mathcal N}(i)} c_{ij} \big(U^j - U^i\big)\big(V^j - V^i\big),
$$
which is the desired result.
\end{proof}

Note that $c_{ij}$ vanishes when $j \notin {\mathcal N}(i)$\Bk. Therefore, when there is no ambiguity it is convenient to write the above double sums on $i$ and $j$ with $i$ and $j$ running from $1$ to $M$. \Bk

As an immediate consequence of Proposition \ref{pro:equalityUV}, we have, by taking $V_h = U_h$,
\begin{equation}
\label{eq:normgrad}
\forall U_h \in X_h,\quad \|\nabla\,U_h\|_{L^2(\Om)} = \frac{1}{\sqrt{2}}\Big(\sum_{i,j=1}^M c_{ij} |U^j - U^i|^2\Big)^{\frac{1}{2}}.
\end{equation}

Now, we consider the case of the product of the gradients by a third function. Beforehand, we introduce the following notation: for indices $i$ and $j$ of two neighboring interior nodes, $\Delta_i \cap \Delta_j$ in two dimensions is the union of two triangles and in three dimensions the union of a number of tetrahedra bounded by a fixed constant, say $L$, determined by the regularity of the mesh. We shall use the following notation
 \begin{equation}
 \label{eq:cijk}
c_{ij,K}= \int_{K} |\nabla\,\phi_i \cdot \nabla\,\phi_j|,\quad w_K = \frac{1}{|K]}\int_{K} w.
\end{equation}
Note that
\begin{equation}
\sum_{K \subset \Delta_i \cap \Delta_j} c_{ij,K} = c_{ij}.
\label{eq:sumcijk}
\end{equation}
Then we have the following proposition:
\begin{proposition}
\label{pro:equality}
Let \eqref{eq:acute} hold. With the notation \eqref{eq:cijk}, the following identity holds for all $w$ in $L^1(\Omega)$: 
\begin{equation}
\label{eq:wnablaunablav}
\forall U_h,V_h \in X_h, \ \int_\Omega w \nabla\,U_h \cdot \nabla\,V_h = -\sum_{i=1}^M U^i \sum_{j=1}^M  \Big(\sum_{K \subset \Delta_i \cap \Delta_j} c_{ij,K} w_K\Big)\big(V^j - V^i\big),
\end{equation}
\end{proposition}

\begin{proof}
It is easy to prove that
\begin{equation}
\label{eq:dij}
 \int_\Omega  w \Bk \nabla\,U_h \cdot \nabla\,V_h = \sum_{i,j =1}^M d_{ij}  U^i V^j,
\end{equation}
where
\begin{equation}
\label{eq:dij1} 
d_{ij} = \int_{\Delta_i \cap \Delta_j} w (\nabla\,\phi_i \cdot \nabla\,
\phi_j) = \int_\Omega w (\nabla\,\phi_i \cdot \nabla\,
\phi_j).
\end{equation}
Again, we have for any $i$,
$$\sum_{j=1}^M d_{ij} = 0, \quad\mbox{and} \quad
d_{ii} = -\sum_{1\le j\le M, j\ne i} d_{ij},
$$
and by substituting this equality into \eqref{eq:dij}, we obtain
\begin{equation}
\label{eq:wnablaunablav1}
 \int_\Omega w \big(\nabla\,U_h \cdot \nabla\,V_h\big) = \sum_{i,j=1}^M U^i  d_{ij} \big(V^j- V^i\big).
 \end{equation}
But, in view of \eqref{eq:cijk} and \eqref{eq:dij1}, and since   $\nabla\,\phi_i \cdot \nabla\,\phi_j$ is a constant in each element $K$ contained in $\Delta_i \cap \Delta_j$,
\begin{equation}
\label{eq:dij2}
d_{ij} = - \sum_{K \subset \Delta_i \cap \Delta_j} c_{ij,K} w_K,
\end{equation}
and \eqref{eq:wnablaunablav} follows by substituting this equation into  \eqref{eq:wnablaunablav1}.
\end{proof}

Note that $d_{ij} = d_{ji}$ owing to \eqref{eq:dij2}.  
The first consequence of Proposition \ref{pro:equality} is that the right-hand side of \eqref{eq:wnablaunablav} is a consistent approximation of $(w, \nabla\,u \cdot \nabla\,v)$. 

\begin{proposition}
\label{pro:approx}
Let \eqref {eq:acute} hold, let $u$ and $v$ belong to $H^2(\Omega)$ and $w$ to $L^\infty(\Omega)$, and let  $U_h = I_h u$, $V_h = I_h v$ be defined by \eqref{eq:Ih}. Then, there exists a constant $C$, independent of $h$, $M$, $u$, $v$, and $w$, such that
\begin{equation}
\label{eq:uvapprox}
\Big|\int_\Omega w \nabla\,u \cdot \nabla\,v  + \sum_{i,j=1}^M U^i   \big(\sum_{K \subset \Delta_i \cap \Delta_j}c_{ij,K} w_K \big)  \big(V^j - V^i\big)\Big|\\
 \le C\,h\,\|w\|_{L^\infty(\Omega)} \|u\|_{H^2(\Omega)} \|v\|_{H^2(\Omega)}.
\end{equation}
\end{proposition}

\begin{proof}
In view of the identity \eqref{eq:wnablaunablav}, the left-hand side of \eqref{eq:uvapprox} is bounded as follows:
\begin{align*}
\Big|\int_\Omega w \big(\nabla\,u \cdot \nabla\,v & - \nabla\,U_h \cdot \nabla\,V_h\big) \Big|\le  \|w\|_{L^\infty(\Omega)}\\
& \times \Big( \|\nabla(u-U_h)\|_{L^2(\Omega)} \|\nabla\,v\|_{L^2(\Omega)} + \|\nabla(v-V_h)\|_{L^2(\Omega)}\|\nabla\,U_h\|_{L^2(\Omega)}\Big).
\end{align*}
From here, \eqref {eq:uvapprox} is a consequence of standard finite element interpolation error.
\end{proof}

Now, if $w$ is in $W^{1,\infty}(\Omega)$, then again, standard finite element approximation shows that there  exists a constant $C$, independent of $h$, $K \subset \Delta_i \cap \Delta_j$, and $w$, such that
\begin{equation}
\label{eq:wapprox}
\big\|w_K - w\big\|_{L^\infty(K)} \le C\,h\,|w|_{W^{1,\infty}(K)} \le C\,h\,  |w|_{W^{1,\infty}(\Omega)}.
\end{equation}
As a consequence, we will show that in the error formula \eqref{eq:uvapprox}, the average
$w_K$
can be replaced by any value of $w$ in $K$. Since all $K$ in  $\Delta_i \cap \Delta_j$ share the edge, say $e_{ij}$, whose end points are the nodes with indices $i$ and $j$, then we can pick the value of $w$ at any point, say $\tilde W^{i,j}$, of $e_{ij}$. At this stage, we choose this value freely, but we prescribe that it be symmetrical with respect to $i$ and $j$, i.e.,
\begin{equation}
\label{eq:Wsymij}
\tilde W^{i,j} = \tilde W^{j,i} .
\end{equation}
Then we have the following approximation result.

\begin{theorem}
\label{thm:mainapprox}
With the assumption and notation of Proposition \ref{pro:approx}, there exists a constant C, independent of $h$ and $M$, such that for all $u$, and $v$ in $H^2(\Omega)$ and $w$ in $W^{1,\infty}(\Omega)$,
\begin{equation}
\label{eq:consist}
\int_\Omega w \nabla\,u \cdot \nabla\,v  = - \sum_{i,j=1}^M U^i  c_{ij}\tilde W^{i,j}  \big(V^j - V^i\big) + R,
\end{equation}
for any arbitrary value $\tilde W^{i,j}$ of $w$ in the common edge $e_{ij}$ satisfying \eqref{eq:Wsymij}, and the remainder $R$ satisfies
\begin{equation}
\label{eq:bddR}
|R| \le C\,h\,|w|_{W^{1,\infty}(\Omega)} \|u\|_{H^2(\Omega)} \|v\|_{H^2(\Omega)}.
\end{equation}
\end{theorem}

\begin{proof}
We infer from \eqref{eq:sumcijk} and \eqref{eq:wnablaunablav}  that
$$
\int_\Omega w \big(\nabla\,U_h \cdot \nabla\,V_h\big) = - \sum_{i,j=1}^M U^i \big(V^j-V^i\big) \sum_{K \subset \Delta_i \cap \Delta_j} c_{ij,K} \big(w_K
 - \tilde W^{i,j}\big)
  - \sum_{i,j=1}^M U^i c_{ij} \big(V^j-V^i\big)\tilde W^{i,j}.
$$
Let
$$R_{ij} = \sum_{K \subset \Delta_i \cap \Delta_j} c_{ij,K} \big(w_K - \tilde W^{i,j}\big),
$$
which is symmetric in $i$ and $j$ by assumption \eqref{eq:Wsymij}.
As in Proposition \ref{pro:equalityUV}, the symmetry of $R_{ij}$ and the anti-symmetry of $V^j - V^i$, imply 
\begin{equation}
\label{eq:antisym}
- \sum_{i,j=1}^M U^i   R_{ij}  \big(V^j - V^i\big)
\le \frac{1}{2} \Big(\sum_{i,j=1}^M |R_{ij}| \big(U^j - U^i\big)^2\Big)^{\frac{1}{2}} \Big(\sum_{i,j=1}^M |R_{ij}| \big(V^j - V^i\big)^2\Big)^{\frac{1}{2}}.
\end{equation}
From the non negativity of $c_{ij,K}$, \eqref{eq:sumcijk}, and \eqref{eq:wapprox}, we infer that
$$|R_{ij}| \le \Big(\sum_{K\subset \Delta_i \cap \Delta_j} c_{ij,K}\Big)  C\,h\,|w|_{W^{1,\infty}(\Omega)} = c_{ij} C\,h\,|w|_{W^{1,\infty}(\Omega)}.
$$
Hence, with \eqref{eq:normgrad} and standard finite element approximation, 
$$
 \Big|\sum_{i,j=1}^M U^i   R_{ij}  \big(V^j - V^i\big)\Big| \le  C\,h\,|w|_{W^{1,\infty}(\Omega)} \|\nabla\,U_h\|_{L^2(\Om)} \|\nabla\, V_h\|_{L^2(\Om)} 
 \le C\,h\,|w|_{W^{1,\infty}(\Omega)} \|u\|_{H^2(\Omega)} \|v\|_{H^2(\Omega)}.
$$
The result follows by combining this inequality with \eqref{eq:uvapprox}.
\end{proof}

The above considerations show that
$$- \sum_{i,j=1}^M U^i  c_{ij}\tilde W^{i,j}  \big(V^j - V^i\big)\ \mbox{is a consistent approximation of}\ \int_\Omega w \nabla\,u \cdot \nabla\,v,
$$
for any symmetric choice of $\tilde W^{i,j}$ in   $e_{ij}$, the common edge of \Bk $\Delta_i \cap \Delta_j$. This will lead to the upwinded space discretization in the next subsection, see also~\cite{Michel03}.  Furthermore, for all real numbers $V^i$ and $\tilde W^{i,j}$ satisfying \eqref{eq:Wsymij}, $1 \le i,j \le M$, the symmetry of $c_{ij}$ and anti-symmetry of $V^j-V^i$ imply
\begin{equation}
\label{eq:sum0}
\sum_{i,j=1}^M  c_{ij}\tilde W^{i,j} (V^j-V^i) = 0.
\end{equation}

%-------------------

\subsection{Fully discrete scheme}
\label{subsec:time-space}

Let $\tau =\frac{T}{N}$ be the time step,  $t_n = n \tau$, the discrete times, $0 \le n \le N$. Regarding time, we shall use the standard $L^2$ projection $\rho_\tau$  defined on  $]t_{n-1}, t_n]$, \Bk for any function $f$ in $L^1(0,T)$,  by 
\begin{equation}
\label{eq:rhotau}
\rho_\tau(f)^n := \rho_\tau(f)|_{]t_{n-1}, t_n]} : = \frac{1}{\tau}\int_{t_{ n-1\Bk}}^{t_{ n \Bk }} f.
\end{equation}
Regarding space, we shall use a standard element-by-element $L^2$ projection $\rho_h$ as well as a 
nodal approximation operator $r_h$ defined at each node $\x_i$ for any function $g \in L^1(\Omega)$ by
\begin{equation}
\label{eq:rh}
r_h(g)(\x_i) = \frac{1}{|\Delta_i|}\int_{\Delta_i} g, \quad 1 \le i \le M,
\end{equation} 
and extended to $\Omega$ by $r_h(g) \in X_h$. The operator $\rho_h$ is defined for any $f$ in $L^1(\Omega)$ by $\rho_h(f)|_K = \rho_K(f)$ where,
in any element $K$,
\begin{equation}
\label{eq:rhoh}
\rho_K(f) = \frac{1}{|K|} \int_K f.
\end{equation}
 The initial saturation  $s_w^0$ is approximated by the operator $r_h$,
\begin{equation}
\label{eq:Sh0}
S_h^0 = r_h(s_w^0).
\end{equation}
The input saturation $s_\mathrm{in}$ is approximated  in space and time \Bk by
\begin{equation}
\label{eq:sinh}
s_\mathrm{in,h, \tau} = \rho_\tau(r_h(s_\mathrm{in})).
\end{equation}
Clearly, \eqref{eq:inputsat} implies \Rd in space and time \Bk
$$0 \le s_\mathrm{in,h, \Rd \tau \Bk} \le 1.
$$
In order to preserve \eqref{eq:flowrate}, the functions $\bar{q}$ and $\underline{q}$ are approximated by the functions $\bar{q}_{h,\tau}$ and $\underline{q}_{h,\tau}$ defined with $r_h$ and corrected as follows:
\begin{equation}
\label{eq:qhdef}
\bar{q}_{h,\tau} = \rho_\tau \left(r_h (\bar{q}) - \frac{1}{|\Omega|}\int_\Omega (r_h (\bar{q}) - \bar{q})\right),\quad
\underline{q}_{h,\tau} = \rho_\tau \left(r_h (\underline{q}) 
- \frac{1}{|\Omega|}\int_\Omega (r_h (\underline{q}) - \underline{q})\right).
\end{equation}
Since $\bar{q}_{h,\tau}$ and $\underline{q}_{h,\tau}$ are piecewise linears in space, they are exactly integrated by the trapezoidal rule and we easily derive from \eqref{eq:flowrate} and \eqref{eq:qhdef} that we  have for all $n$,
\begin{equation}
\big(\bar{q}_h^n, 1\big)_h = \big(\underline{q}_h^n, 1\big)_h.
\label{eq:qhprop}
\end{equation}
The set of primary unknowns is the discrete wetting phase saturation and
the discrete wetting phase pressure, $S_h^{n}$ and $P_{w,h}^{n}$, defined pointwise at time $t_n$ by:
\[
S_h^n = \sum_{i=1}^M S^{n,i} \phi_i, \quad P_{w,h}^n = \sum_{i=1}^M P_{w}^{n,i} \phi_i, \quad 1\leq n\leq N.
\]
Then the discrete non-wetting phase pressure $P_{o,h}^n$ defined by
$$P_{o,h}^n = \sum_{i=1}^M P_o^{n,i} \phi_i, \quad 1\leq n\leq N,$$
is a secondary unknown. 
The upwind scheme we propose for discretizing \eqref{eq:2ph1}--\eqref{eq:2ph2} is inspired by the finite volume scheme introduced and analyzed by Eymard {\it al} in~\cite{Eymard2003}. For each time step $n$, $1 \le n \le N$, the lines of the discrete equations are\\
\framebox{\parbox[c]{\textwidth}{ 
\begin{equation}
\frac{\tilde{m}_i(\varphi)}{\tau} (S^{n,i} - S^{n-1,i})
-\sum_{j=1}^M c_{ij} \eta_w(S^{n,ij}_w)  (P_w^{n,j}-P_w^{n,i})
= m_i \left( f_w(s_{\mathrm{in}}^{n,i}) \bar{q}^{n,i} - f_w(S^{n,i}) \underline{q}^{n,i}\right), 
\label{eq:scheme1}
\end{equation}
\begin{equation}
-\frac{\tilde{m}_i( \varphi)}{\tau} (S^{n,i} - S^{n-1,i})
-\sum_{j=1}^M c_{ij} \eta_o(S^{n,ij}_o) (P_o^{n,j}-P_o^{n,i})
= m_i \left( f_o(s_{\mathrm{in}}^{n,i}) \bar{q}^{n,i} - f_o(S^{n,i}) \underline{q}^{n,i}\right), \label{eq:scheme2}
\end{equation}
\begin{equation}
P_o^{n,i}-P_w^{n,i} = p_c(S^{n,i}), \quad 1 \leq i\leq M, 
\label{eq:scheme3}
\end{equation}
\begin{equation}
\sum_{i=1}^M  m_i P_w^{ n  ,i} = 0.
\label{eq:scheme4}
\end{equation}
}}
Here $i$ runs from $1$ to $M-1$ in \eqref{eq:scheme1} and from $1$ to $M$ in \eqref{eq:scheme2};  the upwind values $S^{n,ij}_w, S^{n,ij}_o$  are defined by 
\begin{equation}
\label{eq:Swij}
S_w^{n,ij}  = \left\{
\begin{array}{c}
S^{n,i} \quad\mbox{if}\quad P_w^{n,i}>P_w^{n,j}\\
S^{n,j} \quad\mbox{if}\quad P_w^{n,i}<P_w^{n,j}\\
\max(S^{n,i},S^{n,j}) \quad\mbox{if}\quad P_w^{n,i} = P_w^{n,j}
\end{array}
\right.
\end{equation}
\begin{equation}
\label{eq:Soij}
S_o^{n,ij} = \left\{
\begin{array}{c}
S^{n,i} \quad\mbox{if}\quad P_o^{n,i}>P_o^{n,j}\\
S^{n,j} \quad\mbox{if}\quad P_o^{n,i}<P_o^{n,j}\\
\min(S^{n,i},S^{n,j}) \quad\mbox{if}\quad P_o^{n,i} = P_o^{n,j}
\end{array}
\right.
\end{equation}
We observe  that
\[
S_w^{n,ij} = S_w^{n,ji}, \quad S_o^{n,ij} = S_o^{n,ji},
\]
 so that, if we interpret in \eqref{eq:scheme1} (respectively, \eqref{eq:scheme2}) $\eta_w(S^{n,ij}_w)$ (respectively, $\eta_o(S^{n,ij}_o)$) as $\tilde W^{i,j}$, then \eqref{eq:Wsymij} and hence \eqref{eq:sum0} hold.

\begin{remark}
\label{rem:i=M}
{\rm Before setting \eqref{eq:scheme1}--\eqref{eq:scheme4} in variational form, observe that:}

{\rm {\bf 1}.\ The scheme \eqref{eq:scheme1}-\eqref{eq:scheme4} forms a square system in the primary unknowns, $S_h^n$ and $P_w^n$.

{\bf 2}.\ Formula \eqref{eq:scheme1} is also valid for $i=M$. Indeed, we pass to the left-hand side the right-hand side of \eqref{eq:scheme1} and set $A^i$ the resulting line of index $i$. Let $\tilde  A^M$ denote what should be the line of index $M$, i.e.,
\begin{align*}
\tilde  A^M = \frac{\tilde{m}_M(\varphi)}{\tau} (S^{n,M} - S^{n-1,M}) 
-&\sum_{j =1}^M c_{Mj} \eta_w(S^{n,Mj}_w)  (P_w^{n,j}-P_w^{n,M})\\
&
- m_M \big( f_w(s_{\mathrm{in}}^{n,M}) \bar{q}^{n,M} - f_w(S^{n,M}) \underline{q}^{n,M}\big).
\end{align*}
Then, in view of  \eqref{eq:sum0},
$$\tilde  A^M = \sum_{i=1}^{M-1} A^i + \tilde  A^M = \sum_{i=1}^{M} \frac{\tilde{m}_i ( \varphi)}{\tau} (S^{n,i} - S^{n-1,i})
-\sum_{i=1}^{M} m_i \big( f_w(s_{\mathrm{in}}^{n,i}) \bar{q}^{n,i} - f_w(S^{n,i}) \underline{q}^{n,i}\big).
$$
By summing in the same fashion the lines of \eqref{eq:scheme2}, we obtain
$$
\sum_{i=1}^{M}\frac{\tilde{m}_i( \varphi)}{\tau} (S^{n,i} - S^{n-1,i}) =- \sum_{i=1}^{M}
 m_i \big( f_o(s_{\mathrm{in}}^{n,i}) \bar{q}^{n,i} - f_o(S^{n,i}) \underline{q}^{n,i}\big).
 $$
A combination of these two equations yields
$$
\tilde  A^M = -\sum_{i=1}^{M}m_i \Big( (f_w(s_{\mathrm{in}}^{n,i})+ f_o(s_{\mathrm{in}}^{n,i})) \bar{q}^{n,i} - (f_w(S^{n,i}) +
f_o(S^{n,i})) \underline{q}^{n,i}\Big)
 = -\sum_{i=1}^{M}m_i (\bar{q}^{n,i}-\underline{q}^{n,i})
= 0,
$$
by virtue of \eqref{eq:fluidfract}, the definition \eqref{eq:rhotau}, and \eqref{eq:flowrate}.

{\bf 3}.\ In \eqref{eq:scheme1} (respectively,  \eqref{eq:scheme2}), any constant can be added to $P_w$ (respectively,  $P_o$), but in view of \eqref{eq:scheme3}, the constant must be the same for both pressures. The last equation \eqref{eq:scheme4} is added to resolve this constant.
}\qed
\end{remark}

As usual, it is convenient to associate time functions $S_{h,\tau}$, $P_{\alpha,h,\tau}$ with the sequences indexed by $n$. These are piecewise constant in time in $]0,T[$, for instance
\begin{equation}
\label{eq:Pconst}
P_{\alpha,h,\tau}(t,x) = P_{\alpha,h}^{n}(x),\ \alpha = w,o, \quad  \forall (t,x)\in  \Omega \times  ]t_{n-1},t_{n}].
\end{equation}

In view of the material of the previous subsection, we introduce the following form:
\begin{equation}
\label{eq:quadupwind}
\forall W_h, U_h, V_h, Z_h \in X_h,\ \ [Z_h, W_h;\Rd V_h,U_h\Bk]_h = \sum_{i,j= 1}^M U^i  c_{ij} \tilde W^{ij} (V^j - V^i),
\end{equation}
where the first argument $Z_h$ indicates that the choice of $\tilde W^{ij}$ depends on $Z_h$. Such dependence, used for the upwinding, will be specified further on, but {\em it is assumed from now on} that $\tilde W^{ij}$ satisfies \eqref{eq:Wsymij}. Considering \eqref{eq:sum0}, the form satisfies the following properties,
\begin{equation}
\label{eq:quadup1}
\forall Z_h, W_h, V_h \in X_h,\quad [Z_h, W_h;\Rd V_h,1\Bk]_h =0,
\end{equation}
\begin{equation}
\label{eq:quadup2}
\forall Z_h, W_h, V_h \in X_h,\quad [Z_h, W_h;V_h,V_h]_h = - \frac{1}{2} \sum_{i,j=1}^M c_{ij}\tilde W_{ij} 
(V^i-V^j)^2.
\end{equation}
This last property is derived by the same argument as in proving \eqref{eq:nablauv}.
%$$V^i(V^j-V^i) = - \frac{1}{2}\big[(V^i)^2 - (V^j)^2 + \big(V^i-V^j\big)^2\big],
%$$
%and from the fact that
%\begin{equation}
%\label{eq:quadup3}
%\sum_{i,j=1}^M c_{ij}\tilde W_{ij} (V^i + V^j)(V^i - V^j) =0,
%\end{equation}
%since $V^i + V^j$ is also symmetric.

With the above notation, and taking into account that \eqref{eq:scheme1} extends to $i=M$, the scheme \eqref{eq:scheme1}--\eqref{eq:scheme4} has the equivalent variational form. Starting from $S_h^0$, see \eqref{eq:Sh0},\\
\framebox{\parbox[c]{\textwidth}{find $S_h^{n}$, $P_{wh}^{n}$, and  $P_{o,h}^{n}$ in $X_h$, for $1 \le n \le N$, solution of, for all $\theta_h$ in $X_h$,
\begin{equation}
 \frac{1}{\tau} (S_h^{n} - S_h^{n-1},\theta_h)_h^\varphi -\big[P_{w,h}^{n},I_h(\eta_w(S_h^{n})); P_{w,h}^{n},\theta_h\big]_h 
 = \big(I_h(f_w(s_{\mathrm{in},h}^{n}))\bar{q}_h^{n} - I_h(f_w(S_h^{n}))\underline{q}_h^{n},\theta_h\big)_h\label{eq:vars1}
\end{equation}
\begin{equation}
-\frac{1}{\tau} (S_h^{n} - S_h^{n-1},\theta_h)_h^\varphi
-\big[P_{o,h}^{n},I_h(\eta_o(S_h^{n})); P_{o,h}^{n},\theta_h\big]_h 
= \big(I_h(f_o(s_{\mathrm{in},h}^{n}))\bar{q}_h^{n} - I_h(f_o(S_h^{n}))\underline{q}_h^{n},\theta_h\big)_h\label{eq:vars2}
\end{equation}
\begin{equation}
P_{o,h}^{n}-P_{w,h}^{n} = I_h(p_c(S_h^{n})),  \label{eq:vars3}
\end{equation}
\begin{equation}
\big(P_{w,h}^{n},1\big)_h = 0,
\label{eq:vars4}
\end{equation}
}}
where the choice of $\eta_w(S_h^{n})$ in the left-hand side of \eqref{eq:vars1} (respectively, $\eta_o(S_h^{n})$ in the left-hand side of \eqref{eq:vars2}) is given by \eqref{eq:Swij} (respectively \eqref{eq:Soij}). Strictly speaking, the interpolation operator $I_h$ is introduced in \eqref{eq:vars1} and
\eqref{eq:vars2} because the forms are defined for functions of $X_h$, but for the sake of simplicity, since only nodal values are used, it may be dropped further on.

We shall see that under the above basic hypotheses, the discrete problem \eqref{eq:vars1}--\eqref{eq:vars4} has at least one solution. In the sequel, we shall use the following discrete auxiliary pressures (compare with \eqref{eq:Var}):
\begin{equation}
\label{eq:discraux}
U_{w, h, \tau} = P_{w,h,\tau} + I_h(p_{wg}(S_{h,\tau})),\quad U_{o, h, \tau} = P_{o,h,\tau} - I_h(p_{og}(S_{h,\tau})).
\end{equation}
The following theorem is the main result of this work:

\begin{theorem}
\label{thm:mainresult}
Under the above basic hypotheses and the additional assumptions \eqref{eq:eta'w}--\eqref{eq:pc'}, the discrete solutions converge up to subsequences as follows:
\begin{align*}
\lim_{(h,\tau) \to (0,0)} S_{h,\tau} &=  s  \quad \mbox{ strongly in}\ L^2(\Omega \times ]0,T[),\\
\lim_{(h,\tau) \to (0,0)} U_{w, h, \tau} = p_w + p_{wg}(s), \quad \lim_{(h,\tau) \to (0,0)} U_{o, h, \tau} &= p_o - p_{og}(s), \quad \mbox{ weakly in}\ L^2(0,T;H^1(\Omega) ),\\
\lim_{(h,\tau) \to (0,0)} P_{\alpha,h,\tau} &= p_\alpha, \quad \mbox{ weakly in}\ L^2(\Omega \times ]0,T[), \alpha = w,o,
\end{align*}
where $p_w + p_{wg}$, $p_o - p_{wg}$, and $s$ solve the weak formulation \eqref{eq:Var}.
\end{theorem} 

The proof of the theorem requires several steps that are covered in the remaining of this work.  
%==================================

\section{First  a priori bounds}
\label{sec:1staprioribdds}

 This section is devoted to basic a priori bounds used in proving existence of a discrete solution. Existence is fairly technical and will be postponed till Section \ref{sec:existence}. The first step is a key bound on the discrete saturation. In a second step, this bound will lead to a pressure estimate and in particular to a bound on the discrete analogue of auxiliary pressures.

%-----------------------

\subsection{Maximum principle}
\label{subsec:maxprinciple}

The scheme \eqref{eq:scheme1}--\eqref{eq:scheme4} satisfies the maximum principle property. 
The proof given below uses a standard argument as in \cite{Eymard2003}.
%Here the finite volume formulation is more convenient for the proof. 

\begin{theorem}\label{thm:maxprinc}
The following bounds hold:
\begin{equation}\label{eq:maxprinc}
0\leq S_{h,\tau} \leq 1. 
\end{equation}
\end{theorem}
\begin{proof}
As $0 \le s_w^0 \le 1$ almost everywhere, by construction \eqref{eq:Sh0}, we immediately have
\[
0\leq \min_\Omega  s_w^0  \leq S_h^0 \leq \max_\Omega  s_w^0  \leq 1.
\]
The proof proceeds by contradiction. Assume that there is an index $n\geq 1$ such that
\[
S_h^{n-1}\leq 1, \quad S_h^{n} > 1.
\]
This means there is a node $i$ such that
\[
S^{n,i}  = \Vert S_h^{n}\Vert_{L^\infty(\Omega)} > 1,
\]
and thus
\[
S^{n,i} > S^{n-1,i}.
\]
Dropping the index $n$ in the rest of the proof,  \eqref{eq:scheme1} and \eqref{eq:scheme2} imply
\begin{eqnarray}
\sum_{j\neq i, j\in {\mathcal N}(i)} c_{ij} \eta_w(S^{ij}_w)  (P_w^{j}-P_w^{i})
+ m_i \left( f_w(s_{\mathrm{in}}^{i}) \bar{q}^{i} - f_w(S^{i}) \underline{q}^{i}\right) > 0,\label{eq:set1}
\\
-\sum_{j\neq i, j\in {\mathcal N}(i)} c_{ij} \eta_o(S^{ij}_o) (P_o^{j}-P_o^{i})
- m_i \left( f_o(s_{\mathrm{in}}^{i}) \bar{q}^{i} - f_o(S^{i}) \underline{q}^{i}\right)>0 .\label{eq:set2}
\end{eqnarray}
We first show that \eqref{eq:set1} holds true with $S^{ij}_w$ replaced by $S^{i}$.
Indeed if $P_w^i > P_w^j$, then $S_w^{ij} =S^i$. If $P_w^i < P_w^j$, then $S_w^{ij} =S^j$, and as $\eta_w$ is increasing and by assumption, $S^j\leq S^i$,
\[
\eta_w(S_w^{ij}) (P_w^j-P_w^i) \leq \eta_w(S^i) (P_w^j-P_w^i).
\]
Finally, the term vanishes when $P_w^i=P_w^j$. Therefore we have in all cases

\begin{equation}
\sum_{j\neq i, j\in {\mathcal N}(i)} c_{ij} \eta_w(S^i)  (P_w^{j}-P_w^{i})
+ m_i \left( f_w(s_\mathrm{in}^{i}) \bar{q}^{i} - f_w(S^{i}) \underline{q}^{i}\right) > 0.\label{eq:set1noup}
\end{equation}
%
%
%Similarly we show that \eqref{eq:set2} holds true with $S^{ij}_o$ replaced by $S^i$. 
%Indeed if $P_o^i>P_o^j$, the result is true by definition.  Let us assume now that $P_o^i<P_o^j$. This implies
%that $S_o^{ij} = S^j$.  Since $\eta_o$ is decreasing and $S^j\leq S^i$ we have
%\[
%-\eta_o(S_o^{ij}) (P_o^j-P_o^i) \leq -\eta_o(S^i) (P_o^j-P_o^i).
%\]
%The term vanishes if $P_o^j=P_o^i$. 
A similar argument gives
\begin{equation}
-\sum_{j\neq i, j\in {\mathcal N}(i)} c_{ij} \eta_o(S^i) (P_o^{j}-P_o^{i})
- m_i \left( f_o(s_\mathrm{in}^{i}) \bar{q}^{i} - f_o(S^{i}) \underline{q}^{i}\right)>0.\label{eq:set2noup}
\end{equation}
The substitution of \eqref{eq:scheme3} into \eqref{eq:set2noup} yields
\begin{equation}
-\sum_{j\neq i, j\in {\mathcal N}(i)} c_{ij} \eta_o(S^i) \big((P_w^{j}-P_w^{i}) + (p_c(S^j)-p_c(S^i)\big))
- m_i \left( f_o(s_\mathrm{in}^{i}) \bar{q}^{i} - f_o(S^{i}) \underline{q}^{i}\right)>0.
\end{equation}
Since $p_c$ is decreasing and $S^i \geq S^j$, the second term in the above sum is negative.  This implies that
\begin{equation}\label{eq:set2aa}
-\sum_{j\neq i, j\in {\mathcal N}(i)} c_{ij} \eta_o(S^i) (P_w^{j}-P_w^{i})
- m_i \left( f_o(s_\mathrm{in}^{i}) \bar{q}^{i} - f_o(S^{i}) \underline{q}^{i}\right)>0.
\end{equation}
The sum on $j$ cancels by multiplying \eqref{eq:set1noup} by $\eta_o(S^i)$, \eqref{eq:set2aa} by $\eta_w(S^i)$, and adding the two. The sign is unchanged because either $\eta_o(S^i)$ or $\eta_w(S^i)$ is strictly positive. Hence,
\[
m_i \eta_o(S^i) \left( f_w(s_\mathrm{in}^{i}) \bar{q}^{i} - f_w(S^{i}) \underline{q}^{i}\right) 
- m_i \eta_w(S^i) \left( f_o(s_\mathrm{in}^{i}) \bar{q}^{i} - f_o(S^{i}) \underline{q}^{i}\right)
> 0.
\]
By definition of $f_w$ and $f_o$,  this reduces to
\begin{equation}\label{eq:boundr}
\eta_o(S^i)  f_w(s_\mathrm{in}^{i}) - \eta_w(S^i) f_o(s_\mathrm{in}^{i}) > 0.
\end{equation}
Now consider the function:
\begin{equation}\label{eq:defr}
r(s) = \eta_o(s) f_w(s_\mathrm{in}^{i}) - \eta_w(s) f_o(s_\mathrm{in}^{i}).
\end{equation}
It is decreasing and $r(s_\mathrm{in}^i) = 0$. Then, since $S^i > 1 \ge s_\mathrm{in}^i$, see \eqref{eq:inputsat}, we have
\[
r(S^i) \leq r(s_\mathrm{in}^i) = 0,
\]
which contradicts \eqref{eq:boundr}. The proof of the lower bound in \eqref{eq:maxprinc} follows the same lines. \end{proof}

%---------------------

\subsection{First pressure bounds}
\label{subsec:1stpressurebd}

The following properties will be used frequently:

\begin{lemma}
\label{lem:PandS}
The fact that $p_c$ is strictly decreasing and \eqref{eq:scheme3} yield the following:
\begin{equation}
P_w^i > P_w^j, \quad \mbox{and}\quad P_o^i\leq P_o^j \quad\mbox{implies}\quad S^i \geq  S^j,
\label{eq:remarkSiSj}
\end{equation}
\begin{equation}
\label{eq:remarkSiSj2}
\mbox{If}\ P_w^i = P_w^j,\ \mbox{then}\ P_o^i \ge P_o^j \ \mbox{if and only if}\  S^i \le S^j.
\end{equation}
\begin{equation}
\label{eq:remarkSiSj3}
\mbox{If}\ P_o^i = P_o^j, \ \mbox{then}\ P_w^i \leq P_w^j,\  \mbox{if and only if}\  \quad S^i \le S^j.
\end{equation}
\end{lemma}

Let us start with a lower bound that removes the degeneracy caused by the mobilities when they multiply the discrete pressures. 

\begin{lemma}
\label{lem:boundetap}
Let $p_{wg}$ and $p_{og}$ be defined in \eqref{eq:defpwgpog}.
We have for all $n$ and any $i$ and $j$
\begin{equation}
\label{eq:boundetap1}
\eta_\ast ( U_w^{n,j} - U_w^{n,i}  )^2 \leq \eta_w(S_w^{n,ij}) (P_w^{n,j}-P_w^{n,i})^2 
+ \eta_o(S_o^{n,ij}) (P_o^{n,j}-P_o^{n,i})^2.
\end{equation}
\end{lemma}

\begin{proof}
 To simplify the notation, we drop the superscript $n$.
The second mean formula for integrals gives
\begin{equation}\label{eq:meanint}
p_{wg}(S^j)-p_{wg}(S^i)  = \int_{S^i}^{S^j}  f_o(s) p_c'(s) ds  = f_o(\xi) (p_c(S^j)-p_c(S^i)),
\end{equation}
for some $\xi$ between $S^i$ and $S^j$.
Using \eqref{eq:scheme3} we write
\[
U_w^{j} - U_w^{i} 
= (1-f_o(\xi)) (P_w^j-P_w^i)  + f_o(\xi) (P_o^j-P_o^i)
= f_w(\xi) (P_w^j-P_w^i)  + f_o(\xi) (P_o^j-P_o^i).
\]
Therefore since $f_w+f_o=1$, we have
\begin{equation}
\label{eq:meanint2}
( U_w^{j} - U_w^{i} )^2  \leq 
 \frac{\eta_w(\xi)}{\eta_w(\xi)+\eta_o(\xi)} (P_w^j-P_w^i)^2
+ \frac{\eta_o(\xi)}{\eta_w(\xi)+\eta_o(\xi)} (P_o^j-P_o^i)^2.
\end{equation}
We now consider six cases.

1)\ If \underline{$P_w^i > P_w^j$ and $P_o^i \leq P_o^j$}, then $\eta_w(S_w^{ij}) = \eta_w(S^i)$ and $\eta_o(S_o^{ij}) = \eta_o(S^j)$ when $P_o^i < P_o^j$; when $P_o^i = P_o^j$, the value of $\eta_o$ does not matter.
From \eqref{eq:remarkSiSj} we then have $S^i \geq S^j$. Since $\eta_w$ is increasing, $\eta_w(\xi) \leq \eta_w(S^i)$ and since
$\eta_o$ is decreasing, $\eta_o(\xi) \leq \eta_o(S^j)$.
Thus we have
\[
( U_w^{j} - U_w^{i} )^2  \leq 
 \frac{\eta_w(S_w^{ij})}{\eta_w(\xi)+\eta_o(\xi)} (P_w^j-P_w^i)^2
+ \frac{\eta_o(S_o^{ij})}{\eta_w(\xi)+\eta_o(\xi)} (P_o^j-P_o^i)^2,
\]
and with \eqref{eq:lowerboundetas}
\begin{equation}
\label{eq:etaintbound}
( U_w^{j} - U_w^{i} )^2  \leq 
 \frac{1}{\eta_\ast} \left(\eta_w(S_w^{ij}) (P_w^j-P_w^i)^2
+ \eta_o(S_o^{ij}) (P_o^j-P_o^i)^2\right).
\end{equation}

2)\ If \underline{$P_w^i>P_w^j$ and $P_o^i>P_o^j$ }, then $\eta_w(S_w^{ij}) = \eta_w(S^i)$ and $\eta_o(S_o^{ij}) = \eta_o(S^i)$. From 
 \[
\eta_o(S^i) (p_c(S^j)-p_c(S^i))  = (\eta_o(S^i)+\eta_w(S^i)) \int_{S^i}^{S^j} f_o(S^i) p_c'(s) ds,
\]
and \eqref{eq:meanint}, we derive
\begin{align*}
\eta_o(S^i) (p_c(S^j)-p_c(S^i)) - &(\eta_o(S^i)+\eta_w(S^i)) (p_{wg}(S^j)-p_{wg}(S^i))
\\
&= (\eta_o(S^i)+\eta_w(S^i)) \int_{S^i}^{S^j} (f_o(S^i)-f_o(s)) p_c'(s) ds.
\end{align*}
As $p_c$ and $f_o$ are decreasing, the above right-hand side is negative. Hence
\begin{equation}
\eta_o(S^i) (p_c(S^j)-p_c(S^i)) - (\eta_o(S^i)+\eta_w(S^i)) (p_{wg}(S^j)-p_{wg}(S^i)) \leq 0.
\label{eq:intetabound}
\end{equation}
We multiply \eqref{eq:intetabound} by $(P_o^j-P_o^i) + (P_w^j-P_w^i) <0$ and use \eqref{eq:scheme3}, 
\[%begin{eqnarray*}
\big(\eta_o(S^i) (p_c(S^j)-p_c(S^i)) - (\eta_o(S^i)+\eta_w(S^i)) (p_{wg}(S^j)-p_{wg}(S^i)) \big)
\left( 2 (P_w^j-P_w^i) + p_c(S^j)-p_c(S^i)\right) \geq 0.
\]
By expanding and using the next inequality implied by \eqref{eq:meanint}, if $f_o(\xi)\neq 0$, 
\[
(p_{wg}(S^j)-p_{wg}(S^i)) (p_c(S^j) - p_c(S^i)) \geq (p_{wg}(S^j)-p_{wg}(S^i))^2,
\]
we obtain
\begin{eqnarray*}
\eta_o(S^i) (p_c(S^j)-p_c(S^i))^2 + 2 \eta_o(S^i) (p_c(S^j)-p_c(S^i))  (P_w^j-P_w^i) 
\geq 
\\
(\eta_o(S^i)+\eta_w(S^i)) (p_{wg}(S^j)-p_{wg}(S^i)) \left( 2 (P_w^j-P_w^i) + p_{wg}(S^j)-p_{wg}(S^i)\right).
\end{eqnarray*}
When $(\eta_o(S^i)+\eta_w(S^i)) (P_w^j-P_w^i)^2$ is added to both sides, this becomes 
\[
\eta_w(S^i) (P_w^j-P_w^i)^2 + \eta_o(S^i) (P_o^j-P_o^i)^2 
\geq (\eta_o(S^i)+\eta_w(S^i)) (U_w^j-U_w^i)^2,
\]
and \eqref{eq:lowerboundetas} implies the desired result. It remains to
consider the case $f_o(\xi) = 0$, i.e., $p_{wg}(S^j) = p_{wg}(S^i)$. If $\eta_o(S^i) \neq 0$, then  \eqref{eq:intetabound} yields
\[
p_c(S^j) - p_c(S^i) \leq 0\ \mbox{which implies}\ P_o^i-P_o^j \geq P_w^i-P_w^j,
\]
and we deduce immediately
\[
\eta_w(S^i) (P_w^j-P_w^i)^2 + \eta_o(S^i) (P_o^j-P_o^i)^2 
\geq (\eta_w(S^i)+\eta_o(S^i)) (P_w^j-P_w^i)^2
\geq \eta_\ast  (P_w^j-P_w^i)^2.
\]
When $\eta_o(S^i) = 0$, we have trivially
\[
\eta_w(S^i) (P_w^j-P_w^i)^2 + \eta_o(S^i) (P_o^j-P_o^i)^2 
= \eta_w(S^i) (P_w^j-P_w^i)^2
\geq \eta_\ast (P_w^j-P_w^i)^2.
\]

3)\ If \underline{$P_w^i\leq P_w^j$ and $P_o^i > P_o^j$ }, then  $\eta_w(S_w^{ij}) = \eta_w(S^j)$ and $\eta_o(S_o^{ij}) = \eta_o(S^i)$ in the case of a strict inequality; also
$S^i \leq  S^j$. Then \eqref{eq:meanint2} and the monotonic properties of $\eta_w$ and $\eta_o$ yield \eqref{eq:boundetap1}. If $P_w^i= P_w^j$, then according to \eqref{eq:remarkSiSj2}, $S^i \leq  S^j$ and the same conclusion holds.

4)\ If \underline{$P_w^i\leq P_w^j$ and $P_o^i=P_o^j$}, then  from  
\eqref{eq:remarkSiSj3}, we have $S^i\leq S^j$ and with \eqref{eq:meanint2}
\[
( U_w^{j} - U_w^{i} )^2  \leq 
 \frac{\eta_w(\xi)}{\eta_w(\xi)+\eta_o(\xi)} (P_w^j-P_w^i)^2
 \leq \frac{\eta_w(S_w^{ij})}{\eta_w(\xi)+\eta_o(\xi)} (P_w^j-P_w^i)^2,
\]
which is the desired result.

5)\ Similarly, if \underline{$P_w^i=P_w^j$ and $P_o^i<P_o^j$}, then from \eqref{eq:remarkSiSj2}, we have $S^j \leq S^i$ and with \eqref{eq:meanint2}
 \[
( U_w^{j} - U_w^{i} )^2  \leq 
 \frac{\eta_o(\xi)}{\eta_w(\xi)+\eta_o(\xi)} (P_o^j-P_o^i)^2
 \leq \frac{\eta_o(S_o^{ij})}{\eta_w(\xi)+\eta_o(\xi)} (P_o^j-P_o^i)^2.
\]

6)\ If \underline{$P_w^i < P_w^j$ and $P_o^i < P_o^j$}, \eqref{eq:boundetap1} follows from the second case by switching $i$ and $j$.
\end{proof}

\smallskip

The pressure bound in the next theorem is the one that arises naturally from the left-hand side of \eqref{eq:vars1} and \eqref{eq:vars2}. 

\begin{theorem}
\label{thm:pressbound1}
There exists a constant $C$, independent of $h$ and $\tau$, such that
\begin{equation}
\tau \sum_{n=1}^N \sum_{i,j=1}^M c_{ij} \big(\eta_w(S_w^{n,ij}) (P_w^{n,i} - P_w^{n,j})^2
+ \eta_o(S_o^{n,ij}) (P_o^{n,i}-P_o^{n,j})^2\big) \leq C.
\label{eq:pressbound1}
\end{equation}
\end{theorem}

\begin{proof} 
We  test \eqref{eq:vars1} by $P_{w,h}^{n}$, \eqref{eq:vars2} by $P_{o,h}^{n}$, add the two equations, multiply by $\tau$ and sum over $n$ from $1$ to $N$.  By using \eqref{eq:vars3} and \eqref{eq:quadup2}, we obtain
\begin{equation}
\label{eq:press1}
\begin{split}
-\sum_{n=1}^N\big(S_h^{n}-S_h^{n-1}, I_h p_c(S_h^{n})\big)_h^{  \varphi }
+ \frac12 \sum_{n=1}^N \tau \sum_{\alpha =w,o}\sum_{i,j =1}^M c_{ij} \eta_\alpha(S_\alpha^{n,ij}) (P_\alpha^{n,i}-P_\alpha^{n,j})^2\\
= \sum_{n=1}^N \tau\sum_{\alpha =w,o}\big(f_\alpha(s_{\mathrm{in},h}^{n}) \bar{q}_h^{n} - f_\alpha(S_h^{n}) \underline{q}_h^{n},
P_{\alpha,h}^{n}\big)_h.
\end{split}
\end{equation}
Following~\cite{Eymard2003}, the first term in \eqref{eq:press1} is treated with
the primitive $g_c$ of $p_c$, see \eqref{eq:gc}. Indeed, by the mean-value theorem, 
 there exists  $\xi$ between $S^{n,i}$ and $S^{n-1,i}$ such that
\[
g_c(S^{n,i})-g_c(S^{n-1,i}) = - (S^{n,i}-S^{n-1,i}) p_c(\xi).
\]
 As  the function $p_c$ is decreasing,  then $p_c(\xi) \geq p_c(S^{n,i})$  when $S^{n,i} \geq S^{n-1,i}$  and  $p_c(\xi) \leq p_c(S^{n,i})$  when  $S^{n,i} \leq S^{n-1,i}$. 
 In  both cases, we have
\[
g_c(S^{n,i})-g_c(S^{n-1,i}) \leq  - (S^{n,i}-S^{n-1,i}) p_c(S^{n,i}),
\]
 and owing that $\varphi$ is positive and constant in time, \eqref{eq:press1} can be replaced by the inequality
\begin{equation}
\label{eq:boundint}
\begin{split}
\big(g_c(S_h^{N})-g_c(S_h^{0}),1\big)_h^\varphi
&+ \frac{1}{2} \sum_{n=1}^N \tau \sum_{\alpha =w,o}\sum_{i,j =1}^M c_{ij} \eta_\alpha(S_\alpha^{n,ij}) (P_\alpha^{n,i}-P_\alpha^{n,j})^2\\
&\le \sum_{n=1}^N \tau\sum_{\alpha =w,o}\big(f_\alpha(s_{\mathrm{in},h}^{n}) \bar{q}_h^{n} - f_\alpha(S_h^{n}) \underline{q}_h^{n},
P_{\alpha,h}^{n}\big)_h.
\end{split}
\end{equation}
As the first term in the above left-hand side is bounded, owing to the continuity of $g_c$ and boundedness of $S_{h,\tau}$, it suffices to handle the right-hand side. Let us drop the superscript $n$ and treat one term in the time sum. Following again~\cite{Eymard2003}, in view of Lemma \ref{lem:boundetap} we use the auxiliary pressures $p_{wg}$ and $p_{wo}$, defined in \eqref{eq:defpwgpog}.  Clearly, \eqref{eq:pwg+pog} and \eqref{eq:scheme3} imply 
\begin{equation}\label{eq:pwgpwgo}
P_w^{i} + p_{wg}(S^{i}) + p_{og}(S^{i}) + p_c(0) = P_o^{i},\quad \forall i .
\end{equation}
Using this, a generic term, say $Y$, in the right-hand side of \eqref{eq:boundint} can be expressed as
\begin{align*}
Y =& \big(\bar{q}_h - \underline{q}_h, U_{w,h} \big)_h + \big(f_o(s_{\mathrm{in},h}) \bar{q}_h - f_o(S_h) \underline{q}_h, p_c(0)\big)_h\\
 & + \big(f_o(s_{\mathrm{in},h}) \bar{q}_h - f_o(S_h) \underline{q}_h, p_{og}(S_h)\big)_h - 
 \big(f_w(s_{\mathrm{in},h}) \bar{q}_h - f_w(S_h) \underline{q}_h, p_{wg}(S_h)\big)_h =
 T_1+\cdots + T_4.
\end{align*}
We now bound each term $T_i$. For $T_1$,  \eqref{eq:qhprop} implies that any constant $\beta$ can be added to $U_{w,h}$, in particular $\beta$ can be chosen so that the sum has zero mean value in $\Omega$. Hence, considering the generalized Poincar\'e inequality
\begin{equation}
\label{eq:genPoinc}
\forall v \in H^1(\Omega),\quad \|v\|_{L^2(\Omega)} \le C \big(\big|\int_\Omega v \big| + \|\nabla\,v\|_{L^2(\Omega)}\big),
\end{equation}
with a constant $C$, depending only on the domain $\Omega$, we have
$$\|U_{w,h} + \beta\|_h \le C \| U_{w,h}+ \beta\|_{L^2(\Omega)} \le C \|\nabla\,U_{w,h}\|_{L^2(\Omega)}, 
$$
with another constant $C$. Then Young's inequality yields
$$|T_1| \leq \frac{C^2}{2\eta_\ast} \|\bar{q}_h-\underline{q}_h\|^2_h
+ \frac{\eta_\ast}{4} \| \nabla\,U_{w,h}\|_{L^2(\Omega)}^2, 
$$
and with  Lemma~\ref{lem:boundetap}, this becomes
$$|T_1| \leq \frac{C^2}{2\eta_\ast} \|\bar{q}_h-\underline{q}_h\|^2_h +  \frac{1}{4} \sum_{i,j=1}^M  c_{ij} \left(\eta_w(S^{ij}) (P_w^{j}-P_w^{i})^2
+\eta_o(S^{ij}) (P_o^{j}-P_o^{i})^2\right).
$$
The term $T_2$ is easily bounded since $p_c(0)$ is a number, and so are the terms $T_3$ and $T_4$, in view of the boundedness of the saturation and the continuity of $p_{og}$ and 
$p_{wg}$. We thus have
$$
|T_2 +T_3 + T_4| \leq C (\Vert \bar{q}_h\Vert_{L^1(\Omega)} + \Vert \underline{q}_h\Vert_{L^1(\Omega)}).
$$
Then substituting these bounds for each $n$ into \eqref{eq:boundint}, we obtain
\begin{align*}
\frac{1}{4}\tau \sum_{n=1}^N \sum_{i,j=1}^M c_{ij} \big(\eta_w(S_w^{n,ij}) (P_w^{n,i} - P_w^{n,j})^2
&+ \eta_o(S_o^{n,ij}) (P_o^{n,i}-P_o^{n,j})^2\big) \le C \big(\|\bar{q}_{h,\tau}-\underline{q}_{h,\tau}\|^2_{L^2(\Om \times ]0,T[)}\\
&+ \|\bar{q}_{h,\tau}\|_{L^1(\Om \times ]0,T[)} + \|\underline{q}_{h,\tau}\|_{L^1(\Om \times ]0,T[)}\big),
\end{align*}
thus proving \eqref{eq:pressbound1}.
\end{proof}

% The next proposition establishes some useful properties. 
%
%\begin{proposition}
%\label{pro:sums}
%\begin{equation}
%\sum_{i} \sum_{j\neq i, j\in\mathcal{N}(i)} c_{ij} \eta_w(S_\alpha^{n+1,ij}) ( (P_\alpha^{n+1,j})^2 - (P_\alpha^{n+1,i})^2) = 0, 
%\quad \alpha = w, o
%\label{eq:symbound}
%\end{equation}
%\begin{equation}
%\sum_{i} \sum_{j\neq i, j\in\mathcal{N}(i)} c_{ij} \eta_w(S_\alpha^{n+1,ij}) ( P_\alpha^{n+1,j} - P_\alpha^{n+1,i}) = 0, \quad \alpha = w, o
%\label{eq:symboundPw}
%\end{equation}
%\end{proposition}
%\begin{proof}
%Let us fix $\alpha = w$ for  ledgibility. 
%Since $c_{ij} = 0$ whenever $j \notin \mathcal{N}(i)$ we can write
%\[
%X := \sum_{i} \sum_{j\neq i, j\in\mathcal{N}(i)} c_{ij} \eta_w(S_w^{n+1,ij}) ( (P_w^{n+1,j})^2 - (P_w^{n+1,i})^2) 
%= \sum_{i} \sum_{j\neq i} c_{ij} \eta_w(S_w^{n+1,ij}) ( (P_w^{n+1,j})^2 - (P_w^{n+1,i})^2) 
%\]
%As the term vanishes if $i=j$ we also have
%\[
%X = \sum_{i, j} c_{ij} \eta_w(S_w^{n+1,ij}) ( (P_w^{n+1,j})^2 - (P_w^{n+1,i})^2). 
%\]
%We note that $c_{ij} = c_{ji}$  and $S_w^{n+1,ij} = S_w^{n+1,ji}$ and we write
%\[
%X = \sum_{j, i} c_{ij} \eta_w(S_w^{n+1,ij}) ( (P_w^{n+1,i})^2 - (P_w^{n+1,j})^2)  = - X.
%\]
%This implies $X = 0$.
%The proof of \eqref{eq:symboundPw} is identical.
%\end{proof}
%%
%

 By combining Theorem  \ref{thm:pressbound1} with Lemma
\ref{lem:boundetap}, we immediately derive a bound on the discrete auxiliary pressures. The bound \eqref {eq:bddglobpres} with $\alpha =o$ follows from the same with $\alpha =w$, \eqref{eq:pwg+pog}, and \eqref{eq:scheme3}.  

\begin{theorem}
\label{thm:globalpressurebd}
We have for $\alpha = w,o$
\begin{equation}
\label{eq:bddglobpres}
\eta_\ast \|\nabla\, U_{\alpha,h,\tau}\|^2_{L^2(\Om \times ]0,T[)} \le C,
\end{equation}
with the constant $C$ of \eqref{eq:pressbound1}.
\end{theorem}

%===========================

\section{Existence of numerical solution}
\label{sec:existence}

We fix $n\geq1$ and assume there exists a solution $(S_h^{n-1}, P_{w,h}^{n-1})$ at time $t^{n-1}$ with $0 \le S_h^{n-1} \le 1$. We want to show existence of a solution $(S_h^{n}, P_{w,h}^{n})$ by means of
the topological degree~\cite{Deimling85,Dinca09}.

Let $\theta$ be a constant parameter in $[0,1]$. For any continuous function $f : [0,1] \rightarrow \real$, we define the transformed function $\widetilde{f}: [0,1] \rightarrow \real$ by
$$\forall s\in[0,1],\quad \widetilde{f}(s)=f(ts+(1-t)\theta).$$
Since $\theta$ is fixed, when $t=0$, $\widetilde{f}(s) = f(\theta)$, a constant independent of $s$. Now, \eqref{eq:vars4} implies that any solution $P_{w,h,\tau}$ of \eqref{eq:vars1}--\eqref{eq:vars4} belongs to the following subspace $X_{0,h}$ of $X_h$,
\begin{equation}
X_{0,h} = \{ \Lambda_h \in X_h; \int_\Omega \Lambda_h=0\}.
\label{eq:def_X0h}
\end{equation}
This suggests to define the mapping $\mathcal{F}: [0,1]\times X_h \times X_{0,h} \rightarrow X_h\times X_{0,h}$ by
\[
\mathcal{F}(t,\zeta,\Lambda) = (A_h,A_h+B_h),
\]
where $A_h$, respectively $B_h$, solves for all $\Theta_h\in X_h$,
\begin{equation}
 \label{eq:def_Ah_exist_proof}
\begin{split}
 (A_h,\Theta_h) =  \frac{1}{\tau} (\zeta_h - S_h^{n-1},\Theta_h)_h^\varphi 
 -\big[\Lambda_h,I_h(\widetilde{\eta_w}(\zeta_h)); \Lambda_h,\Theta_h\big]_h 
 \\
 - \big(I_h(\widetilde{f_w}(s_{\mathrm{in},h}^{n})) t \bar{q}_h^{n} - I_h(\widetilde{f_w}(\zeta_h)) t \underline{q}_h^{n},\Theta_h\big)_h,
\end{split}
\end{equation}
\begin{equation}
 \label{eq:def_Bh_exist_proof}
\begin{split}
 (B_h,\Theta_h) =  -\frac{1}{\tau} (\zeta_h - S_h^{n-1},\Theta_h)_h^\varphi 
 -\big[P_{o,h},I_h(\widetilde{\eta_o}(\zeta_h)); P_{o,h},\Theta_h\big]_h 
 \\
 - \big(I_h(\widetilde{f_o}(s_{\mathrm{in},h}^{n})) t \bar{q}_h^{n} - I_h(\widetilde{f_o}(\zeta_h)) t \underline{q}_h^{n},\Theta_h\big)_h,
\end{split}
\end{equation}
and  $P_{o,h}$ is defined by
\begin{equation}
\label{eq:def_Po_exist_proof}
 P_{o,h}=\Lambda_h - I_h(\widetilde{p_c}(\zeta_h)). 
\end{equation}
The choice of $\widetilde{\eta_w}(\zeta_h)$ in \eqref{eq:def_Ah_exist_proof} (respectively $\widetilde{\eta_o}(\zeta_h)$ in \eqref{eq:def_Bh_exist_proof})
is given by \eqref{eq:Swij} (respectively \eqref{eq:Soij}) where $\Lambda_h$ plays the role of $P_{w,h}$ and $P_{o,h}$ is defined in \eqref{eq:def_Po_exist_proof}. 
As in \eqref{eq:Swij} and \eqref{eq:Soij}, it leads us to introduce the variables $\zeta_w^{ij}$ and $\zeta_o^{ij}$ for all $1 \leq i,j \leq M$.
Clearly, \eqref{eq:def_Ah_exist_proof}--\eqref{eq:def_Po_exist_proof} determine uniquely $A_h$ and $B_h$, and it is easy to check that $A_h+B_h$ belongs to $X_{0,h}$. 

The mapping $t\mapsto \mathcal{F}(t,\zeta_h,\Lambda_h)$ is continuous. Indeed, since the space has finite dimension, we only need to check continuity of the upwinding. By splitting $x$ into its positive and negative part, $x = x^+ + x^-$, the upwind term, say $\widetilde{\eta_w}(\zeta_w^{ij}) (P_w^j-P_w^i)$ reads
$$\widetilde{\eta_w}(\zeta_w^{ij}) (P_w^j-P_w^i) = \eta_w(t \zeta^i + (1-t) \theta) \big((P_w^j-P_w^i)_-\big) + \eta_w(t \zeta^j +(1-t) \theta) \big((P_w^j-P_w^i)_+\big),
$$
which is continuous with respect to $t$.

We remark that $\mathcal{F}(1,\zeta_h,\Lambda_h) = {\bf 0}$ implies that
$(\zeta_h, \Lambda_h)$ solves \eqref{eq:vars1}--\eqref{eq:vars4}.
Conversely, if $(\zeta_h, \Lambda_h)$ solves \eqref{eq:vars1}--\eqref{eq:vars4}
then $\mathcal{F}(1,\zeta_h,\Lambda_h) = {\bf 0}$. Thus showing  existence of a 
solution to the problem \eqref{eq:vars1}--\eqref{eq:vars4} is equivalent to showing  existence of a zero of $\mathcal{F}(1,\zeta_h,\Lambda_h)$.
Before proving existence of a zero, we use the estimates 
established in the previous section to determine an a priori bound of any zero $(\zeta_h,\Lambda_h)$ of $\mathcal{F}(1,\zeta_h,\Lambda_h)$. 

%-----------------------------

\subsection{A priori bounds on $(\zeta_h,\Lambda_h)$}
\label{subsec:apriorizero}

In the following we consider $t\in[0,1]$ and $(\zeta_h,\Lambda_h)\in X_h \times X_{0,h}$ that
satisfy 
\begin{equation}
 \mathcal{F}(t,\zeta_h,\Lambda_h) = {\bf 0}.
 \label{eq:problem_F_eq_0}
\end{equation}

We first show that $\zeta_h$ satisfies a maximum principle.
\begin{proposition} 
\label{pro:existence_zeta_bound}
The following bounds hold for all $(t,\zeta_h,\Lambda_h)$ satisfying \eqref{eq:problem_F_eq_0}:
 \begin{equation}\label{eq:existence_maxprinc}
0\leq \zeta_h \leq 1.
\end{equation}
\end{proposition}

\begin{proof}
Either $t\in]0,1]$ or $t=0$. The proof
for $t\in]0,1]$ follows closely the argument used in proving Theorem~\ref{thm:maxprinc} 
and is left to the reader. For $t=0$ we proceed again by contradiction.
Assume first that $\|\zeta_h\|_{L^\infty(\Omega)}>1$, i.e., there is a node $i$ such that
\[
\zeta^i = \|\zeta_h\|_{L^\infty(\Omega)} > 1 \geq S^{n-1,i}.
\]
As $t=0$, \eqref{eq:problem_F_eq_0} reduces to
\[
\sum_{j\neq i} c_{ij} \eta_w(\theta) (\Lambda^i-\Lambda^j) > 0,\qquad
-\sum_{j\neq i} c_{ij} \eta_o(\theta) (\Lambda^i-\Lambda^j) > 0, \qquad \forall 1\leq i \leq M.
\]
Since $\eta_o$ and $\eta_w$ are non-negative functions satisfying \eqref{eq:lowerboundetas}, the inequalities above yield a contradiction.
A similar argument is used to show that $\zeta_h \geq 0$.
\end{proof}

Next we show the following bound on $\Lambda_h$.
\begin{proposition}
\label{pro:existence_lambda_bound}
There is a constant $C$ such that for all $t\in[0,1]$ we have
\begin{eqnarray}
\eta_\ast \sum_{i,j=1}^M c_{ij} 
\left( \Lambda^j - \Lambda^i + p_{wg}(t\zeta^j + (1-t)\theta) - p_{wg}(t\zeta^i + (1-t)\theta) \right)^2 \leq C.
\label{eq:Lambdabound1}
\end{eqnarray}
\end{proposition}

\begin{proof}
The proof follows closely that of Theorem~\ref{thm:globalpressurebd}. 
First we show there exists a constant $C_1$ independent of $t$ such that
\begin{eqnarray*}
\sum_{i,j=1}^M c_{ij} \Big(\eta_w(t \zeta_w^{ij} + (1-t) \theta) (\Lambda^j-\Lambda^i)^2
+  \eta_o(t \zeta_o^{ij} + (1-t) \theta) (P_{o,h}^j-P_{o,h}^i)^2 \leq C_1,
\end{eqnarray*}
with $P_{o,h}$ defined in \eqref{eq:def_Po_exist_proof}.
This bound is obtained via  arguments similar to those used in proving Theorem~\ref{thm:pressbound1}. 
The main difference is that the formula is neither summed over $n$ nor multiplied by the time step 
 $\tau$. As a consequence, the constant $C_1$ includes a term of the form
 $\tau^{-1} \|g_c\|_{L^\infty(\Om)}$ arising from the  bound of the discrete  time derivative.
To finish the proof we must show that 
\begin{eqnarray*}
\eta_\ast \left( \Lambda^{j} - \Lambda^{i} 
+ p_{wg}(t \zeta^{j} + (1-t)\theta) - p_{wg}(t \zeta^{i} + (1-t)\theta) 
\right)^2
\leq \eta_w(t \zeta^{ij} + (1-t)\theta) (\Lambda^{j}-\Lambda^{i})^2 
\\
+ \eta_o(t \zeta_o^{ij} + (1-t)\theta) (P_o^{j}-P_o^{i})^2.
\end{eqnarray*}
By \eqref{eq:lowerboundetas}, this is trivially satisfied  when $t=0$. When 
 $t\in]0,1]$, the argument is the same as in the proof
of Lemma~\ref{lem:boundetap}.
\end{proof}

Propositions  \ref{pro:existence_zeta_bound} and \ref{pro:existence_lambda_bound}  are combined to obtain a bound on $\|\zeta_h\|_h +\|\Lambda_h\|_h$.
\begin{proposition}
 \label{pro:def_R0_R1}
 There exists a constant $R_1>0$, independent of $t\in[0,1]$, such that any solution $(\zeta_h,\Lambda_h)$ of \eqref{eq:problem_F_eq_0} satisfies
 \begin{equation}
 \label{eq:bddR1}
 \|\zeta_h\|_h +\|\Lambda_h\|_h \le R_1.
 \end{equation}
\end{proposition}

\begin{proof} According to 
Proposition~\ref{pro:existence_zeta_bound}, there exists
a constant $C_1$ independent of $t$ such that
$$
\| \zeta_h\|_h \leq C_1.
$$
To establish a bound on $\|\Lambda_h\|_h$, we infer from \eqref{eq:defpwgpog}
that the function $|p_{wg}|$ is bounded by $p_c(0)-p_c(1)$ because
 $f_o$ is bounded by one and $p_c$ is a decreasing function. Thus \eqref{eq:Lambdabound1}
implies that there exists  a constant $C_2$ 
independent of $t$ that satisfies
\begin{eqnarray}
\sum_{i,j=1}^M c_{ij} \left( \Lambda^j - \Lambda^i  \right)^2 \leq C_2,\quad \mbox{i.e.,}\ \| \nabla \Lambda_h \|_{L^2(\Omega)} \leq \frac{\sqrt{C_2}}{\sqrt{2}},
\label{eq:existence_bound_Lambdai}
\end{eqnarray}
owing to \eqref{eq:normgrad}. As $\Lambda_h \in X_{0,h}$, the generalized Poincar\'e 
inequality \eqref{eq:genPoinc}
shows there exists a constant $C_3$ independent of $t$ such that
$$
\| \Lambda_h \|_{L^2(\Omega)} \leq C_3.
%\label{eq:existence_bound_Lambdah_L2}
$$
Then the equivalence of norm \eqref{eq:equivnorm} yields
$$
\| \Lambda_h \|_h \leq C_4,
%\label{eq:existence_bound_Lambdah_h}
$$
and  \eqref{eq:bddR1} follows by setting $R_1=C_1 + C_4$, a constant independent of $t$.
\end{proof}

%-----------------------

\subsection{Proof of existence}
\label{subsec:endproofexist}

For any $R>0$, let $B_R$ denote the ball
\begin{equation}
B_{R} = \{ (\zeta_h,\Lambda_h) \in X_h\times X_{0,h}: \, \Vert\zeta_h\Vert_h +\Vert\Lambda_h\Vert_h  \leq R\},
\label{eq:def_R0}
\end{equation}
and let $R_0=R_1+1$, where $R_1$ is the constant of \eqref{eq:bddR1}. Since all solutions $(\zeta_h,\Lambda_h)$ of  \eqref{eq:problem_F_eq_0} are in the ball $B_{R_1}$, this function has no zero on the boundary $\partial B_{R_0}$.
Existence of a  solution of  \eqref{eq:vars1}--\eqref{eq:vars4} follows from the following result:
\begin{theorem}
The equation $\mathcal{F}(1,\zeta_h,\Lambda_h) = {\bf 0}$ has at least one solution 
$(\zeta_h,\Lambda_h)\in B_{R_0}$.
 \end{theorem}

\begin{proof}
The proof proceeds in two steps.\\
First, we show that the system with $t=0$ has a solution:
\[
\mathcal{F}(0,\zeta_h,\Lambda_h) = 0.
\]
This is a square linear system in finite dimension, so existence is equivalent to uniqueness. 
Thus we assume that it has two solutions, and for convenience, 
we still denote by $(\zeta_h,\Lambda_h)$ the difference between the two solutions.
The system reads
\begin{eqnarray}
\frac{\tilde{m}_i}{\tau} \zeta_h^i 
-\sum_{j\neq i, j\in \mathcal{N}(i)} c_{ij} \eta_w(\theta) (\Lambda^j-\Lambda^i)= 0, \quad 1\leq i\leq M,
\label{eq:thirdonea}
\\
-\frac{\tilde{m}_i}{\tau} \zeta_h^i
-\sum_{j\neq i, j\in \mathcal{N}(i)} c_{ij} \eta_o(\theta)
(\Lambda^j-\Lambda^i) = 0, \quad 1\leq i\leq M,
\\
\sum_i m_i \Lambda^i = 0.
\label{eq:thirdone}
\end{eqnarray} 
We add the first two equations, 
multiply by $\Lambda^i$, and sum over i. Then \eqref{eq:normgrad}
and \eqref{eq:quadup2} imply that $\Lambda_h$ is a constant and finally 
 \eqref{eq:thirdone} shows that this constant is zero.  This yields $\zeta_h = 0$.
 
Next, we argue on the topological degree. 
Since the topological degree of a linear map is the sign of its determinant, we have
\[
d(\mathcal{F}(0,\zeta_h,\Lambda_h),B_{R_0},0) \neq 0.
\]
We also know that
$d(\mathcal{F}(t,\zeta_h,\Lambda_h),B_{R_0},0)$ is independent of $t$ since the mapping $t\mapsto \mathcal{F}(t,\zeta_h,\Lambda_h)$
is continuous and for every $t\in[0,1]$, if $\mathcal{F}(t,\zeta_h,\Lambda_h)=0$, then $(\zeta_h,\Lambda_h)$ does not belong to $\partial B_{R_0}$.
Therefore we have
\[
d(\mathcal{F}(1,\zeta_h,\Lambda_h), B_{R_0}, 0) = 
d(\mathcal{F}(0,\zeta_h,\Lambda_h),B_{R_0},0) \neq 0.
\]
This implies that $\mathcal{F}(1,\zeta_h,\Lambda_h)$ has a zero $(\zeta_h,\Lambda_h)\in B_{R_0}$.
\end{proof}

%==================================

\section{Additional pressure estimates}
\label{sec:addpressbdd}

The pressure estimates \eqref{eq:pressbound1} and \eqref{eq:bddglobpres} are not sufficient to pass to the limit in the scheme \eqref{eq:vars1}--\eqref{eq:vars4}. These are nonlinear equations and we need strong convergences that do not stem directly from \eqref{eq:pressbound1} and \eqref{eq:bddglobpres}. Following~\cite{Eymard2003}, we propose to derive a bound for the gradient of \Rd $g$, see \eqref{eq:g},
at \Bk  $s= S_{h,\tau}$. 
Then, under suitable assumptions on the behavior of $\eta_w^\prime$, $\eta_o^\prime$, and $p_c^\prime$, we shall prove the strong convergence of $g(S_{h,\tau})$ in $L^2(\Om \times ]0,T[)$ and in turn the strong convergence of
$S_{h,\tau}$ in $L^2(\Om \times ]0,T[)$.

Estimating the gradient of $g(S_{h,\tau})$ is a long and intricate process; it is based on the fact that
$$|g(S^{n,j})-g(S^{n,i})|^2 \le C\big(f_w(S^{n,i})-f_w(S^{n,j})\big)\big(g(S^{n,i})-g(S^{n,j})\big),
$$
see \eqref{eq:deltafdeltag}. Therefore, we must derive a bound for the product of the gradients of $g$ and $f_w$. 
This is split into several steps. 

%-------------------------
\subsection{A preliminary inequality}
\label{subsec:prelimineq}

Our starting step is the following inequality:

\begin{proposition}
\label{pro:f.eta.gradp}
There exists a constant $C_1$ independent of $h$ and $\tau$ such that
\begin{equation}
\label{eq:f.eta.gradp1}
- \sum_{n=1}^N \tau \sum_{\alpha = o,w} \big[P_{\alpha,h}^n,\eta_\alpha(S_{\alpha,h}^n);f_\alpha(S_h ^n), P_{\alpha,h}^n\big]_h = R_1,
\end{equation}
where the remainder $R_1$ satisfies $|R_1| \le C_1$.
\end{proposition}

\begin{proof}
By testing \eqref{eq:vars1} with $I_hf_w(S_h^{n})$ and \eqref{eq:vars2} with $I_hf_o(S_h^{n})$, adding the resulting equalities, and multiplying by $\tau$, we obtain
\begin{equation}
\label{eq:sumf.eta.gradp4}
\begin{split}
\sum_{n=1}^{N} &\big(S_h^n-S_h^{n-1},f_w(S_h^{n}) - f_o(S_h^{n})\big)_h^\varphi - \sum_{n=1}^N \tau \sum_{\alpha = o,w}\big[P_{\alpha,h}^n,\eta_\alpha(S_{\alpha,h}^n);f_\alpha(S_h^n), P_{\alpha,h}^n\big]_h\\
& = \int_0^T \Big(\big(\bar q_{h,\tau}, \sum_{\alpha = o,w} f_\alpha(s_{\mathrm{in},h,\tau} ) f_\alpha(S_{h,\tau})\big)_h- \big(\underline q_{h,\tau}, \sum_{\alpha = o,w}(f_\alpha(S_{h,\tau}))^2\big)_h\Big) \le 4  \|\bar q\|_{L^1(\Om \times ]0,T[)},
\end{split}
\end{equation}
in view of \eqref{eq:fluidfract} and \eqref{eq:flowrate}. To control the time difference of $S_{h,\tau}$, 
we introduce the global flux defined by
\begin{equation}
\label{eq:globflux}
 \forall x \in [0,1],\quad G(x) = \int_0^x \big(f_w(s) - f_o(s)\big) d s,
\end{equation}
and we write
  $$(S_h^{n}- S_h^{n-1})\big(f_w(S_h^{n}) - f_o(S_h^{n})\big) = (S_h^{n}- S_h^{n-1}) G^\prime(S_h^{n}).
$$ 
But by \eqref{eq:fluidfract}, $G^\prime(x) = 2 f_w(x)-1$ is increasing. Hence, considering that
$$G(S_h^n)- G(S_h^{n-1}) = (S_h^n- S_h^{n-1})G^\prime(c),$$
for some $c$ between $S_h^n- S_h^{n-1}$, we easily check that 
$$G(S_h^n)- G(S_h^{n-1}) \le  (S_h^n- S_h^{n-1}) G^\prime(S_h^{n}).
$$
Thus, the properties of $\varphi$ imply
$$\sum_{n=1}^{N} \big(S_h^n-S_h^{n-1},f_w(S_h^{n}) - f_o(S_h^{n})\big)_h^\varphi \ge (G(S_h^N),1)_h^\varphi-
(G(S_h^0),1)_h^\varphi.
$$
But the boundedness of $S_{h,\tau}$, the continuity of $f_\alpha$, and the properties of $\varphi$ imply
$$\big| (G(S_h^N),1)_h^\varphi-(G(S_h^0),1)_h^\varphi\big| \le C^\prime,
$$
with a constant $C^\prime$ independent of $h$ and $\tau$. By substituting these inequalities into \eqref{eq:sumf.eta.gradp4} we
derive \eqref{eq:f.eta.gradp1} with $C_1 = 4\, \|\bar q\|_{L^1(\Om \times ]0,T[)} +  C^\prime$.
\end{proof}

%-------------------------
\subsection{Some discrete total flux inequalities}
\label{subsec:discglobflux}

In this section, it is convenient to work directly on the scheme \eqref{eq:scheme1}--\eqref{eq:scheme2}. For each index $i$, the sum of the equations \eqref{eq:scheme1} and \eqref{eq:scheme2}  give, for $1 \le i \le M$ and $1 \le n \le N$,
$$
- \sum_{j\neq i, j\in \mathcal{N}(i)} c_{ij}\Big[\eta_w(S_w^{n,ij})(P_w^{n,j}-P_w^{n,i}) + 
\eta_o(S_o^{n,ij})(P_o^{n,j}-P_o^{n,i})\Big] =m_i(\bar q^{n,i}- \underline{q}^{n,i}).
$$
Following~\cite{Eymard2003}, this suggests to define a discrete anti-symmetric upwinded total flux, 
\begin{equation}
\label{eq:discglobfluxF}
F^{n,ij} = - \eta_w(S_w^{n,ij})(P_w^{n,j}-P_w^{n,i}) - \eta_o(S_o^{n,ij})(P_o^{n,j}-P_o^{n,i});
\end{equation}
it satisfies
\begin{equation}
\label{eq:sumFlux}
\sum_{j\neq i, j\in \mathcal{N}(i)} c_{ij} F^{n,ij} = m_i(\bar q^{n,i}- \underline{q}^{n,i}).
\end{equation}
 This identity yields a first bound for the discrete total flux.
 
\begin{proposition}
\label{pro:sumdiscglobflux}
The discrete total flux $F^{n,ij}$ satisfies the following bounds for $\alpha = w,o$:
\begin{equation}
\label{eq:sumflx}
\big|\sum_{n=1}^{N} \tau \sum_{i,j=1}^M  f_\alpha^2(S^{n,i}) c_{ij}F^{n,ij} \big| \le 2 \, \|\bar q\|_{L^1(\Om \times ]0,T[)}.
\end{equation}
\end{proposition}

\begin{proof}
The statement  follows by multiplying \eqref{eq:sumFlux} with $\tau \,f_\alpha^2(S^{n,i})$, and summing 
$$
\sum_{n=1}^{N} \tau \sum_{i,j=1}^M f_\alpha^2(S^{n,i})  c_{ij} F^{n,ij} = \sum_{n=1}^{N} \tau \sum_{i=1}^M m_i f_\alpha^2(S^{n,i}) (\bar q^{n,i}- \underline{q}^{n,i})\le 2\, \|\bar q\|_{L^1(\Om \times ]0,T[)}. 
$$
\end{proof}

To simplify some of the calculations below, it is convenient to drop the time superscript $n$, when there is no ambiguity, and restore it when needed.

By using the relation \eqref{eq:cap.pressure}, $F^{i,j}$ can also be written as
\begin{equation}
\label{eq:discglobfluxF2}
\begin{split}
F^{ij} =& - \big(\eta_w(S_w^{ij})+ \eta_o(S_o^{ij})\big)(P_w^{j}-P_w^{i}) - \eta_o(S_o^{ij})(p_c(S^j)-p_c(S^i))\\
 = & - \big(\eta_w(S_w^{ij})+ \eta_o(S_o^{ij})\big)(P_o^{j}-P_o^{i}) + \eta_w(S_w^{ij})(p_c(S^j)-p_c(S^i)).
 \end{split}
\end{equation}
In order to insert it into \eqref{eq:f.eta.gradp1}, we bring forward $F^{ij}$ in the expressions for $\eta_\alpha(S_\alpha^{ij}) (P_\alpha^{j}-P_\alpha^{i})$, $\alpha = w,o$. Starting from the identity
\begin{align}
\eta_w(S_w^{ij}) (P_w^{j}-P_w^{i}) = f_w(S_w^{ij}) \Big[\big(\eta_w(S_w^{ij})+ \eta_o(S_o^{ij})\big)(P_w^{j}-P_w^{i}) + \eta_o(S_o^{ij})(p_c(S^j)-p_c(S^i))\nonumber \\
 - \eta_o(S_o^{ij})(p_c(S^j)-p_c(S^i)) + \big(\eta_o(S_w^{ij})- \eta_o(S_o^{ij})\big)(P_w^{j}-P_w^{i})\Big],\nonumber
\end{align}
the expression \eqref{eq:discglobfluxF2} leads to 
\begin{equation}
\eta_w(S_w^{ij}) (P_w^{j}-P_w^{i}) = f_w(S_w^{ij}) \Big[-F^{ij}
- \eta_o(S_o^{ij})(p_c(S^j)-p_c(S^i)) + \big(\eta_o(S_w^{ij})- \eta_o(S_o^{ij})\big)(P_w^{j}-P_w^{i})\Big].
 \label{eq:discglobfluxFw} 
\end{equation}
Similarly,
\begin{equation}
\eta_o(S_o^{ij}) (P_o^{j}-P_o^{i}) = f_o(S_o^{ij}) \Big[-F^{ij}
+ \eta_w(S_w^{ij})(p_c(S^j)-p_c(S^i)) + \big(\eta_w(S_o^{ij})- \eta_w(S_w^{ij})\big)(P_o^{j}-P_o^{i})\Big].
 \label{eq:discglobfluxFo} 
\end{equation}
We also introduce the anti-symmetric quantities that collect the terms other than $F^{ij}$ in \eqref{eq:discglobfluxFw} and
\eqref{eq:discglobfluxFo}, 
\begin{equation}
\label{eq:Cij}
C_w^{ij}= \eta_o(S_o^{ij})\big(p_c(S^{j})-p_c(S^{i})\big)- \big(\eta_o(S_w^{ij})-\eta_o(S_o^{ij})\big)(P_w^{j}-P_w^{i}),
\end{equation}
\begin{equation}
\label{eq:Dij}
C_o^{ij}= -\eta_w(S_w^{ij})\big(p_c(S^{j})-p_c(S^{i})\big)- \big(\eta_w(S_o^{ij})-\eta_w(S_w^{ij})\big)(P_o^{j}-P_o^{i}).
\end{equation}
With this notation, we have
$$\eta_\alpha(S_\alpha^{ij})(P_\alpha^{j}-P_\alpha^{i}) = f_\alpha(S_\alpha^{ij}) \big[-F^{ij} -C_\alpha^{ij}\big],\ \alpha =w,o.
$$
Thus, the term that is summed over $i$ in \eqref{eq:f.eta.gradp1} has the expression
\begin{equation}
\label{eqn:FCD}
-\sum_{\alpha =w,o} f_\alpha(S^{i})\sum_{j\neq i, j\in \mathcal{N}(i)} c_{ij}\eta_\alpha(S_\alpha^{ij})(P_\alpha^{j}-P_\alpha^{i}) 
 = \sum_{\alpha =w,o} f_\alpha(S^{i})\sum_{j\neq i, j\in \mathcal{N}(i)} c_{ij} f_\alpha(S_\alpha^{ij})
\big(F^{ij} + C_\alpha^{ij}\big).
\end{equation}

Now, we reintroduce the superscript $n$ and to simplify, we set
\begin{equation}
\label{eq:A10}
A_{1,i,n} = \sum_{\alpha = w,o} f_\alpha (S^{n,i})\sum_{j=1}^M c_{ij}f_\alpha(S_\alpha^{n,ij})F^{n,ij},
\end{equation}
\begin{equation}
\label{eq:Awn}
A_{\alpha,i,n} = f_\alpha(S^{n,i})\sum_{j=1}^M c_{ij}f_\alpha(S_\alpha^{n,ij})C_\alpha^{n,ij}.
\end{equation}
With this notation, our next proposition is derived by substituting  \eqref{eqn:FCD}--\eqref{eq:Awn} into
\eqref{eq:f.eta.gradp1}. 
\begin{proposition}
\label{pro:discglobflux}
We have, with the remainder $R_1$ of \eqref{eq:f.eta.gradp1},
\begin{equation}
\label{eq:f.eta.gradpflux}
\sum_{n=1}^{N} \tau \sum_{i=1}^M A_{1,i,n} + \sum_{n=1}^{N} \tau \sum_{i=1}^M \sum_{\alpha =w,o} A_{\alpha,i,n}   =  R_1.
\end{equation}
\end{proposition}

We must transform suitably each term in this sum to bring forward $g$. Let us start with the first term of \eqref{eq:f.eta.gradpflux}, i.e., the combination of the discrete total flux.

%----------------------

\subsection{Combination of the discrete total flux}
\label{sec:A1}

To simplify, let $A_1$ denote the first term, 
$$A_1 = \sum_{n=1}^{N} \tau \sum_{i,j=1}^M\sum_{\alpha =w,o} \Big[f_\alpha(S^{n,i}) c_{ij}f_\alpha(S_\alpha^{n,ij})F^{n,ij}\Big].
$$
Inspired by \eqref{eq:sumflx}, we introduce the difference
$$A_2 = A_1 - \sum_{n=1}^{N} \tau \sum_{i,j=1}^M  \big(f_w^2(S^{n,i}) + f_o^2(S^{n,i})\big) c_{ij}F^{n,ij}.
$$
Clearly, $A_2$ collects the discrepancies arising from the upwinding,
\begin{equation}
\label{eq:T2}
A_2 = \sum_{n=1}^{N} \tau \sum_{i,j=1}^M \sum_{\alpha =w,o} \Big[f_\alpha(S^{n,i}) c_{ij}\big(f_\alpha(S_\alpha^{n,ij})-f_\alpha(S^{n,i})\big)F^{n,ij}\Big].
\end{equation}
As \eqref{eq:sumflx} yields
\begin{equation}
\label{eq:T2-T1}
A_1 = A_2 + R_2,\quad \mbox{with} \ |R_2| \le 4 \, \|\bar q\|_{L^1(\Om \times ]0,T[)},
\end{equation}
a bound for $A_1$ stems from  a bound for $A_2$. To this end, in view of \eqref{eq:T2}, it is useful to consider the four subsets of indices $j \in \mathcal{N}(i), j\ne i$, union and intersection:
\begin{equation}
\label{eq:setsN}
\begin{split}
&\mathcal{N}_w(i) = \{j \in \mathcal{N}(i)\,;\, P_w^{n,j} > P_w^{n,i}\},\quad \mathcal{N}_o(i) = \{j \in \mathcal{N}(i)\,;\, P_o^{n,j} > P_o^{n,i}\}\\
&\mathcal{N}_{w,S}(i) =\{j \in \mathcal{N}(i), j\ne i\,;\, P_w^{n,j} = P_w^{n,i}, S^{n,j} \ge S^{n,i}\},\\ & \mathcal{N}_{o,S}(i) =\{j \in \mathcal{N}(i),j\ne i\,;\, P_o^{n,j} = P_o^{n,i}, S^{n,j} \le S^{n,i}\},\\
& \mathcal{UN}(i) = \mathcal{N}_w(i) \cup \mathcal{N}_o(i)\cup \mathcal{N}_{w,S}(i)\cup \mathcal{N}_{o,S}(i),\\
& \mathcal{N}_{\mathcal F}(i) = \{j \in \mathcal{N}(i)\,;\, P_w^{n,i} > P_w^{n,j} \ \mbox{and}\ P_o^{n,i} > P_o^{n,j}\}.
\end{split}
\end{equation}
Strictly speaking, these subsets should we written with the superscript $n$, but we omit it for the sake of simplicity.  
Then we have the following bound for $A_2$:

\begin{proposition}
\label{pro:bddT2}
There exists a constant $C_2$, independent of $h$ and $\tau$, such that
\begin{equation}
\label{eq:bddT2} 
A_2 = -\frac{1}{2}\sum_{n=1}^{N} \tau \sum_{i=1}^M \sum_{j\in \mathcal{UN}(i)} c_{ij} \big(f_w(S^{n,j})-f_w(S^{n,i})\big)^2 F^{n,ij} + R_3,
\end{equation}
where the remainder $R_3$ satisfies
$$|R_3| \le C_2 = 2  \, \|\bar q\|_{L^1(\Om \times ]0,T[)}.
$$
\end{proposition}

\begin{proof}
Let us drop the superscript $n$.
By definition, $f_w(S_w^{ij})-f_w(S^{i}) = 0$ when $P_w^{i} > P_w^{j}$ and when $P_w^{i} = P_w^{j}$ and $S^{i} > S^{j}$. Similarly, $f_o(S_o^{ij})-f_o(S^{i}) = 0$ when $P_o^{i} > P_o^{j}$ and when 
$P_o^{i} = P_o^{j}$ and $S^{n,i} < S^{j}$. Therefore, the $n$th term in $A_2$, say $a_2$, reduces to
$$
a_2 = \sum_{i=1}^M \sum_{\alpha =w,o} f_\alpha(S^{i})\sum_{j\in\mathcal{N}_\alpha(i)\cup \mathcal{N}_{\alpha,S}(i) } c_{ij}\big(f_\alpha (S^{j})-f_\alpha(S^{i})\big)F^{ ij}.
$$
By expanding the products, this can be written
\begin{equation}
a_2 = - \frac{1}{2} \sum_{i=1}^M \sum_{\alpha =w,o}\sum_{j\in\mathcal{N}_\alpha(i)\cup \mathcal{N}_{\alpha,S}(i) } c_{ij}\big(f_\alpha^2(S^{i})- f_\alpha^2(S^{j}) + (f_\alpha(S^{i})-f_\alpha(S^{j}))^2\big)F^{ ij}.
\label{eq:T2_1}
\end{equation}
Since $c_{ij}$ vanishes when $j$ is not a neighbor of $i$, we have, by interchanging $i$ and $j$ and using the anti-symmetry of $F^{ij}$ and the symmetry of $c_{ij}$, 
\begin{equation}
\label{eq:symcantisymF1}
- \sum_{i=1}^M\sum_{j\in\mathcal{N}_w(i)} c_{ij} f_w^2(S^{j}) F^{ ij} = \sum_{i=1,j =1,P_w^{j} < P_w^{i} }^M c_{ij}f_w^2(S^{i}) F^{ ij}.
\end{equation}
Similarly,
\begin{equation}
\label{eq:symcantisymF2}
- \sum_{i=1}^M\sum_{j\in\mathcal{N}_{w,S}(i)} c_{ij}f_w^2(S^{j}) F^{ ij} = \sum_{i=1,j =1,P_w^{i} = P_w^{j}, S^{i} \ge S^{j} }^M c_{ij}f_w^2(S^{i}) F^{ ij}.
\end{equation}
Hence
$$- \frac{1}{2} \sum_{i=1}^M \sum_{j\in\mathcal{N}_w(i)} c_{ij}\big(f_w^2(S^{i})-f_w^2(S^{j})) F^{ ij}= - \frac{1}{2} \sum_{i=1,j=1, P_w^{i} \ne P_w^{j}}^M c_{ij}f_w^2(S^{i}) F^{ ij},
$$
and 
$$
- \frac{1}{2} \sum_{i=1}^M \sum_{j\in\mathcal{N}_{w,S}(i)} c_{ij}\big(f_w^2(S^{i})- f_w^2(S^{j})\big) F^{ ij}= - \frac{1}{2} \sum_{i=1,j=1, P_w^{i} = P_w^{j}}^M c_{ij}f_w^2(S^{i}) F^{ ij},
$$
because there is no contribution from the indices $i,j$ such that $P_w^{i} = P_w^{j}, S^{i} = S^{j}$ since in this case the factor $F^{ ij} =0$. The same is true for the non-wetting phase. Thus
$$
- \frac{1}{2}\sum_{\alpha =w,o}\sum_{i=1}^M \sum_{j\in\mathcal{N}_\alpha(i)\cup \mathcal{N}_{\alpha,S}(i) } c_{ij}\big(f_\alpha^2(S^{i})-f_\alpha^2(S^{j})\big) F^{ ij} = - \frac{1}{2}  \sum_{\alpha = w,o}\sum_{i=1,j=1}^M c_{ij}f_\alpha^2(S^{i}) F^{ ij}.
$$
%i.e.,
%\begin{equation}
%\label{eq:symcantisymF3}
%- \frac{1}{2}\sum_{n=1}^{N} \tau \sum_{i=1}^M \sum_{j\in\mathcal{N}_w(i)\cup \mathcal{N}_{w,S}(i) } c_{ij}\big(f_w^2(S^{n,i})-f_w^2(S^{n,j})\big) F^{n, ij} =  - \frac{1}{2}\sum_{n=1}^{N} \tau \sum_{i=1}^M \sum_{j \in \mathcal{N}(i)} c_{ij} f_w^2(S^{n,i}) F^{n, ij}.
%\end{equation}
By comparing with \eqref{eq:sumflx}, we see that
\begin{equation}
\label{eq:symcantisymF4}
\big|\frac{1}{2}\sum_{n=1}^{N} \tau\sum_{i=1}^M \sum_{\alpha =w,o} \sum_{j\in\mathcal{N}_\alpha(i)\cup \mathcal{N}_{\alpha,S}(i) } c_{ij}\big(f_\alpha^2(S^{n,i})-f_\alpha^2(S^{n,j})\big) F^{n, ij}\big| \le 2\|\bar q\|_{L^1(\Om \times ]0,T[)}.
\end{equation}
 This and the equality
$$\big(f_o(S^{n,j})-f_o(S^{n,i})\big)^2 = \big(f_w(S^{n,j})-f_w(S^{n,i})\big)^2,
$$
readily imply \eqref{eq:bddT2}.
\end{proof}

Now, we set
\begin{equation}
\label{eq:T3} 
A^{ij} = c_{ij}\big(f_w(S^{j})-f_w(S^{i})\big)^2 F^{ij},\quad a_3 = -\frac{1}{2}\sum_{i=1}^M \sum_{j\in \mathcal{UN}(i)}A^{ij}.
\end{equation}
The next proposition simplifies the expression for $a_3$.

\begin{proposition}
\label{pro:expT3}
We have
\begin{equation}
\label{eq:expT3} 
a_3 = \sum_{i=1}^M \sum_{j\in \mathcal{N}_{\mathcal F}(i)} c_{ij} \big(f_w(S^{j})-f_w(S^{i})\big)^2 F^{ij}.
\end{equation}
\end{proposition}

\begin{proof}
By expanding the indices in the set $\mathcal{UN}(i)$, interchanging the indices $i$ and $j$, and using the anti-symmetry of $A^{ij}$, we derive
$$
a_3 =  \frac{1}{2}\Big(\big(\sum_{\alpha = w,o}\sum_{ P_\alpha^i > P_\alpha^j} A^{ij}\big) + \sum_{ P_w^i = P_w^j,S^j \le S^i} A^{ij} + \sum_{ P_o^i = P_o^j,S^i \le S^j} A^{ij}\Big).
$$
Now, we split the first two sums above as follows:
$$
\sum_{\alpha = w,o}\sum_{ P_\alpha^i > P_\alpha^j} A^{ij} = 2 
 \sum_{ P_w^i > P_w^j, P_o^i> P_o^j } A^{ij} +
\sum_{ P_w^i > P_w^j, P_o^i \le P_o^j } A^{ij} +
\sum_{ P_o^i > P_o^j, P_w^i \le P_w^j } A^{ij}.
$$
This leads to
$$
a_3 = \sum_{ j\in\mathcal{N}_{\mathcal F}(i) } A^{ij} + \frac{1}{2}\Big(\sum_{ P_w^i > P_w^j, P_o^i \le P_o^j} A^{ij} + \sum_{ P_o^i > P_o^j,P_w^i \le P_w^j } A^{ij}
+ \sum_{ P_w^i = P_w^j,S^j \le S^i} A^{ij} + \sum_{ P_o^i = P_o^j,S^i \le S^j} A^{ij}\Big).
$$
The anti-symmetry of $A^{ij}$ gives
$$\sum_{ P_w^i > P_w^j, P_o^i \le P_o^j} A^{ij} = - \sum_{ P_w^j > P_w^i, P_o^j < P_o^i} A^{ij} - \sum_{ P_w^j > P_w^i, P_o^j = P_o^i} A^{ij}.
$$
By substituting and applying twice again the anti-symmetry of $A^{ij}$, we derive 
\begin{equation}
a_3 = \sum_{j\in\mathcal{N}_{\mathcal F}(i) } A^{ij} + \frac{1}{2}\Big(\sum_{ P_o^i > P_o^j,P_w^i = P_w^j} A^{ij} + \sum_{ P_w^i > P_w^j, P_o^i = P_o^j} A^{ij}
- \sum_{ P_w^i = P_w^j,S^i \le S^j} A^{ij} - \sum_{ P_o^i = P_o^j,S^j \le S^i} A^{ij}\Big).
\label{eq:interA3}
\end{equation}
Note that 
$$ \sum_{ P_o^i > P_o^j,P_w^i = P_w^j} A^{ij} = \sum_{ P_o^i \ge P_o^j,P_w^i = P_w^j} A^{ij},
$$
since the additional term  is zero. Therefore, in view of first \eqref{eq:remarkSiSj2} and next
 \eqref{eq:remarkSiSj3}, 
$$ \sum_{ P_o^i > P_o^j,P_w^i = P_w^j} A^{ij} = \sum_{ P_w^i = P_w^j, S^i \le S^j} A^{ij},\quad \sum_{ P_w^i > P_w^j, P_o^i =P_o^j} A^{ij} =\sum_{ P_o^i = P_o^j, S^i \ge S^j} A^{ij}.
$$
Thus all  terms multiplying $\frac{1}{2}$ in \eqref{eq:interA3} are cancelled and we recover \eqref{eq:expT3}.
\end{proof}

By applying \eqref{eq:T2-T1} and Propositions \ref{pro:bddT2} and  \ref{pro:expT3}, $A_1$ has the following expression:
\begin{proposition}
\label{thm:bdd1stline}
We have
\begin{equation}
\label{eq:bdd1stline}
\begin{split}
\sum_{n=1}^{N} \tau &\sum_{i,j=1}^M \sum_{\alpha =w,o}\Big[f_\alpha(S^{n,i}) c_{ij}f_\alpha(S_\alpha^{n,ij})F^{n,ij}\Big]\\ 
& = \sum_{n=1}^{N} \tau\sum_{i=1}^M \sum_{j\in\mathcal{N}_{\mathcal F}(i)} c_{ij} \big(f_w(S^{n,j})-f_w(S^{n,i})\big)^2 F^{n,ij} + R_4,
\end{split}
\end{equation}
where
\begin{equation}
\label{es:bddR3}
|R_4| \le 6  \, \|\bar q\|_{L^1(\Om \times ]0,T[)}.
\end{equation}
\end{proposition}
This settles the contribution of the first term of \eqref{eq:f.eta.gradpflux}; the second terms are handled in the next subsection.

%-------------------------
\subsection{Terms involving the capillary pressure and mobility}
\label{subsec:bddmobility}

These are the terms $A_{\alpha,i,n}$ defined  in \eqref{eq:Awn}.
By virtue of the anti-symmetry of $C^{ij}$, we can write for $\alpha =w,o$
\begin{equation}
\label{eq:Cij1}
\sum_{i,j=1}^M f_\alpha(S^{i}) c_{ij} f_\alpha(S_\alpha^{ij})\,C_\alpha^{ij} =  -\frac{1}{2}
\sum_{i,j=1}^M \big(f_\alpha(S^{j})-f_\alpha(S^{i})\big) c_{ij} f_\alpha(S_\alpha^{ij})\,C_\alpha^{ij}.
\end{equation}
Owing to \eqref{eq:fluidfract}, the term with $\alpha =o$ in the right-hand side is
$\frac{1}{2}
\sum_{i,j=1}^M \big(f_w(S^{j})-f_w(S^{i})\big) c_{ij} f_o(S_o^{ij})\,C_o^{ij}$.
Therefore, 
\begin{equation}
\label{eq:A41}
\sum_{\alpha =w,o}\sum_{i=1}^M A_{\alpha,i,n} = \frac{1}{2} \sum_{i,j=1}^M   c_{ij}\big(f_w(S^{n,j})-f_w(S^{n,i})\big)\big(
-f_w(S_w^{n,ij})\,C_w^{n,ij} + f_o(S_o^{n,ij})\,C_o^{n,ij}\big).
\end{equation}
Let $K^{ij}$ denote the symmetric term
$$K^{ij} := c_{ij}\big(f_w(S^{n,j})-f_w(S^{n,i})\big)\big(
-f_w(S_w^{n,ij})\,C_w^{n,ij} + f_o(S_o^{n,ij})\,C_o^{n,ij}\big);
$$
 by virtue of this symmetry, we have
\begin{equation}
\label{eq:A42}
\sum_{\alpha =w,o}\sum_{i=1}^M A_{\alpha,i,n} = \sum_{ P_w^{n,i} > P_w^{n,j}} K^{n,ij} +  \frac{1}{2}\sum_{ P_w^{n,i} = P_w^{n,j}} K^{n,ij}.
\end{equation}

%-------------------------
\subsection{Combining all terms}
\label{subsec:allterms}

By substituting \eqref{eq:bdd1stline} and \eqref{eq:A42} into \eqref{eq:f.eta.gradpflux}, we obtain the next lemma.

\begin{lemma}
\label{lem:comb}
We have
\begin{equation}
\label{eq:allterms}
\begin{split}
-\sum_{n=1}^N& \tau \sum_{\alpha = o,w}\big[P_{\alpha,h}^n,\eta_\alpha(S_{\alpha,h}^n);f_\alpha(S_h^n), P_{\alpha,h}^n\big]_h  =    
\sum_{n=1}^{N} \tau \Big[\sum_{i=1}^M \Big (\sum_{j\in \mathcal{N}_{\mathcal F}(i)} c_{ij} \big(f_w(S^{n,j})-f_w(S^{n,i})\big)^2 F^{n,ij}\\
& -\sum_{ P_w^{n,i} > P_w^{n,j}} c_{ij}\big(f_w(S^{n,j})-f_w(S^{n,i})\big)\big(
f_w(S_w^{n,ij})\,C_w^{n,ij} - f_o(S_o^{n,ij})\,C_o^{n,ij}\big)\\
& - \frac{1}{2}\sum_{ P_w^{n,i} = P_w^{n,j}} c_{ij}\big(f_w(S^{n,j})-f_w(S^{n,i})\big)\big(
f_w(S_w^{n,ij})\,C_w^{n,ij} - f_o(S_o^{n,ij})\,C_o^{n,ij}\big)\Big)\Big]+R_4.
\end{split}
\end{equation}
with $R_4$ bounded by \eqref{es:bddR3}. 
\end{lemma}

Thus, to bring forward $g$, we must suitably combine the terms of the above sum over $i$,
and this is done by examining all pairs of indices $(i,j)$ involved in \eqref{eq:allterms}, i.e., the pairs of indices in the following sets: (i) $P_w^i > P_w^j$ and $P_o^i > P_o^j$, (ii) $P_w^i > P_w^j$ and  $P_o^i < P_o^j$, (iii) $P_w^i > P_w^j$ and  $P_o^i = P_o^j$, (iv) $P_w^i = P_w^j$ and $P_o^i > P_o^j$, (v) $P_w^i = P_w^j$ and $P_o^i < P_o^j$. Note that the sixth case that would be $P_w^i = P_w^j$ and $P_o^i = P_o^j$ brings no information because it implies that $S^i = S^j$.   

For the argument below, we shall use the following intermediate result.

\begin{proposition}
\label{pro:gcapSi--Sj}
For each indices $i$ and $j$, there exist (non unique) points $\alpha$ and $\alpha^\prime$ between $S^i$ and $S^j$ such that
\begin{equation}
\label{eq:gcapSi-Sj1}
g(S^j) - g(S^i) = - \eta_o(\alpha) f_w(\alpha) \big(p_c(S^{j})-p_c(S^{i})\big) = - \eta_w(\alpha^\prime) f_o(\alpha^\prime) \big(p_c(S^{j})-p_c(S^{i})\big).
\end{equation}
\end{proposition}

\begin{proof}
By the definition \eqref{eq:g},
\begin{equation}
\label{eq:deltag}
g(S^j) - g(S^i) = - \int_{S^i}^{S^j} \eta_o(x) f_w(x) p^\prime_c(x)\,dx
 = - \int_{S^i}^{S^j} \eta_w(x) f_o(x) p^\prime_c(x)\,dx.
\end{equation}
Since the functions $\eta_o f_w$ and $\eta_w f_o$ are continuous and do not change sign between $S^i$ and $S^j$, \eqref{eq:gcapSi-Sj1} follows from the second mean formula for integrals.
\end{proof}

%--------------------

To simplify, the superscript $n$ is dropped.

\subsubsection{The case $P_w^i > P_w^j$ and $P_o^i > P_o^j$} The following holds:

\begin{proposition}
\label{pro:cap.Si<Sj}
Let $P_w^i> P_w^j$ and $P_o^i> P_o^j$; then the factor of $\tau$ in \eqref{eq:allterms} satisfies
\begin{equation}
\label{eq:cap.Si<Sj1}
\begin{split}
c_{ij} \big(f_w(S^{j})-f_w(S^{i})\big)&\Big(\big(f_w(S^{j})-f_w(S^{i})\big) F^{ij} - \big(f_w(S^i) C_w^{ij} - f_o(S^i) C_o^{ij}\big)\Big)\\
& \ge c_{ij} \big(f_w(S^{j})-f_w(S^{i})\big) \big(g(S^{j})-g(S^{i})\big).
\end{split}
\end{equation}
\end{proposition}

\begin{proof}
Let $E^{ij}$ denote the left-hand side of \eqref{eq:cap.Si<Sj1}.
In the case $P_w^i> P_w^j$ and $P_o^i> P_o^j$, an expansion of $F^{ij}$,  and $C_\alpha^{ij}$ yields
\begin{equation}
\nonumber
\begin{split}
 E^{ij} = & = -c_{ij}\big(f_w(S^{j})-f_w(S^{i})\big)^2\big(\eta_w(S^i)  + \eta_o(S^i)\big)\big(P_w^j  - P_w^i\big)\\
 & -c_{ij}\big(f_w(S^{j})-f_w(S^{i})\big)\big(p_c(S^{j})-p_c(S^{i})\big)\Big( \eta_o(S^i)\big(f_w(S^{j})-f_w(S^{i})\big) + 2 \frac{\eta_w(S^i)\eta_o(S^i)} {\eta_w(S^i)  + \eta_o(S^i)}\Big).
 \end{split}
\end{equation} 
As $P_w^i> P_w^j$, the first line in the above right-hand side is nonnegative and hence
$$
E^{ij} \ge - c_{ij} \big(f_w(S^{j})-f_w(S^{i})\big)\big(p_c(S^{j})-p_c(S^{i})\big) \eta_o(S^i) \big(f_w(S^{j})+f_w(S^{i})\big).  
$$
Now, either $S^i \le S^j$ or $S^i > S^j$. If $S^i \le S^j$,
then $f_w(S^{j}) \ge f_w(S^{i})$ and  $p_c(S^j) \le p_c(S^i)$ because $f_w$ is increasing and $p_c$ is decreasing. This implies in particular that
\begin{equation}
\nonumber
\begin{split}
E^{ij} \ge& - c_{ij} \big(f_w(S^{j})-f_w(S^{i})\big)\big(p_c(S^{j})-p_c(S^{i})\big) \eta_o(S^i) f_w(S^{j})\\ 
&\ge - c_{ij} \big(f_w(S^{j})-f_w(S^{i})\big)\big(p_c(S^{j})-p_c(S^{i})\big) \eta_o(x) f_w(x), 
\end{split}
\end{equation}
for any $x$ between $S^i$ and $S^j$. Then \eqref{eq:cap.Si<Sj1} follows from the first part 
of \eqref{eq:gcapSi-Sj1}.

If $S^i > S^j$, then $f_w(S^{j})-f_w(S^{i}) \le 0$, $p_c(S^j)-p_c(S^i) \ge 0$, and
we infer from \eqref{eq:cap.pressure} that $E^{ij}$ reads
\begin{equation}
\nonumber
\begin{split}
E^{ij} =& c_{ij} \big(f_w(S^{j})-f_w(S^{i})\big)\Big[\big(\eta_w(S^i)  + \eta_o(S^i)\big)\big(f_w(S^{j})-f_w(S^{i})\big)\big(P_o^{i}-P_o^{j}\big)\\
&+ \big(f_w(S^{j})-f_w(S^{i})\big)\big(\eta_w(S^i)  + \eta_o(S^i)\big)\big(p_c(S^{j})-p_c(S^{i})\big)\\
&+ \big(p_c(S^{i})-p_c(S^{j})\big)
\Big(\eta_o(S^i)\big(f_w(S^{j})-f_w(S^{i})\big)  + 2\frac{\eta_w(S^i)\eta_o(S^i)} {\eta_w(S^i)  + \eta_o(S^i)}\Big) \Big]. 
 \end{split}
\end{equation}
Since $P_o^{i}-P_o^{j}>0$, the first line in the above right-hand side is nonnegative, and thus
\begin{equation}
\nonumber
\begin{split}
E^{ij} \ge& c_{ij} \big(f_w(S^{i})-f_w(S^{j})\big)\big(p_c(S^{j})-p_c(S^{i})\big)\Big[\big(f_w(S^{i})-f_w(S^{j})\big)\big(\eta_w(S^i)  + \eta_o(S^i)\big)\\
&+\big(f_w(S^{j})-f_w(S^{i})\big)\eta_o(S^i) + 2\frac{\eta_w(S^i)\eta_o(S^i)} {\eta_w(S^i)  + \eta_o(S^i)}\Big],
 \end{split}
\end{equation}
which reduces to
$$E^{ij} \ge c_{ij} \big(f_w(S^{i})-f_w(S^{j})\big)\big(p_c(S^{j})-p_c(S^{i})\big)\eta_w(S^i)\big(f_o(S^j) + f_o(S^i)\big).
$$
This leads for instance to
\begin{equation}
\nonumber
\begin{split}
E^{ij} \ge& -c_{ij} \big(f_w(S^{j})-f_w(S^{i})\big)\big(p_c(S^{j})-p_c(S^{i})\big)\eta_w(S^i)f_o(S^j)\\
& \ge -c_{ij} \big(f_w(S^{j})-f_w(S^{i})\big)\big(p_c(S^{j})-p_c(S^{i})\big) \eta_w(x) f_o(x),
 \end{split}
\end{equation}
for any $x$ between $S^i$ and $S^j$. Then  \eqref{eq:cap.Si<Sj1}
 follows from the second part of \eqref{eq:gcapSi-Sj1}.
\end{proof}

%-------------------------

\subsubsection{The case $P_w^i > P_w^j$ and $P_o^i < P_o^j$}

We have

\begin{proposition}
\label{pro:cap.Poi<Poj}
Let $P_w^i> P_w^j$ and $P_o^i< P_o^j$; then the factor of $\tau$ in \eqref{eq:allterms} satisfies
\begin{equation}
\label{eq:cap.Poi<Poj1}
-c_{ij} \big(f_w(S^{j})-f_w(S^{i})\big) \big(f_w(S^i) C_w^{ij} - f_o(S^j) C_o^{ij}\big)
 \ge c_{ij} \big(f_w(S^{j})-f_w(S^{i})\big) \big(g(S^{j})-g(S^{i})\big).
\end{equation}
\end{proposition}

\begin{proof}
Let $E^{ij}$ denote the left-hand side of \eqref{eq:cap.Poi<Poj1}. We have $P_w^j-P_w^i <0$ and  $P_o^j-P_o^i >0$. Then $S_w^{ij} =S^i$ and $S_o^{ij} =S^j$; also
$S^j \le S^i$ which implies that $\eta_w(S^j) \le \eta_w(S^i)$ and $\eta_o(S^j) \ge \eta_o(S^i)$.
The expression for $C_\alpha^{ij}$ becomes (see \eqref{eq:Cij} and \eqref{eq:Dij})
$$C_w^{ij}= \eta_o(S^{j})\big(p_c(S^{j})-p_c(S^{i})\big)- \big(\eta_o(S^{i})-\eta_o(S^{j})\big)(P_w^{j}-P_w^{i}),
$$
$$
C_o^{ij}= -\eta_w(S^{i})\big(p_c(S^{j})-p_c(S^{i})\big)- \big(\eta_w(S^{j})-\eta_w(S^{i})\big)(P_o^{j}-P_o^{i}).
$$
Hence
\begin{equation}
\nonumber
\begin{split}
E^{ij} = &  - c_{ij} \big(f_w(S^{j})-f_w(S^{i})\big)\Big(f_w(S^{i}) \big(\eta_o(S^{j})\big(p_c(S^{j})-p_c(S^{i})\big)
+ \big(\eta_o(S^{j})-\eta_o(S^{i})\big)(P_w^{j}-P_w^{i})\big)\\
& + f_o(S^{j})\big(\eta_w(S^{i})\big(p_c(S^{j})-p_c(S^{i})\big)+ \big(\eta_w(S^{j})-\eta_w(S^{i})\big)(P_o^{j}-P_o^{i}))\Big)\\
&= - c_{ij} \big(f_w(S^{j})-f_w(S^{i})\big)\Big((P_o^{j}-P_o^{i}) \big(f_w(S^{i})\big(\eta_o(S^{j})-\eta_o(S^{i})\big) + f_o(S^{j})\big(\eta_w(S^{j})-\eta_w(S^{i})\big)\big)\\
&+ \big(p_c(S^{j})-p_c(S^{i})\big)\big(f_w(S^{i})\eta_o(S^i) + f_o(S^{j})\eta_w(S^i)\big)\Big)\\
&= - c_{ij} \big(f_w(S^{j})-f_w(S^{i})\big)\Big((P_o^{j}-P_o^{i}) \big(f_w(S^{i})\big(\eta_o(S^{j})-\eta_o(S^{i})\big) + f_o(S^{j})\big(\eta_w(S^{j})-\eta_w(S^{i})\big)\big)\\
&+ \big(p_c(S^{j})-p_c(S^{i})\big)\eta_w(S^i)\big(f_o(S^{i}) + f_o(S^{j})\big)\Big).
\end{split}
\end{equation}
It follows from the above considerations that
\begin{equation}
\nonumber
\begin{split}
- c_{ij} \big(f_w(S^{j})-f_w(S^{i})&\big)\big(p_c(S^{j})-p_c(S^{i})\big)\eta_w(S^i)\big(f_o(S^{i}) + f_o(S^{j})\big)\\
 &\ge - c_{ij} \big(f_w(S^{j})-f_w(S^{i})\big)\big(p_c(S^{j})-p_c(S^{i})\big)\eta_w(S^i)f_o(S^{j})\big)\\
 &\ge - c_{ij} \big(f_w(S^{j})-f_w(S^{i})\big)\big(p_c(S^{j})-p_c(S^{i})\big)\eta_w(x)f_o(x)\big), 
\end{split}
\end{equation}
for any $x$ between $S^i$ and $S^j$. Now, the sign of the factor
$$f_w(S^{i})\big(\eta_o(S^{j})-\eta_o(S^{i})\big) + f_o(S^{j})\big(\eta_w(S^{j})-\eta_w(S^{i})\big)$$
is not clear. If it is nonnegative, then the whole term
$$- c_{ij} \big(f_w(S^{j})-f_w(S^{i})\big)(P_o^{j}-P_o^{i}) \big(f_w(S^{i})\big(\eta_o(S^{j})-\eta_o(S^{i})\big) + f_o(S^{j})\big(\eta_w(S^{j})-\eta_w(S^{i})\big)\big)$$
is also nonnegative,  
$$E^{ij} \ge - c_{ij} \big(f_w(S^{j})-f_w(S^{i})\big)\big(p_c(S^{j})-p_c(S^{i})\big)\eta_w(x)f_o(x)\big).
$$
and \eqref{eq:cap.Poi<Poj1} follows from \eqref{eq:gcapSi-Sj1}. If 
$$f_w(S^{i})\big(\eta_o(S^{j})-\eta_o(S^{i})\big) + f_o(S^{j})\big(\eta_w(S^{j})-\eta_w(S^{i})\big) <0,$$
then we rewrite $E^{ij}$ in terms of $P_w$,
\begin{equation}
\nonumber
\begin{split}
E^{ij} =  - c_{ij} \big(f_w(S^{j})-f_w(S^{i})\big)&\Big[(P_w^j-P_w^i)\big(f_w(S^{i})\big(\eta_o(S^{j})-\eta_o(S^{i})\big) + f_o(S^{j})\big(\eta_w(S^{j})-\eta_w(S^{i})\big)\\
& + \big(p_c(S^{j})-p_c(S^{i})\big)\big(f_w(S^{j})\eta_o(S^j) + \eta_w(S^{j})f_o(S^{j})\big)\Big].
\end{split}
\end{equation}
Since the first line is now nonnegative, we infer
\begin{align*}
E^{ij} \ge&  - c_{ij} \big(f_w(S^{j})-f_w(S^{i})\big)\big(p_c(S^{j})-p_c(S^{i})\big) \eta_o(S^j)\big(f_w(S^{j}) + f_w(S^{i})\big)\\
& \ge - c_{ij} \big(f_w(S^{j})-f_w(S^{i})\big)\big(p_c(S^{j})-p_c(S^{i})\big)\eta_o(x)f_w(x),
\end{align*}
again for any $x$ between $S^i$ and $S^j$, and the result follows from  \eqref{eq:gcapSi-Sj1}.
\end{proof}

%-------------------------

\subsubsection{The case $P_w^i > P_w^j$ and $P_o^i = P_o^j$}

In this case, $p_c(S^j)- p_c(S^i) >0$, $S^j \le S^i$, and we have the following result:

\begin{proposition}
\label{pro:cap.Poi=Poj}
Let $P_w^i> P_w^j$ and $P_o^i= P_o^j$; then the factor of $\tau$ in \eqref{eq:allterms} satisfies
\begin{equation}
\label{eq:cap.Poi=Poj1}
-c_{ij} \big(f_w(S^{j})-f_w(S^{i})\big) \big(f_w(S^i) C_w^{ij} - f_o(S^j) C_o^{ij}\big)
 \ge c_{ij} \big(f_w(S^{j})-f_w(S^{i})\big) \big(g(S^{j})-g(S^{i})\big).
\end{equation}
\end{proposition}

\begin{proof}
Let $E^{ij}$ denote the left-hand side of \eqref{eq:cap.Poi=Poj1}. We have
$C_w^{ij} = \eta_o(S^i)\big(p_c(S^j)- p_c(S^i)\big)$ and $C_o^{ij} = -\eta_w(S^i)\big(p_c(S^j)- p_c(S^i)\big)$. Hence
\begin{align*}
E^{ij} = -c_{ij} \big(f_w(S^{j})-f_w(S^{i})\big)\big(f_w(S^i)\eta_o(S^i) + f_o(S^j)\eta_w(S^i)\big)\big(p_c(S^j)- p_c(S^i)\big)\\
 =  -c_{ij} \big(f_w(S^{j})-f_w(S^{i})\big)\big(p_c(S^j)- p_c(S^i)\big)\big(\eta_o(S^i)f_w(S^i) + \eta_w(S^i)f_o(S^j)\big).
\end{align*}
From here,  \eqref{eq:cap.Poi=Poj1} is derived  as in the end of the proof of Proposition \ref{pro:cap.Poi<Poj}.
\end{proof} 

%--------------------

\subsubsection{The case $P_w^i = P_w^j$ and $P_o^i > P_o^j$}

In this case, $p_c(S^j)\le p_c(S^i)$ and $S^j \ge S^i$. We have the following result:

\begin{proposition}
\label{pro:cap.Pwi=Pwj.Poi>Poj}
Let $P_w^i= P_w^j$ and $P_o^i> P_o^j$; then the factor of $\tau$ in \eqref{eq:allterms} satisfies
\begin{equation}
\label{eq:cap.Pwi=Pwj.Poi>Poj1}
-\frac{1}{2}c_{ij} \big(f_w(S^{j})-f_w(S^{i})\big) \big(f_w(S^j) C_w^{ij} - f_o(S^i) C_o^{ij}\big)
 \ge \frac{1}{2}c_{ij} \big(f_w(S^{j})-f_w(S^{i})\big) \big(g(S^{j})-g(S^{i})\big).
\end{equation}
\end{proposition}

\begin{proof}
In this case,  $C_w^{ij} = \eta_o(S^i)\big(p_c(S^j)- p_c(S^i)\big)$ and $C_o^{ij} = -\eta_w(S^i)\big(p_c(S^j)- p_c(S^i)\big)$.
Then the left-hand side $E^{ij}$  of \eqref{eq:cap.Pwi=Pwj.Poi>Poj1} is
\begin{align*} 
E^{ij} =& -\frac{1}{2}c_{ij} \big(f_w(S^{j})-f_w(S^{i})\big)\big(p_c(S^j)- p_c(S^i)\big) \big(f_w(S^j) \eta_o(S^i) + f_o(S^i) \eta_w(S^i)\big)\\
& \ge -\frac{1}{2}c_{ij} \big(f_w(S^{j})-f_w(S^{i})\big)\big(p_c(S^j)- p_c(S^i)\big) f_w(S^j) \eta_o(S^i),
\end{align*}
and the proof of \eqref{eq:cap.Pwi=Pwj.Poi>Poj1} proceeds as above.
\end{proof}
%--------------------

\subsubsection{The case $P_w^i = P_w^j$ and $P_o^i < P_o^j$}

In this case, $ p_c(S^j)\ge p_c(S^i)$ and hence  $S^i \ge S^j$. We have the following result:

\begin{proposition}
\label{pro:cap.Pwi=Pwj.Poi<Poj}
Let $P_w^i= P_w^j$ and $P_o^i< P_o^j$; then the factor of $\tau$ in \eqref{eq:allterms} satisfies
\begin{equation}
\label{eq:cap.Pwi=Pwj.Poi<Poj1}
-\frac{1}{2}c_{ij} \big(f_w(S^{j})-f_w(S^{i})\big) \big(f_w(S^i) C_w^{ij} - f_o(S^j) C_o^{ij}\big)\\
\ge \frac{1}{2}c_{ij} \big(f_w(S^{j})-f_w(S^{i})\big) \big(g(S^{j})-g(S^{i})\big).
\end{equation}
\end{proposition}

\begin{proof}
In this case, $C_w^{ij} = \eta_o(S^j)\big(p_c(S^j)- p_c(S^i)\big)$ and $C_o^{ij} = -\eta_w(S^j)\big(p_c(S^j)- p_c(S^i)\big)$.
Then the left-hand side $E^{ij}$  of \eqref{eq:cap.Pwi=Pwj.Poi>Poj1} is
\begin{align*} 
E^{ij} =& -\frac{1}{2}c_{ij} \big(f_w(S^{j})-f_w(S^{i})\big)\big(p_c(S^j)- p_c(S^i)\big) \big(f_w(S^i
) \eta_o(S^j) + f_o(S^j) \eta_w(S^j)\big)\\
& \ge -\frac{1}{2}c_{ij} \big(f_w(S^{j})-f_w(S^{i})\big)\big(p_c(S^j)- p_c(S^i)\big) \eta_o(S^j)f_w(S^i),
\end{align*}
and the proof of \eqref{eq:cap.Pwi=Pwj.Poi<Poj1} ends as above.
\end{proof}

%-------------------------
\subsection{Auxiliary bound for the gradient of $g$}
\label{subsec:firstBddgrad.g}

The following theorem is the first outcome of this section.

\begin{theorem}
\label{thm:grad.g}
There exists a constant $C$, independent of $h$ and $\tau$, such that
\begin{equation}
\label{eq:Bddgrad.g1}
\Big|\int_0^T \int_\Omega \nabla (I_h(f_\alpha(S_{h,\tau})))\cdot \nabla(I_h(g(S_{h,\tau})))\Big| \le C,\quad \alpha = w,o.
\end{equation}
\end{theorem}

\begin{proof}
 Owing to \eqref{eq:fluidfract}, it suffices to prove  \eqref{eq:Bddgrad.g1} when $\alpha =w$. \Bk By applying Propositions \ref{pro:cap.Si<Sj}--\ref{pro:cap.Pwi=Pwj.Poi<Poj} to Lemma \ref{lem:comb} and combining with Proposition \ref{pro:f.eta.gradp}, we readily derive that
\begin{equation}
\label{eq:Bddgrad.g2}
\sum_{n=1}^N \tau \sum_{i=1}^M \sum_{j \in \mathcal{N}(i), P_w^{\Rd n, \Bk i} \ge P_w^{\Rd n, \Bk j} } c_{ij} \big(f_w(S^{\Rd n, \Bk j})-f_w(S^{\Rd n, \Bk i})\big) \big(g(S^{\Rd n, \Bk j})-g(S^{\Rd n, \Bk i})\big) \le C,
\end{equation}
with a constant $C$ independent of $h$ and $\tau$. Therefore, \eqref{eq:Bddgrad.g1} will follow if we bound the summand for all $j$ such that $P_w^{\Rd n, \Bk i} < P_w^{\Rd n, \Bk j}$. But the symmetry of the summand implies that
\begin{align*}
\sum_{i=1}^M \sum_{j \in \mathcal{N}(i), P_w^{\Rd n, \Bk i} < P_w^{\Rd n, \Bk j} }& c_{ij} \big(f_w(S^{\Rd n, \Bk j})-f_w(S^{\Rd n, \Bk i})\big) \big(g(S^{\Rd n, \Bk j})-g(S^{\Rd n, \Bk i})\big)\\
& = \sum_{i=1}^M \sum_{j \in \mathcal{N}(i), P_w^{ \Rd n, \Bk i} > P_w^{ \Rd n, \Bk j} } c_{ij} \big(f_w(S^{\Rd n, \Bk j})-f_w(S^{\Rd n, \Bk i})\big) \big(g(S^{\Rd n, \Bk j})-g(S^{\Rd n, \Bk i})\big).
\end{align*}
Hence
\begin{align*}
\int_\Omega \nabla (I_h(f_w(S_{h}^n)))\cdot &\nabla(I_h(g(S_{h}^n)))  
=  2\sum_{i=1}^M \sum_{j \in \mathcal{N}(i), P_w^{ \Rd n, \Bk i} > P_w^{\Rd n, \Bk j} } c_{ij} \big(f_w(S^{\Rd n, \Bk j})-f_w(S^{\Rd n, \Bk i})\big) \big(g(S^{\Rd n, \Bk j})-g(S^{\Rd n, \Bk i})\big)\\
& + \sum_{i=1}^M \sum_{j \in \mathcal{N}(i), P_w^{\Rd n, \Bk i} = P_w^{\Rd n, \Bk j} } c_{ij} \big(f_w(S^{\Rd n, \Bk j})-f_w(S^{\Rd n, \Bk i})\big) \big(g(S^{\Rd n, \Bk j})-g(S^{\Rd n, \Bk i})\big),
\end{align*}
and \eqref{eq:Bddgrad.g1}, with another constant $C$, follows by substituting this equality into \eqref{eq:Bddgrad.g2}.
\end{proof}

%-------------------------
\subsection{Bound for the gradient of $g$}
\label{subsec:Bddgrad.g}

In order to deduce from \eqref{eq:Bddgrad.g1} a direct bound for the gradient of $g$, we need to sharpen the assumptions on the mobility.

%---------------------

\subsubsection{Assumptions on the mobility}

Here we assume that
$$ \eta_w(0) = \eta_o(1) =0,$$
and $\eta_\alpha \in W^{1,\infty}(0,1)$, $\alpha = w,o$. Furthermore, we assume that, for all $x \in ]0,1[$, 
\begin{equation}
\label{eq:eta'w}
 \alpha_w x^{\theta_w-1} \le \eta^\prime_w(x) \le \frac{1}{\alpha_w}x^{\theta_w-1},\quad \theta_w \ge 1,0 <\alpha_w \le 1,
\end{equation}
\begin{equation}
\label{eq:eta'o}
 \alpha_o (1-x)^{\theta_o-1} \le -\eta^\prime_o(x) \le \frac{1}{\alpha_o}(1-x)^{\theta_o-1},\quad \theta_o \ge 1, 0< \alpha_o \le 1,
\end{equation}
\begin{equation}
\label{eq:pc'}
 \frac{1}{\alpha_3}x^{\beta_3-1} (1-x)^{\beta_4-1} \ge -p_c^\prime(x) \ge \alpha_3
x^{\beta_3-1} (1-x)^{\beta_4-1},\quad  \beta_3, \beta_4>0, 0 < \alpha_3 \le 1.
\end{equation}
 From \eqref{eq:eta'w} and \eqref{eq:eta'o}, we deduce respectively, for all $x \in ]0,1[$,
\begin{equation}
\label{eq:etaw_1}
 \frac{\alpha_w}{\theta_w} x^{\theta_w} \le \eta_w(x) \le \frac{1}{\alpha_w\theta_w} x^{\theta_w},
\end{equation}
\begin{equation}
\label{eq:etao_1}
 \frac{\alpha_o}{\theta_o} (1-x)^{\theta_o} \le \eta_o(x) \le \frac{1}{\alpha_o\theta_o}(1-x)^{\theta_o}.
\end{equation}
The sum of these two inequalities reads for all $x \in ]0,1[$,
$$
\frac{\alpha_w}{\theta_w} x^{\theta_w} + \frac{\alpha_o}{\theta_o} (1-x)^{\theta_o} \le \eta_w(x) + \eta_o(x) \le \frac{1}{\alpha_w\theta_w} x^{\theta_w} + \frac{1}{\alpha_o\theta_o}(1-x)^{\theta_o}.
$$
Let $\ell$ denote the lower bound in this inequality. It is easy to check that $\ell$ is a nonnegative continuous function of $x$ on $[0,1]$, hence uniformly continuous. Therefore, it is bounded and as it does not vanish in this interval, it is bounded away from zero. Thus there exists a positive constant $C_{\rm min}$ such that 
\begin{equation}
\label{eq:bddsh}
\forall x \in [0,1],\quad C_{\rm min} \le \ell(x):  = \frac{\alpha_w}{\theta_w} x^{\theta_w} + \frac{\alpha_o}{\theta_o} (1-x)^{\theta_o} \le C_{\rm max},
\end{equation}
where
\begin{equation}
\label{eq:Cmx}
C_{\rm max} = {\rm max}_{x \in [0,1]} \Big(\frac{1}{\alpha_w\theta_w} x^{\theta_w} + \frac{1}{\alpha_o\theta_o}(1-x)^{\theta_o}\Big).
\end{equation}

%---------------------

\subsubsection{Properties of the derivatives of $f_w$ and $g$}

By definition, we have
$$g^\prime(x) = -\frac{\eta_w(x)\eta_o(x)}{\eta_w(x) + \eta_o(x)}p_c^\prime(x),
$$
which is positive in $]0,1[$. 
Considering \eqref{eq:lowerboundetas}, \eqref{eq:etaw_1}, and \eqref{eq:etao_1}, we infer
\begin{equation}
\label{eq:upbddgprime}
g^\prime(x) \le \frac{1}{\eta_\ast \alpha_3}\frac{1}{\alpha_w \theta_w}\frac{1}{\alpha_o \theta_o} x^{\theta_w -1 + \beta_3}(1-x)^{\theta_o -1 + \beta_4},
\end{equation}
thus implying that $g^\prime$ is a bounded function, i.e., $g$ is Lipschitz continuous. Note that the Lipschitz constant $L$ of $g$ is bounded by
\begin{equation}
\label{eq:Lipsconst}
L \le \frac{1}{\alpha \alpha_3}\frac{1}{\alpha_w \theta_w}\frac{1}{\alpha_o \theta_o}{\rm max}_{x \in [0,1]}\big(x^{\theta_w-1 + \beta_3} (1-x)^{\theta_o-1 + \beta_4}\big).
\end{equation}
On the other hand, \eqref{eq:pc'}--\eqref{eq:bddsh} yield for all $x \in ]0,1[$,
\begin{equation}
\label{eq:lbgprime}
g^\prime(x)  \ge \frac{\alpha_3}{C_{\rm max}}\frac{\alpha_w}{\theta_w}  \frac{\alpha_o}{\theta_o} x^{\theta_w -1+\beta_3} (1-x)^{\theta_o -1+\beta_4} >0.
\end{equation}
Thus $g \in W^{1,\infty}(0,1)$ is a strictly monotonic increasing function on $[0,1]$ with range $[0,\beta]$ for some $\beta >0$, hence invertible with inverse $g^{-1}\in W^{1,\infty}(0,\beta)$. 

Now, we turn to $f_w$. By definition, we have
\begin{equation}
\label{eq:fwprime}
f^\prime_w(x) = \frac{1}{(\eta_w(x) + \eta_o(x))^2} \big(\eta_o(x)\eta^\prime_w(x) - \eta_w(x)\eta^\prime_o(x)\big).
\end{equation} 
The inequalities \eqref{eq:eta'w}--\eqref{eq:bddsh} imply that
$$ 
f^\prime_w(x) \ge \frac{1}{C_{\rm max}^2} \alpha_o \alpha_w\big[\frac{1}{\theta_o}x^{\theta_w-1}(1-x)^{\theta_o} + \frac{1}{\theta_w}x^{\theta_w}(1-x)^{\theta_o-1}\big].
$$
Thus,
\begin{equation}
\label{eq:fwprime1} 
\forall x \in [0,\frac{3}{4}],\quad f^\prime_w(x) \ge \frac{\alpha_o \alpha_w}{C_{\rm max}^2 \theta_o} \big(\frac{1}{4}\big)^{\theta_o} x^{\theta_w-1},
\end{equation}
and  
\begin{equation}
\label{eq:fwprime2}
\forall x \in [\frac{1}{4},1],\quad f^\prime_w(x) \ge \frac{\alpha_o \alpha_w}{C_{\rm max}^2 \theta_w} \big(\frac{1}{4}\big)^{\theta_w} (1-x)^{\theta_o-1}.
\end{equation}

Let us use these results to compare $g^\prime$ and $f_w^\prime$. It follows from \eqref{eq:upbddgprime} that
$$ 
\forall x \in [0,\frac{3}{4}],\quad g^\prime(x) \le \Big(\frac{1}{\alpha \alpha_3}\frac{1}{\alpha_w \theta_w}\frac{1}{\alpha_o \theta_o}  \frac{\theta_o C_{\rm max}^2}
{\alpha_o \alpha_w}\Big) \frac{\alpha_o \alpha_w} {C_{\rm max}^2 \theta_o}  x^{\theta_w-1},
$$
and by setting 
$$C_1 = \Big(\frac{1}{\alpha \alpha_3}\frac{1}{\alpha_w \theta_w}\frac{1}{\alpha_o \theta_o} 4^{\theta_o} \frac{\theta_o C_{\rm max}^2}{\alpha_o \alpha_w}\Big)
$$
and comparing with \eqref{eq:fwprime1}, we obtain
\begin{equation}
\label{eq:gprime1}
\forall x \in [0,\frac{3}{4}],\quad g^\prime(x) \le C_1 f^\prime_w(x).
\end{equation}
Similarly,
$$\forall x \in [\frac{1}{4},1],\quad g^\prime(x) \le \Big(\frac{1}{\alpha \alpha_3}\frac{1}{\alpha_w \theta_w}\frac{1}{\alpha_o \theta_o}  \frac{\theta_w C_{\rm max}^2}
{\alpha_o \alpha_w}\Big)\frac{\alpha_o \alpha_w}{C_{\rm max}^2 \theta_w}  (1-x)^{\theta_o-1},
$$
so that, by setting 
$$C_2 = \Big(\frac{1}{\alpha \alpha_3}\frac{1}{\alpha_w \theta_w}\frac{1}{\alpha_o \theta_o} 4^{\theta_w} \frac{\theta_w C_{\rm max}^2}{\alpha_o \alpha_w}\Big)
$$
and comparing with \eqref{eq:fwprime2}, we deduce 
\begin{equation}
\label{eq:gprime2}
\forall x \in [\frac{1}{4},1],\quad g^\prime(x) \le  C_2 f^\prime_w(x).
\end{equation}
This leads to the desired relation between the derivative of $f_w$ and $g$:
\begin{equation}
\label{eq:gprime}
\forall x \in [0,1],\quad g^\prime(x) \le C f^\prime_w(x),
\end{equation}
where $C = {\rm max}(C_1, C_2)$.

The main result of this section follows by combining \eqref{eq:gprime} with \eqref{eq:Bddgrad.g1}.

\begin{theorem}
\label{thm:grad.gfinal}
Under the assumptions \eqref{eq:eta'w}--\eqref{eq:pc'} on the derivatives of the mobilities and capillary pressure, there exists a constant $C$, independent of $h$ and $\tau$, such that
\begin{equation}
\label{eq:Bddgrad.final}
\|\nabla(I_h(g(S_{h,\tau})))\|_{L^2(\Omega \times ]0,T[)} \le C.
\end{equation}
\end{theorem}

\begin{proof}
Let $(i,j)$ be any pair of indices. If $S^{n,i} \le S^{n,j}$, then by \eqref{eq:gprime},
$$f_w(S^{n,j})-f_w(S^{n,i}) = \int_{S^{n,i}}^{S^{n,j}} f^\prime_w(x)\,dx \ge C\int_{S^{n,i}}^{S^{n,j}} g^\prime(x)\,dx = C(g(S^{n,j})-g(S^{n,i})).
$$
As $g$ is increasing, we have $g(S^{n,j})-g(S^{n,i}) \ge 0$. Therefore
\begin{equation}
\label{eq:deltafdeltag}
\big(f_w(S^{n,j})-f_w(S^{n,i})\big)\big(g(S^{n,j})-g(S^{n,i})\big) \ge C |g(S^{n,j})-g(S^{n,i})|^2.
\end{equation}
By changing both signs, the same result holds when $S^{n,j} < S^{n,i}$. Then \eqref{eq:Bddgrad.final} follows from \eqref{eq:Bddgrad.g1}.
\end{proof}

%==================================

\section{Convergence}
\label{sec:Convg}
%--------------------

The interpolants of $p_{\alpha g}(S_{h,\tau})$, $g(S_{h,\tau})$, and $p_c(S_{h,\tau})$ play an important part in this work, see Theorems \ref{thm:globalpressurebd} and \ref{thm:grad.gfinal}, and 
\eqref{eq:vars3}. Therefore, we begin by studying convergence properties first of $I_h(g(S_{h,\tau}))$ and  $I_h(p_{  \alpha \Bk g}(S_{h,\tau}))$,  $\alpha = w,o$, and next $I_h(p_c(S_{h,\tau}))$. \Bk Some results will stem from an interesting relation between differences in values of $S_{h,\tau}$ and $g(S_{h,\tau})$.

\subsection{Properties of  $I_h(g(S_{h,\tau}))$ and $I_h(p_{ \alpha \Bk g}(S_{h,\tau}))$,  $\alpha = w,o$ \Bk}
\label{sec:gIhg}

%------------------------------

\subsubsection{Convergence properties of  $I_h(g(S_{h,\tau}))$} 

Let  $K$ be  an element of ${\mathcal T}_h$ with vertices $\Aa_i$, $1 \le i \le d+1$ (local numbers); then
$$\int_K |g(S_{h,\tau})(t_n)|^2 = \int_K |g(\sum_{i=1}^{d+1}S^{n,i} \phi_i)|^2. $$
As $0 \le S^{n,i}, \phi_i \le 1$ and  $g$ is increasing, we have 
$$0 \le \sum_{i=1}^{d+1}S^{n,i} \phi_i \le \sum_{i=1}^{d+1}S^{n,i},\quad g (\sum_{i=1}^{d+1}S^{n,i} \phi_i) \le \sum_{i=1}^{d+1}  g(S^{n,i}).
$$
Hence
$$\int_K |g(S_{h,\tau})(t_n)|^2  \le (d+1)|K| \sum_{i=1}^{d+1}  (g(S^{n,i}))^2 .
$$
As a consequence, there exist constants $C$, $D$, $E$, independent of $n$, $h$ and $\tau$, such that
\begin{equation}
\label{eq:boundg(Sh)}
\|g(S_{h,\tau})(t_n)\|_{L^2(\Omega)} \le C \|I_h(g(S_{h,\tau})(t_n))\|_h \le D \|I_h(g(S_{h,\tau})(t_n))\|_h^{\varphi} \le E \|I_h(g(S_{h,\tau})(t_n))\|_{L^2(\Omega)},
\end{equation}
owing to \eqref{eq:equivnorm} and \eqref{eq:equivnormphi}. These inequalities carry over to the norm in $L^2(\Omega\times ]0,T[)$.

Now, let us prove the following convergence property of  $I_h(g(S_{h,\tau}))$.

\begin{lemma}
\label{lem:convIhgSh}
Under the assumptions of Theorem \ref{thm:grad.gfinal}, we have
\begin{equation}
\label{eq:convIhgSh1}
\lim_{(h,\tau) \to (0,0)} \|g(S_{h,\tau}) - I_h(g(S_{h,\tau}))\|_{L^2(\Om\times ]0,T[)} = 0.
\end{equation}
\end{lemma}

\begin{proof}
For any $\x$ in any element $K$ of ${\mathcal T}_h$, we have
$$I_h(g(S_{h,\tau}))(\x,t_n) -  g(S_{h,\tau})(\x,t_n) = \sum_{i=1}^{d+1} g(S^{n,i})\phi_i(\x) - g\big(\sum_{i=1}^{d+1} S^{n,i} \phi_i(\x)\big).
$$
As $S_{h,\tau}$ is a polynomial of degree one in $K$, it attains its maximum and its minimum in space at vertices of $K$, say $g(S^{n,\ell})$ and $g(S^{n,r})$ are its maximum and minimum respectively. Thus, recalling that $g$ is a nonnegative monotonically increasing function, 
$$\sum_{i=1}^{d+1} g(S^{n,i})\phi_i(\x) \le g(S^{n,\ell}),\quad g\big(\sum_{i=1}^{d+1} S^{n,i} \phi_i(\x)\big) \ge g(S^{n,r}).
$$
Hence
\begin{equation}
\label{eq:convIhgSh2}
\|I_h(g(S_{h,\tau})) -  g(S_{h,\tau})\|^2_{L^2(\Om\times ]0,T[)} \le \sum_{n=1}^N \tau \sum_{K \in {\mathcal T}_h} |K| |g(S^{n,\ell})- g(S^{n,r})|^2.
\end{equation}
For any node $i$, let 
$$\kappa_i = {\rm Max}\,|K|,$$
 where the maximum is taken over all elements $K$ in $\Delta_i$. Then we can readily check that
$$\sum_{n=1}^N \tau \sum_{K \in {\mathcal T}_h} |K| |g(S^{n,\ell})- g(S^{n,r})|^2 \le C
\sum_{n=1}^N \tau \sum_{i=1}^M \kappa_i \sum_{j \in \mathcal{N}(i)}  \big|g(S^{n,j})-g(S^{n,i})\big|^2,
$$
where $C$ is a bound for the maximum number of elements that share a common edge, bound independent of $h$ and $\tau$ by virtue of the regularity of the mesh. Now, recall the classical formula in each $d$-simplex $K$, 
\begin{equation}
\label{eq:formcij}
\int_K |\nabla\,\phi_i \cdot \nabla\, \phi_j |=  \frac{1}{d^2} \frac{|F_i| |F_j|}{|K|}|\n_i \cdot \n_j|,
\end{equation}
where $F_i$ is the face opposite to the vertex $\Aa_i$ and 
$\n_i$ is the exterior (to $K$) unit normal to the face $F_i$. The regularity of the mesh implies that there exists a constant $c_0$, independent of $h$ and $\tau$, such that
$$|\n_i \cdot \n_j| \ge c_0.$$
Hence, using again the regularity of the mesh, we obtain
$$\int_K |\nabla\,\phi_i \cdot \nabla\, \phi_j | \ge  C \, h_K^{d-2},
$$
and denoting by $\rho_{ij}$ the minimum of $h_K$ for all $K$ in $\Delta_i \cap \Delta_j$, we deduce
\begin{equation}
\label{eq:formcij1}
c_{ij} \ge  C \rho_{ij}^{d-2},
\end{equation}
with another constant $C$ independent of $h$ and $\tau$. By collecting  these results, we derive
\begin{equation}
\label{eq:convIhgSh30}
\|I_h(g(S_{h,\tau})) -  g(S_{h,\tau})\|^2_{L^2(\Om\times ]0,T[)} \le C \sum_{n=1}^N \tau \sum_{i=1}^M \kappa_i \sum_{j \in \mathcal{N}(i)} \big(\frac{1}{\rho_{ij}^{d-2}}\big) c_{ij}\big|g(S^{n,j})-g(S^{n,i})\big|^2.
\end{equation}
With another application of the regularity of the mesh, this becomes
\begin{equation}
\label{eq:convIhgSh3}
\|I_h(g(S_{h,\tau})) -  g(S_{h,\tau})\|^2_{L^2(\Om\times ]0,T[)} \le C h^2 \|\nabla(I_h(g(S_{h,\tau})))\|^2_{L^2(\Om\times ]0,T[)},
\end{equation}
(note that the power of $h$ is independent of the dimension) and the limit \eqref{eq:convIhgSh1} follows from Theorem \ref{thm:grad.gfinal}.
\end{proof}

%------------------------

\subsubsection{Relation between $g(S^{n,j})-g(S^{n,i})$ and $S^{n,j}-S^{n,i}$}

Here, we derive an upper bound for $S^{n,j}-S^{n,i}$ in terms of $g(S^{n,j})-g(S^{n,i})$.

\begin{lemma}
\label{lem:lowerbdg}
Under the assumptions \eqref{eq:eta'w}--\eqref{eq:pc'} on the derivatives of the mobilities and capillary pressure, there exists a constant $C$, independent of $h$ and $\tau$, such that for all $i$, $j$, and $n$
\begin{equation}
\label{eq:lowerbdg1}
|S^{n,j}-S^{n,i}| \le C\, |g(S^{n,j})-g(S^{n,i})|^{\frac{1}{r}},
\end{equation}
where $r = \rm{max}(\theta_o + \beta_4,\theta_w + \beta_3 ) >1$.
\end{lemma}

\begin{proof}
To simplify, we set $c = S^{n,i}$, $d = S^{n,j}$ and assume $c < d$. From  \eqref{eq:lbgprime}, it follows that
\begin{equation}
\label{eq:diffg1}
g(d)-g(c) \ge \frac{\alpha_3}{C_{\rm max}} \frac{\alpha_w}{\theta_w}\frac{\alpha_o}{\theta_o}\int_c^d x^{\theta_w + \beta_3 -1} (1-x)^{\theta_o + \beta_4 -1}.
\end{equation}  
For the sake of brevity, we do not specify the constant factor in \eqref{eq:diffg1} and write
$$
g(d)-g(c) \ge  C_1 \int_c^d x^{\theta_w + \beta_3 -1} (1-x)^{\theta_o + \beta_4 -1}.
$$
Now, we argue according to the positions of $c$ and $d$. There are four cases.

1)\ If $\frac{1}{8} \le c \le \frac{7}{8}$, then \eqref{eq:diffg1} gives
\begin{align*}
g(d)-g(c) &\ge  C_1 \big(\frac{1}{8}\big)^{\theta_w + \beta_3 -1}\int_c^d (1-x)^{\theta_o + \beta_4 -1}\\
&= C_1 \big(\frac{1}{8}\big)^{\theta_w + \beta_3 -1}\frac{1}{\theta_o + \beta_4}\Big((1-c)^{\theta_o + \beta_4} - (1-d)^{\theta_o + \beta_4}\Big).
\end{align*}
But
\begin{align*}
(1-c)^{\theta_o + \beta_4} - (1-d)^{\theta_o + \beta_4} &= (1-c)^{\theta_o + \beta_4-1}(d-c) + (1-d)\big((1-c)^{\theta_o + \beta_4-1} - (1-d)^{\theta_o + \beta_4-1}\big)\\
& \ge (d-c)\big(\frac{1}{8}\big)^{\theta_o + \beta_4 -1}.
\end{align*}
Hence
\begin{equation}
\label{eq:case1}
g(d)-g(c)\ge \frac{C_1}{\theta_o + \beta_4} \big(\frac{1}{8}\big)^{\theta_w + \beta_3 +\theta_o + \beta_4 -2} (d-c).
\end{equation}

2) If $c>\frac{7}{8}$, then $d>\frac{7}{8}$ and \eqref{eq:diffg1} gives 
$$g(d)-g(c) \ge  C_1 \big(\frac{7}{8}\big)^{\theta_w + \beta_3 -1}\frac{1}{\theta_o + \beta_4}\Big((1-c)^{\theta_o + \beta_4} - (1-d)^{\theta_o + \beta_4}\Big).
$$
Let us set $a = 1-d$, $b = d-c$ $\gamma = \theta_o + \beta_4 -1 >0$. We can also write
$$
(1-c)^{\theta_o + \beta_4} - (1-d)^{\theta_o + \beta_4} = a^{\gamma +1}\Big(\big(1 + \frac{b}{a}\big)^{\gamma +1}-1\Big).
$$
It is easy to check that the function
$$x \mapsto  (1 + x)^{\gamma +1}-1 -x^{\gamma +1} $$
vanishes at $x =0$ and is strictly monotonic increasing, hence is strictly positive for $x >0$. Hence 
$$a^{\gamma +1}\Big(\big(1 + \frac{b}{a}\big)^{\gamma +1}-1\Big) > a^{\gamma +1} \big(\frac{b}{a}\big)^{\gamma +1} = b^{\gamma +1}.
$$
Thus
$$
(1-c)^{\theta_o + \beta_4} - (1-d)^{\theta_o + \beta_4} \ge (d-c)^{\theta_o + \beta_4},
$$
and
\begin{equation}
\label{eq:case2} 
g(d)-g(c)\ge C_1 \big(\frac{7}{8}\big)^{\theta_w + \beta_3 -1}\frac{1}{\theta_o + \beta_4}(d-c)^{\theta_o + \beta_4}.
\end{equation}

3)\ If  $c <\frac{1}{8}$ and $d  <\frac{7}{8}$, then the integrand $1-x \ge 1-d >  \frac{1}{8}$ and by the above argument,
\begin{equation}
\label{eq:case3}
\begin{split} 
g(d)-g(c) &\ge  C_1 \big(\frac{1}{8}\big)^{\theta_o + \beta_4 -1}\frac{1}{\theta_w + \beta_3}\Big(d^{\theta_w + \beta_3} - c^{\theta_w + \beta_3}\Big)\\
& \ge  C_1 \big(\frac{1}{8}\big)^{\theta_o + \beta_4 -1}\frac{1}{\theta_w + \beta_3}(d-c)^{\theta_w + \beta_3} .
\end{split}
\end{equation}

4)\ If  $c <\frac{1}{8}$ and $d  >\frac{7}{8}$, then $c < \frac{1}{6}(d-c) < \frac{1}{2}(d-c) < d$. Therefore, we can write
\begin{equation}
\label{eq:case4}
\begin{split}
g(d)-g(c) &\ge  C_1 \int_{\frac{1}{6}(d-c)}^{\frac{1}{2}(d-c)}x^{\theta_w + \beta_3 -1} (1-x)^{\theta_o + \beta_4 -1}\\
&\ge  C_1 \big(\frac{1}{2}\big)^{\theta_o + \beta_4 -1}\frac{1}{\theta_w + \beta_3} \Big(\big(\frac{1}{2}(d-c)\big)^{\theta_w + \beta_3} - \big(\frac{1}{6}(d-c)\big)^{\theta_w + \beta_3}\Big) \\
& \ge C_1 \big(\frac{1}{2}\big)^{\theta_o + \theta_w + \beta_3+ \beta_4 -1}\frac{1}{\theta_w + \beta_3}\Big( 1-\big(\frac{1}{3}\big)^{\theta_w + \beta_3}\Big) (d-c)^{\theta_w + \beta_3}.
\end{split}
\end{equation}
Since $d-c \le1$, $\theta_o + \beta_4 >1$, and $\theta_w + \beta_3 >1$, we have in all cases
$$g(d)-g(c) \ge C_2 (d-c)^{{\rm max}(\theta_o + \beta_4,\theta_w + \beta_3)},
$$where $C_2$ is the minimum of the constant factors in \eqref{eq:case1}--\eqref{eq:case4}.
\end{proof}

The convergence to zero of the differences $I_h(p_{\Rd \alpha \Bk g}(S_{h,\tau})) - p_{\Rd \alpha \Bk g}(S_{h,\tau})$, \Rd $\alpha = w,o$,  \Bk 
follows from this lemma and Theorem
\ref{thm:grad.gfinal}.

\begin{lemma}
\label{lem:convpwg}
Under the assumptions \eqref{eq:eta'w}--\eqref{eq:pc'} on the derivatives of the mobilities and capillary pressure, there exists a constant $C$, independent of $h$ and $\tau$, such that
\begin{equation}
\label{eq:Ipwg-pwg}
\|I_h(p_{\alpha g}(S_{h,\tau})) -p_{\alpha g}(S_{h,\tau})\|_{L^2(\Omega \times ]0,T[)} \le C\,h^{\gamma_\alpha},\ \alpha =w,o,
\end{equation}
where $\gamma_w  = \frac{\beta_3}{r}$,  $\gamma_o  = \frac{\beta_4}{r}$ and  in both cases, \Bk $r$ is the exponent of Lemma  \ref{lem:lowerbdg}. 
\end{lemma}

\begin{proof}
 Let us start with $\alpha =w$. 
Arguing as in the proof of Lemma \ref{lem:convIhgSh}, with $-p_{wg}$ (monotonic increasing) instead of $g$, the analogue of \eqref{eq:convIhgSh30} holds for $-p_{wg}(S_{h,\tau})$, with the same notation
\begin{equation}
\label{eq:convIhpwgh1}
\|I_h(p_{wg}(S_{h,\tau})) -p_{wg}(S_{h,\tau})\|^2_{L^2(\Omega \times ]0,T[)} \le C \sum_{n=1}^N \tau \sum_{i=1}^M  \sum_{j \in \mathcal{N}(i)} \big(\frac{\kappa_i}{c_{ij}}\big) c_{ij}\big |p_{wg}(S^{n,j})-p_{wg}(S^{n,i})\big|^2,
\end{equation}
and the result will stem from an adequate upper bound for $p_{wg}(S^{n,j})-p_{wg}(S^{n,i})$, for all neighbors $j$ of $i$. To this end, we proceed as in Lemma \ref{lem:lowerbdg}. Let $c= S^{n,i}$, $d= S^{n,j}$ and suppose again that $c < d$; then by \eqref{eq:defpwgpog}, \eqref{eq:pc'}, \eqref{eq:etao_1}, and \eqref{eq:bddsh},
\begin{equation}
\label{eq:convIhpwgh2}
|p_{wg}(S^{n,j})-p_{wg}(S^{n,i})| \le \frac{1}{C_{\rm min} \alpha_3 \alpha_o \theta_o} \int_c^d x^{\beta_3-1} (1-x)^{\theta_o+\beta_4-1},
\end{equation}
that we write as
$$|p_{wg}(S^{n,j})-p_{wg}(S^{n,i})| \le C_1^\prime \int_c^d x^{\beta_3-1} (1-x)^{\theta_o+\beta_4-1}.
$$
Here, the discussion reduces to three cases.

1)\ If $\frac{1}{8} \le c \le \frac{7}{8}$, since $\theta_o+\beta_4-1 >0$,
\begin{equation}
\label{eq:caseIhpwgh1}
\int_c^d x^{\beta_3-1} (1-x)^{\theta_o+\beta_4-1} \le 8^{1-\beta_3}\int_c^d  (1-x)^{\theta_o+\beta_4-1}
\le 8^{1-\beta_3} (d-c).
\end{equation}

2)\ Likewise, if $c > \frac{7}{8}$,
\begin{equation}
\label{eq:caseIhpwgh2}
\int_c^d x^{\beta_3-1} (1-x)^{\theta_o+\beta_4-1} \le \big(\frac{8}{7}\big)^{1-\beta_3} (d-c).
\end{equation}

3)\ If $c <\frac{1}{8}$,
\begin{equation}
\label{eq:caseIhpwgh3}
\int_c^d x^{\beta_3-1} (1-x)^{\theta_o+\beta_4-1} \le \int_c^d x^{\beta_3-1} =\frac{1}{\beta_3}\big(d^{\beta_3}-c^{\beta_3} \big) \le \frac{1}{\beta_3}\big(d-c\big)^{\beta_3}.
\end{equation}
Indeed, by Jensen's inequality, valid for $0<\beta_3 \le 1$,
$$d =c+(d-c) \le \big(c^{\beta_3} + (d-c)^{\beta_3}\big)^{\frac{1}{\beta_3}},\quad \mbox{i.e.,}\  d^{\beta_3} \le c^{\beta_3} + (d-c)^{\beta_3}.
$$
Consequently, in all cases,
\begin{equation}
\label{eq:allcaseIhpwgh}
|p_{wg}(S^{n,j})-p_{wg}(S^{n,i})| \le C_2^\prime \big|S^{n,j}-S^{n,i}\big|^{\beta_3}.
\end{equation}
Thus, by substituting into \eqref{eq:convIhpwgh1}, applying Lemma \ref{lem:lowerbdg}, and setting $\gamma_w = \frac{\beta_3}{r}$, we infer
\begin{align*}
\|I_h(p_{wg}(S_{h,\tau})) -p_{wg}(S_{h,\tau})\|^2_{L^2(\Omega \times ]0,T[)} &\le C \sum_{n=1}^N \sum_{i=1}^M  \sum_{j \in \mathcal{N}(i)} \frac{\kappa_i}{c_{ij}} \tau  c_{ij}\big |S^{n,j}-S^{n,i}|^{2\beta_3}\\
& \le C \sum_{n=1}^N \sum_{i=1}^M  \sum_{j \in \mathcal{N}(i)}  \frac{\kappa_i}{c_{ij}} \tau c_{ij} A_{ij}^{2\gamma_w},
\end{align*}
where $A_{ij} =  |g(S^{n,j})-g(S^{n,i})|$. Note that $r > \beta_3$, hence $\gamma_w < 1$. Then
\begin{align*}
\|I_h(p_{wg}(S_{h,\tau})) -&p_{wg}(S_{h,\tau})\|^2_{L^2(\Omega \times ]0,T[)} \le C \sum_{n=1}^N \sum_{i=1}^M  \sum_{j \in \mathcal{N}(i)}  \frac{\kappa_i}{c_{ij}} (\tau c_{ij})^{1-\gamma_w}\big(\tau c_{ij}\big)^{\gamma_w} A_{ij}^{2\gamma_w}\\
& \le C \Big(\sum_{n=1}^N \sum_{i=1}^M  \sum_{j \in \mathcal{N}(i)}  \tau c_{ij}A_{ij}^2\Big)^{\gamma_w}
\Big(\sum_{n=1}^N \sum_{i=1}^M  \sum_{j \in \mathcal{N}(i)}  \big(\frac{\kappa_i}{c_{ij}}\big)^{\frac{1}{1-\gamma_w}}\tau c_{ij}\Big)^{1-\gamma_w}.
\end{align*}
But
$$\Big(\sum_{n=1}^N \sum_{i=1}^M  \sum_{j \in \mathcal{N}(i)}  \big(\frac{\kappa_i}{c_{ij}}\big)^{\frac{1}{1-\gamma_w}}\tau c_{ij}\Big)^{1-\gamma_w} \le C \big(T|\Omega|\big)^{1-\gamma_w} \sup_{i,j}\big(\frac{\kappa_i}{c_{ij}}\big)^{\gamma_w},
$$
and \eqref{eq:Ipwg-pwg} with $\alpha = w$ follows from \eqref{eq:formcij1},  the regularity of the mesh, and Theorem
\ref{thm:grad.gfinal}.

When $\alpha =o$, the proof is based on fact that $-p_{og}$ is nonnegative, monotonically increasing, and satisfies the inequality
$$-p_{og}(x) \le \frac{1}{C_{\rm min}} \frac{1}{\alpha_3 \alpha_w \theta_w} \int_0^x x^{\theta_w + \beta_3-1} (1-x)^{\beta_4 -1}.
$$
By comparing with \eqref{eq:convIhpwgh2}, we see that the above argument carries over to $p_{og}$ with  $\beta_3$ replaced by $\beta_4$. \Bk
\end{proof}

 Finally, with the notation of Lemma \ref{lem:convpwg}, the following bound regarding $p_c(S_{h,\tau})$ follows from \eqref{eq:Ipwg-pwg} and \eqref{eq:pwg+pog}, and the fact that $p_c(0)$ is a constant:
\begin{equation}
\label{eq:Ipc-pc}
\|I_h(p_{c}(S_{h,\tau})) -p_{c}(S_{h,\tau})\|_{L^2(\Omega \times ]0,T[)} \le C\,h^{\gamma},
\end{equation}
where $\gamma  = \frac{1}{r}\rm{min}(\beta_3, \beta_4)$.
\Bk

%---------------------

\subsection{Weak convergence}
\label{sec:weakconv}
All constants below are independent of $h$ and $\tau$. 

The bound \eqref{eq:maxprinc} on the discrete saturation $S_{h,\tau}$ implies that there exists a function $\bar s \in L^\infty(\Omega \times ]0,T[)$ and a subsequence of $(h,\tau)$ not indicated, such that
\begin{equation}
\label{eq:wklimS}
\lim_{(h,\tau) \to (0,0)} S_{h,\tau}  = \bar s \quad \mbox{ weakly* in}\ L^\infty(\Omega \times ]0,T[).
\end{equation}

\begin{proposition}
\label{pro:bddS}
The limit function $\bar s$ satisfies
\begin{equation}
\label{eq:boundS1}
\forall (x,t)\ \mbox{a.e. in }  \Omega \times]0,T[,\quad 0 \le \bar s(x,t) \le 1.
\end{equation}
\end{proposition}

\begin{proof}
The convergence \eqref{eq:wklimS} means that for all $\psi \in L^1(\Omega \times ]0,T[)$,
$$ \int_{\Omega \times ]0,T[} S_{h,\tau} \psi \to \int_{\Omega \times ]0,T[} \bar s \psi \ \mbox{and} \ \int_{\Omega \times ]0,T[} (1-S_{h,\tau}) \psi \to \int_{\Omega \times ]0,T[} (1-\bar s) \psi.
$$
We argue by contradiction. Suppose that $\bar s >1$ on a set of positive measure, say $D$, and take $\psi = (\bar s-1)_+$, the positive part of $\bar s-1$. Then
$$0 \le  \int_{\Omega \times ]0,T[} \big(1-S_{h,\tau}\big) \psi \to \int_{\Omega \times ]0,T[} (1-\bar s) (\bar s-1)_+ = \int_D (1-\bar s) (\bar s-1)_+,
$$
thus contradicting the fact that $(1-\bar s)<0$ on $D$. This proves that $\bar s \le 1$. The proof that $\bar s\ge 0$ is similar.
\end{proof}

Regarding the pressure, the bound \eqref{eq:bddglobpres} yields weak convergence, up to a subsequence, of the gradient of $U_{\alpha,h,\tau}$. \Bk We can deduce weak convergence of the sequences \Rd themselves \Bk  by applying \eqref{eq:genPoinc}.
Indeed,
$$\int_\Omega U_{w,h,\tau} = \big(U_{w,h,\tau} ,1\big)_h = \big(I_h(p_{wg}(S_{h,\tau})),1\big)_h,
$$
owing to \eqref{eq:vars4}. Then \Bk the properties of $p_{wg}$ and the boundedness of $S_{h,\tau}$ imply that 
$$\big| \big(I_h(p_{wg}(S_{h,\tau})),1\big)_h\big| \le C.
$$
 Similarly,
$$
\int_\Omega U_{o,h,\tau}  = \big(I_h(p_{wg}(S_{h,\tau}))+ p_c(0),1\big)_h,
$$
a bounded quantity. \Bk
Then we infer from \eqref{eq:genPoinc} that
\begin{equation}
\label{eq:bddpL2}
\|U_{\alpha,h,\tau}\|_{L^2(\Omega\times ]0,T[)}  \le C, \quad \alpha =w,o.
\end{equation}
\Rd With \Bk  this, \eqref{eq:bddglobpres}   implies that there exist functions $\bar W_{\Rd \alpha \Bk} \in L^2(0,T;H^1(\Omega))$, \Rd $\alpha = w,o$, \Bk  and a subsequence of $h$ and $\tau$ (not indicated) such that,
\begin{equation}
\label{eq:wklimprs}
\lim_{(h,\tau) \to (0,0)} U_{\alpha,h,\tau}  =\bar W_\alpha,\ \mbox{ weakly in}\ L^2(0,T;H^1(\Omega)) . 
\end{equation}

Likewise, the function $I_h(g(S_{h,\tau}))$ is bounded in $L^2(\Omega \times ]0,T[)$ and it follows from this and \eqref{eq:Bddgrad.final} that there exists a function $\bar K\in L^2(0,T,H^1(\Omega))$ such that, up to a subsequence,
\begin{equation}
\label{eq:wklimg}
\lim_{(h,\tau) \to (0,0)} I_h(g(S_{h,\tau})) =\bar K \quad \mbox{ weakly in}\ L^2(0,T,H^1(\Omega)).
\end{equation}
This implies in particular that for almost every time $t$, $I_h(g(S_{h,\tau}))$ converges strongly in $L^2(\Omega)$.
But as is well-known, these convergences are not sufficient to pass to the limit in the nonlinear terms: they must be supplemented by a bound for a fractional derivative in time of $S_{h,\tau}$ that yields compactness in time. This will stem via a  bound for a fractional derivative in time of $g(S_{h,\tau})$.

%---------------------

\subsection{Compactness in time}
\label{sec:compact}

 Following the argument introduced by Kazhikhov, see~\cite{Lions78}, and recalling that 
 $\|\cdot \|^\varphi_h$ is equivalent to the $L^2$ norm in finite dimension, we want to derive first a fractional estimate in time for $I_h(g(S_{h,\tau}))$ and next for $g(S_{h,\tau})$. The following lemma is a preliminary bound written in terms of sums of the pointwise values in time.
\begin{lemma}
\label{lem:fractbdd}  
Under the assumptions of Theorem \ref{thm:grad.gfinal}, there exist  constants $C$, independent of $h$ and $\tau$,
such that for all integers $1 \le \ell \le N-1$,
\begin{equation}
\label{eq:Kazhikhov}
\sum_{m=1}^{N-\ell} \tau \big(\|g(S_h^{m+\ell}) -g(S_h^{m})\|^\varphi_h\big)^2 \le C(\ell \tau), \ 
\sum_{m=1}^{N-\ell} \tau \|g(S_h^{m+\ell}) -g(S_h^{m})\|_{L^2(\Omega)}^2 \le C(\ell \tau).
\end{equation}
\end{lemma}
 
\begin{proof}
The starting point is the inequality
\begin{equation}
\label{eq:Kazhikhov1}
\sum_{m=1}^{N-\ell} \tau \big(\|g(S_h^{m+\ell}) -g(S_h^{m})\|^\varphi_h\big)^2 \le L \sum_{m=1}^{N-\ell}\tau \Big( g(S_h^{m+\ell}) -g(S_h^{m}), S_h^{m+\ell} -S_h^{m}\Big)_h^\varphi,
\end{equation}
owing that $g$ is Lipschitz continuous and increasing. Thus, by writing 
$$S_h^{m+\ell}-S_h^{m} = \sum_{k=1}^\ell \big(S_h^{m+k}-S_h^{m+k-1} \big),
$$
testing each line of \eqref{eq:vars1} taken at level $m+k$ with $I_h\big(g(S_h^{m+\ell}) -g(S_h^{m})\big)$, and applying \eqref{eq:Kazhikhov1}, we obtain
\begin{equation}
\label{eq:Kazhikhov2}
\begin{split}
\sum_{m=1}^{N-\ell} \tau \big(\|g(S_h^{m+\ell}) -g(S_h^{m})\|^\varphi_h\big)^2& \le L \sum_{m=1}^{N-\ell} \tau \sum_{k=1}^\ell \tau \Big|\big(f_w(s_\mathrm{in,h}^{m+k})\bar{q}_h^{m+k} - f_w(S_h^{m+k}) \underline{q}_h^{m+k}, g(S_h^{m+\ell}) -g(S_h^{m})\big)_h\\
&+\big[P_{w,h}^{m+k}, I_h(\eta_w(S_h^{m+k})); P_{w,h}^{m+k}, I_h\big(g(S_h^{m+\ell}) -g(S_h^{m})\big)\big]_h \Big|.
\end{split}
\end{equation}
It is easy to check that, on one hand, with $r=\ell$ or $r=0$,
\begin{align*}
\Big|\big[P_{w,h}^{m+k}&, I_h(\eta_w(S_h^{m+k})); P_{w,h}^{m+k}, I_h\big(g(S_h^{m+r})\big)\big]_h\Big|\\
&= \frac{1}{2}\Big| \sum_{i,j =1}^M \big(g(S^{m+r,j}) - g(S^{m+r,i})\big) c_{ij} \eta_w(S_w^{m+k,ij})\big(P_w^{m+k,j} - P_w^{m+k,i}\big)\Big|\\
& \le \frac{1}{4}\sum_{i,j =1}^M c_{ij} \eta_w(S_w^{m+k,ij})\Big(\big|g(S^{m+r,j}) - g(S^{m+r,i})\big|^2+ \big|P_w^{m+k,j} - P_w^{m+k,i}\big|^2\Big)\\
 & \le \frac{1}{4}\sum_{i,j =1}^M c_{ij}\Big(\eta_w(1) \big|g(S^{m+r,j}) - g(S^{m+r,i})\big|^2
 +  \eta_w(S_w^{m+k,ij})\big|P_w^{m+k,j} - P_w^{m+k,i}\big|^2\Big),
\end{align*}
since $\eta_w$ is increasing and $S_{h,\tau}$ is bounded by one. On the other hand, 
$$
\big|\big(f_w(s_\mathrm{in,h}^{m+k})\bar{q}_h^{m+k} - f_w(S_h^{m+k}) \underline{q}_h^{m+k}, g(S_h^{m+\ell}) -g(S_h^{m})\big)_h\big| \le C \big(\|\bar{q}^{m+k}\|_{L^1(\Omega)} + \|\underline{q}^{m+k}\|_{L^1(\Omega)}\big),
$$
where here and below, $C$ denotes constants that are independent of $\ell$, $h$, and $\tau$.
Therefore, in view of \eqref{eq:normgrad}
\begin{equation}
\label{eq:Kazhikhov3}
\begin{split}
&\sum_{m=1}^{N-\ell} \tau \big(\|g(S_h^{m+\ell}) -g(S_h^{m})\|^\varphi_h\big)^2 \le L \sum_{m=1}^{N-\ell} \tau\Big( \Big[\frac{1}{8}\eta_w(1) (\ell \tau) \sum_{r=\ell,0}\|\nabla\,I_h(g(S_h^{m+r}))\|^2_{L^2(\Omega)} \\
& + \frac{1}{2}\sum_{k =1}^\ell \tau \sum_{i,j =1}^M c_{ij}\eta_w(S_w^{m+k,ij})\big|P_w^{m+k,j} - P_w^{m+k,i}\big|^2\Big]
 + C  \sum_{k =1}^\ell \tau  \big(\|\bar{q}^{m+k}\|_{L^1(\Omega)} + \|\underline{q}^{m+k}\|_{L^1(\Omega)}\big)\Big)\\
& \le \frac{1}{8}\eta_w(1) L (\ell \tau) \Big[\sum_{m=1+ \ell}^{N} \tau 
\|\nabla\,I_h(g(S_h^{m}))\|^2_{L^2(\Omega)}+ \sum_{m=1}^{N-\ell} \tau \|\nabla\,I_h(g(S_h^{m}))\|^2_{L^2(\Omega)}\Big]\\
& + \frac{1}{2} L \sum_{m=1}^{N-\ell} \tau \sum_{k =1}^\ell \tau\Big( \sum_{i,j =1}^M c_{ij}\eta_w(S_w^{m+k,ij})\big|P_w^{m+k,j} - P_w^{m+k,i}\big|^2
+ C   \big(\|\bar{q}^{m+k}\|_{L^1(\Omega)} + \|\underline{q}^{m+k}\|_{L^1(\Omega)}\big)\Big).
\end{split}
\end{equation}
By \eqref{eq:Bddgrad.final}, il suffices to bound the terms in the last  line above. This is achieved by interchanging the sums over $m$ and $k$. Let $n = m+k$; $n$ runs from $2$ to $N$ and $m$ runs from 
${\rm max}(1,n-\ell)$ to ${\rm min}(n-1,N-\ell)$. Thus
\begin{align*}
\sum_{m=1}^{N-\ell} \tau \sum_{k =1}^\ell \tau &\sum_{i,j =1}^M c_{ij}\eta_w(S_w^{m+k,ij})\big|P_w^{m+k,j} - P_w^{m+k,i}\big|^2 \\
&= \sum_{n=2}^{N} \tau\big( \sum_{m ={\rm max}(1,n-\ell)}^{{\rm min}(n-1,N-\ell)}
\tau\Big) \sum_{i,j =1}^M c_{ij}\eta_w(S_w^{n,ij})\big|P_w^{n,j} - P_w^{n,i}\big|^2.
\end{align*}
But ${\rm min}(n-1,N-\ell) - {\rm max}(1,n-\ell) \le \ell-1$. Hence
\begin{equation}
\label{eq:Kazhikhov4}
\sum_{m=1}^{N-\ell} \tau \sum_{k =1}^\ell \tau \sum_{i,j =1}^M c_{ij}\eta_w(S_w^{m+k,ij})\big|P_w^{m+k,j} - P_w^{m+k,i}\big|^2  \le (\ell \tau)\sum_{n=2}^{N} \tau \sum_{i,j =1}^M c_{ij}\eta_w(S_w^{n,ij})\big|P_w^{n,j} - P_w^{n,i}\big|^2,
\end{equation}
and we know from \eqref{eq:pressbound1} that this last sum over $n$ is bounded. In the same fashion,
\begin{equation}
\label{eq:Kazhikhov5}
\sum_{m=1}^{N-\ell} \tau \sum_{k =1}^\ell \tau \big(\|\bar{q}^{m+k}\|_{L^1(\Omega)} + \|\underline{q}^{m+k}\|_{L^1(\Omega)}\big) \le (\ell \tau) \big(\|\bar{q}\|_{L^1(\Omega\times ]0,T[)} + \|\underline{q}\|_{L^1(\Omega\times ]0,T[)}\big).
\end{equation}
Then, under the assumptions of Theorem \ref{thm:grad.gfinal}, \eqref{eq:Kazhikhov} follows by 
substituting \eqref{eq:Bddgrad.final}, \eqref{eq:Kazhikhov4}, and \eqref{eq:Kazhikhov5} into \eqref{eq:Kazhikhov3}. The second inequality  stems from the first and \eqref{eq:boundg(Sh)}.
\end{proof}

The next theorem transforms \eqref{eq:Kazhikhov} into integrals.

\begin{theorem}
\label{thm:Kolmg}
Under the assumptions of Theorem \ref{thm:grad.gfinal}, there exists a constant $C$, independent of $h$, and $\tau$,
such that for all real numbers $\delta$, $0< \delta < T$, 
\begin{equation}
\label{eq:Kolmogorov}
\int_{0}^{T-\delta}  \big\|g(S_{h,\tau}(t+\delta)) -g(S_{h,\tau}(t)\big\|^2_{L^2(\Omega)} \,d t \le C\delta .
\end{equation}
Similarly,
\begin{equation}
\label{eq:Kolmogorov.h}
\int_{0}^{T-\delta}  \big\|I_h\big(g(S_{h,\tau}(t+\delta)) -g(S_{h,\tau}(t)\big)\big\|^2_{L^2(\Omega)} \,d t \le C\delta ,
\end{equation}
with another constant $C$, independent of $h$, and $\tau$.
\end{theorem}
 
\begin{proof}
The argument is not new, see for instance~\cite{Walkingt10}, but we recall it for the reader's convenience. The discussion depends on the value of $\delta$; there are three cases: (i)\ $0 < \delta < \tau$, 
(ii)\ $\delta = \ell \tau$, $1 \le \ell \le N-1$,
(iii)\ $\delta = \tau(\ell + \eta)$, $1 \le \ell \le N-1$, $0<\eta <1$.

(i)\ If $0 < \delta < \tau$, we have for all $f$ in $L^1(0,T)$
$$\int_0^{T-\delta} f(t)\,dt = \sum_{m=0}^{N-2} \Big(\int_{t_m}^{t_{m+1}-\delta} f(t) \,dt 
+ \int_{t_{m+1}-\delta}^{t_{m+1}} f(t) \,dt\Big) + \int_{t_{N-1}}^{t_{N}-\delta} f(t) \,dt.
$$
On the interval $(t_m, t_{m+1}-\delta)$, by convention, see \eqref{eq:Pconst} applied to $S_{h,\tau}$,
$S_{h,\tau} (t+ \delta) = S_h^{m+1} = S_{h,\tau}(t)$ 
and on the interval $( t_{m+1}-\delta, t_{m+1})$, $S_{h,\tau} (t+ \delta) = S_h^{m+2}$. Therefore
\begin{align*}
\int_0^{T-\delta}  \big\|g(S_{h,\tau}(t+\delta)) -g(S_{h,\tau}(t))\big\|^2_{L^2(\Omega)} \,d t &= \sum_{m=0}^{N-2} \delta  \big\|g(S_{h}^{m+2}) -g(S_{h}^{m+1})\big\|^2_{L^2(\Omega)}\\
&= \sum_{m=1}^{N-1} \delta  \big\|g(S_{h}^{m+1}) -g(S_{h}^{m})\big\|^2_{L^2(\Omega)},
\end{align*}
and  \eqref{eq:Kazhikhov} with $\ell=1$ yields \eqref{eq:Kolmogorov} for this value of $\delta$.

(ii)\ Let $\delta = \ell \tau$, for instance consider $\ell =2$. We have
$$\int_{t_m}^{t_{m+1}}\big\|g(S_{h,\tau}(t+2\tau)) -g(S_{h,\tau}(t))\big\|^2_{L^2(\Omega)} \,d t =   
\tau \big\|g(S_{h}^{m+3}) -g(S_{h}^{m+1})\big\|^2_{L^2(\Omega)}.
$$
Thus
$$
\int_0^{T-2 \tau}  \big\|g(S_{h,\tau}(t+\delta)) -g(S_{h,\tau}(t))\big\|^2_{L^2(\Omega)} \,d t = \tau \sum_{m=1}^{N-2} \big\|g(S_{h}^{m+2}) -g(S_{h}^{m})\big\|^2_{L^2(\Omega)},
$$
and \eqref{eq:Kazhikhov} with $\ell=2$ yields \eqref{eq:Kolmogorov}. This argument applies to any $\ell$ with $1 \le \ell \le N-1$.

(iii)\ Let $\delta = \tau(\ell + \eta)$ for some $0<\eta <1$, and for instance consider again $\ell =2$. Then for any integer $m$, $0 \le m \le N-3$,
$$\int_{t_m}^{t_{m+1}- \tau\eta} \big\|g(S_{h,\tau}(t+ \tau(2+\eta)) -g(S_{h,\tau}(t))\big\|^2_{L^2(\Omega)} = \tau(1-\eta)\big\|g(S_{h}^{m+3}) -g(S_{h}^{m+1})\big\|^2_{L^2(\Omega)},$$
and
$$\int_{t_{m+1}- \tau \eta}^{t_{m+1}} \big\|g(S_{h,\tau}(t+\tau(2+\eta)) -g(S_{h,\tau}(t))\big\|^2_{L^2(\Omega)} = \tau \eta\big\|g(S_{h}^{m+4}) -g(S_{h}^{m+1})\big\|^2_{L^2(\Omega)}.$$
Hence
\begin{equation}
\nonumber
\begin{split}
&\int_{0}^{T-\tau(2+ \eta)}  \big\|g(S_{h,\tau}(t+\tau(2+ \eta))) -g(S_{h,\tau}(t))\big\|^2_{L^2(\Omega)} \,d t\\
& = \sum_{m=1}^{N-2} \tau(1-\eta)  \big\|g(S_{h}^{m+2}) -g(S_{h}^{m})\big\|^2_{L^2(\Omega)}
+ \sum_{m=1}^{N-3} \tau \eta \big\|g(S_{h}^{m+3}) -g(S_{h}^{m})\big\|^2_{L^2(\Omega)}.
\end{split}
\end{equation}
Then  \eqref{eq:Kazhikhov} with $\ell =2$ and $\ell =3$  implies that
$$
\int_{0}^{T-\tau(2+ \eta)}  \big\|g(S_{h,\tau}(t+\tau(2+ \eta)) -g(S_{h,\tau}(t))\big\|^2_{L^2(\Omega)} \,d t 
\le C\big(2 \tau(1-\eta) + 3 \tau\eta\big) = C\tau (2+ \eta).
$$
More generally, we have
\begin{align*}
&\int_0^{T-\tau(\ell + \eta)}  \big\|g(S_{h,\tau}(t+\tau(\ell+ \eta)) -g(S_{h,\tau}(t))\big\|^2_{L^2(\Omega)} \,d t\\
& = \tau(1-\eta)\sum_{m=1}^{N-\ell}   \big\|g(S_{h}^{m+\ell}) -g(S_{h}^{m})\big\|^2_{L^2(\Omega)}
+ \tau \eta  \sum_{m=1}^{N-(\ell+1)}  \big\|g(S_{h}^{m+\ell+1}) -g(S_{h}^{m})\big\|^2_{L^2(\Omega)},
\end{align*}
and we apply \eqref{eq:Kazhikhov} with $\ell$ to the first sum and $\ell+1$ to the second sum. This proves the first part of \eqref{eq:Kolmogorov}. The proof of \eqref{eq:Kolmogorov.h} follows likewise from \eqref{eq:Kazhikhov} with $\varphi =1$.
\end{proof}

%---------------------

\subsection{Strong convergence}
\label{sec:strongconv}

With Theorem \ref{thm:Kolmg}, it follows from Kolmogorov's theorem that the sequence $I_h(g(S_{h,\tau}))$ is compact in $L^2(\Omega \times ]0,T[)$, see~\cite{Lions78}. Thus,
again up to a subsequence, $I_h(g(S_{h,\tau}))$ converges strongly in $L^2(\Omega \times ]0,T[)$. Since it converges weakly to $\bar K$ in $L^2(0,T;H^1(\Omega))$ ($\bar K$ belongs also to $L^\infty(\Omega \times ]0,T[)$), uniqueness of the limit implies 
\begin{equation}
\label{eq:stlimg}
\lim_{(h,\tau) \to (0,0)} I_h(g(S_{h,\tau})) = \bar K \quad \mbox{ strongly in}\ L^2(\Omega \times ]0,T[).
\end{equation}
By Lemma \ref{lem:convIhgSh}, this also implies 
\begin{equation}
\label{eq:stlimg.noh}
\lim_{(h,\tau) \to (0,0)} g(S_{h,\tau}) = \bar K \quad \mbox{ strongly in}\ L^2(\Omega \times ]0,T[).
\end{equation}
From here, let us prove the strong convergence of $S_{h,\tau}$. Recall that  $g$ is invertible with range $]0,\beta[$ and inverse $g^{-1}\in W^{1,\infty}(]0,\beta[)$. Let $F_{h,\tau} = g(S_{h,\tau})$; then 
$$S_{h,\tau} = g^{-1}(F_{h,\tau}).
$$
The strong convergence of $F_{h,\tau}$ and the continuity of $g^{-1}$ imply the strong convergence of $S_{h,\tau}$ to $g^{-1}(\bar K)$ in $L^2(\Omega \times ]0,T[)$, and since $S_{h,\tau}$ converges weakly to $\bar s$, uniqueness of the limit implies that $\bar s = g^{-1}(\bar K)$, i.e.,
\begin{equation}
\label{eq:slimSh}
\lim_{(h,\tau) \to (0,0)} S_{h,\tau} = \bar s = g^{-1}(\bar K) \quad \mbox{ strongly in}\ L^2(\Omega \times ]0,T[).
\end{equation}
This strong convergence and the continuity of $g$,  $p_{ \alpha \Bk g}$,  $\alpha = w,o$, and $p_c$, \Bk also imply that
\begin{equation}
\label{eq:slimgSh/pwgSh}
\lim_{(h,\tau) \to (0,0)} g(S_{h,\tau}) = g(\bar s),\quad  \lim_{(h,\tau) \to (0,0)} p_{ \alpha \Bk g}(S_{h,\tau}) = p_{ \alpha \Bk g}(\bar s),  \alpha = w,o, \Bk \quad \lim_{(h,\tau) \to (0,0)} p_c(S_{h,\tau}) = p_c(\bar s), 
\end{equation}
\ all \Bk strongly in $L^2(\Omega \times ]0,T[)$.
Furthermore  Lemma \ref{lem:convpwg} and \eqref{eq:Ipc-pc} \Bk  yield
 \begin{equation}
\label{eq:slimIhpwgSh}
\lim_{(h,\tau) \to (0,0)} I_h(p_{\Rd \alpha \Bk g}(S_{h,\tau})) = p_{\Rd \alpha \Bk g}(\bar s), \quad  
\lim_{(h,\tau) \to (0,0)} I_h(p_{c}(S_{h,\tau})) = p_{c}(\bar s), \Bk \quad 
\mbox{strongly in}\ L^2(\Omega \times ]0,T[).
\end{equation}
In view of \eqref{eq:wklimprs}, this convergence implies that $P_{\alpha,h,\tau}$ converges weakly in $L^2(\Omega \times ]0,T[)$ to some function  $\bar{p_{ \alpha \Bk}} \in L^2(\Omega \times ]0,T[)$, $\alpha =w,o$. Furthermore, uniqueness of the limit implies that $\bar W_{\alpha}$, the limit function of $U_{\alpha,h,\tau}$   has the form
\begin{equation}
\label{eq:splitW} 
\bar W_{w} = \bar{p_w} + p_{wg}(\bar s),\quad  \bar W_{o} = \bar{p_o} - p_{og}(\bar s). 
\end{equation}

%==================================

\section{Identification of the limit}
\label{sec:identlim}

Let us pass to the limit in the equations of the scheme. 
This is done in several steps because we do not have convergence of the pressure gradient. 

%==========================

\subsection{The  upwind diffusions}
\label{sec:pressure}

Since the discrete auxiliary pressures $U_{\alpha,h,\tau} $  converge weakly to  $\bar W_{\Rd \alpha \Bk}$ in $L^2(0,T;H^1(\Omega))$, instead of treating directly  the upwind diffusion terms  
$\big[P_{\Rd \alpha \Bk,h, \tau},I_h(\eta_{\Rd \alpha \Bk}(S_{h,\tau})); P_{\Rd \alpha \Bk,h, \tau},\theta_h\big]_h,$
we begin with  $\big[P_{\alpha,h, \tau},I_h(\eta_\alpha(S_{h,\tau})); U_{\alpha,h, \tau} ,\theta_h\big]_h$.

%---------------- 

\subsubsection{Discrete auxiliary pressure}
Let us start with the wetting phase, the treatment of the non-wetting phase being much the same. \Bk

Let $v$ be a smooth function, say $v \in {\mathcal C}^1(\bar \Omega \times [0,T])$ and let $V_{h,\tau} = \rho_\tau(I_h(v))$. Assume for the moment that $\bar s$, the limit of $S_{h,\tau}$, is sufficiently smooth, say $\bar s \in W^{1,\infty}(\Omega \times ]0,T[)$ and let 
$\bar s_\tau = \bar  s(t_n) $ in $]t_{n-1},t_n]$. Then assumption \eqref{eq:eta'w} implies
\begin{equation}
\label{eq:meanetaSw}
\big\|\frac{1}{\tau}  \int_{t_{n-1}}^{t_{n}} \eta_w(\bar s)\,dt - \eta_w(\bar s_\tau^{n})\|_{L^\infty(\Omega)} \le C \tau \|\eta_w^\prime\|_{L^\infty(0,1)} \|\partial_t \bar s\|_{L^\infty(\Omega \times ]0,T[)}.
\end{equation}
Proceeding as in Section \ref{subsec:motiv}, we treat the upwinding in several steps and consider first
\begin{equation}
\label{eq:meanetaSw0}
\int_0^T\int_\Omega \eta_w(\bar s) \nabla\,U_{w,h,\tau} \cdot \nabla\,V_{h,\tau} = \int_0^T\int_\Omega \nabla\,U_{w,h,\tau} \cdot \nabla\,V_{h,\tau}\big(\rho_\tau(\eta_w(\bar s)) -\eta_w(\bar s_\tau) + \eta_w(\bar s_\tau)\big).
\end{equation}
But in view of \eqref{eq:meanetaSw},
\begin{align*}
\Big|\int_0^T\int_\Omega \nabla\,U_{w,h,\tau} \cdot \nabla\,V_{h,\tau}\big(\rho_\tau(\eta_w(\bar s)) -\eta_w(\bar s_\tau))&\Big| \le C \tau \|\eta_w^\prime\|_{L^\infty(0,1)} 
\|\partial_t \bar s\|_{L^\infty(\Omega \times ]0,T[)}\\
 & \times\|U_{w,h,\tau}\|_{L^2(0,T;H^1(\Omega)}\|V_{h,\tau}\|_{L^2(0,T;H^1(\Omega))},
\end{align*}
and the boundedness of all  factors of $\tau$, owing to \eqref{eq:bddglobpres} and the regularity of $v$, implies
\begin{equation}
\label{eq:meanetaSw1}
\lim_{(h,\tau) \to (0,0)} \int_0^T\int_\Omega \nabla\,U_{w,h,\tau} \cdot \nabla\,V_{h,\tau}\big(\rho_\tau(\eta_w(\bar s)) -\eta_w(\bar s_\tau)) =0.
\end{equation}
Next  the weak convergence of $U_{w,h,\tau}$ to $\bar W_{\Rd w \Bk}$ in $L^2(0,T;H^1(\Omega))$, the strong convergence of $V_{h,\tau}$ to $v$ in $L^\infty(0,T;W^{1,\infty}(\Omega))$, the continuity of $\eta_w$, the regularity of $\bar s$, and \eqref{eq:meanetaSw1} imply 
$$\lim_{(h,\tau) \to (0,0)}\int_0^T\int_\Omega \eta_w(\bar s_\tau) \nabla\,U_{w,h,\tau} \cdot \nabla\,V_{h,\tau} = \int_0^T\int_\Omega \eta_w(\bar s) \nabla\,\bar W_{\Rd w \Bk} \cdot \nabla\,v.
$$
Let us expand the expression in the above left-hand side.
With the notation  \eqref{eq:cijk}, in view of Proposition \ref{pro:equality} we have 
$$
\int_0^T\int_\Omega \eta_w(\bar s_\tau) \nabla\,U_{w,h,\tau} \cdot \nabla\,V_{h,\tau} = -\sum_{n=1}^{N}\tau\sum_{i,j=1}^M U_w^{n,i} 
\ \Big(\sum_{K \subset \Delta_i \cap \Delta_j} c_{ij,K} (\eta_w(\bar s_\tau^{n}))_K\Big)\big(V^{n,j} - V^{n,i}\big).
$$
By symmetry, this becomes
\begin{equation}
\label{eq:Pw+pg1}
\int_0^T\int_\Omega \eta_w(\bar s_\tau^{n}) \nabla\,U_{w,h,\tau} \cdot \nabla\,V_{h,\tau}
= \frac{1}{2}\sum_{n=1}^{N}\tau\sum_{i,j=1}^M  
 \Big(\sum_{K \subset \Delta_i \cap \Delta_j} c_{ij,K} (\eta_w(\bar s_\tau^{n}))_K\Big)\big(U_w^{n,j}- 
U_w^{n,i}\big)\big(V^{n,j} - V^{n,i}\big).
\end{equation}
Hence
\begin{equation}
\label{eq:Pw+pg2}
\lim_{(h,\tau) \to (0,0)} \frac{1}{2}\sum_{n=1}^{N} \tau\sum_{i,j=1}^M \Big(\sum_{K \subset \Delta_i \cap \Delta_j} c_{ij,K} (\eta_w(\bar s_\tau^{n}))_K\Big)\big(U_w^{n,j}- 
U_w^{n,i}\big)\big(V^{n,j} - V^{n,i}\big)
 = \int_0^T\int_\Omega \eta_w(\bar s) \nabla\,\bar W_{\Rd w \Bk} \cdot \nabla\,v.
\end{equation}
According to \eqref{eq:eta'w} and the regularity of $\bar s$, $\eta_w(\bar s)$ belongs to $L^\infty(0,T;W^{1,\infty}(\Omega))$, and \eqref{eq:wapprox} gives
$$
\big\| (\eta_w(\bar s_\tau^{n}))_K - \eta_w(\bar s_\tau)\big\|_{L^\infty(K)} \le  C\,h\,  \|\eta_w^\prime\|_{L^\infty(0,1)}\|\nabla\,\bar s\|_{L^\infty(\Omega \times ]0,T[)},
$$
that allows to replace $(\eta_w(\bar s_\tau^{n}))_K$ by any value of $\eta_w(\bar s_\tau^{n})$ in $K$. Let us choose the upwind value of $\bar s_\tau^{n}$ as in \eqref{eq:Swij}, i.e.,
\begin{equation}
\label{eq:Swijexact}
\bar s_{w,\tau}^{n,ij}  = \left\{
\begin{array}{c}
\bar s_\tau^{n}(\x_i) \quad\mbox{if}\quad P_w^{n,i}>P_w^{n,j}\\
\bar s_\tau^{n}(\x_j) \quad\mbox{if}\quad P_w^{n,i}<P_w^{n,j}\\
\max(\bar s_\tau^{n}(\x_i),\bar s_\tau^{n}(\x_j)) \quad\mbox{if}\quad P_w^{n,i} = P_w^{n,j},
\end{array}
\right.
\end{equation}
and set
$$R_{ij} = \sum_{K \subset \Delta_i \cap \Delta_j} c_{ij,K} \big((\eta_w(\bar s_\tau^{n}))_K - \eta_w(\bar s_{w,\tau}^{n,ij})\big).
$$
By proceeding as in Theorem 
\ref{thm:mainapprox} and applying \eqref{eq:bddglobpres}, the regularity of $v$, and the approximation properties of $I_h$, we obtain
\begin{align*}
&\frac{1}{2}\sum_{n=1}^{N}\tau \sum_{i,j=1}^M   R_{ij} \big(U_w^{n,j}- 
U_w^{n,i}\big)\big(V^{n,j} - V^{n,i}\big) \\
&\le \frac{1}{2}\sum_{n=1}^{N}\tau \Big(\sum_{i,j=1}^M |R_{ij}| \big(U_w^{n,j} - U_w^{n,i}\big)^2\Big)^{\frac{1}{2}} \Big(\sum_{i,j=1}^M |R_{ij}| \big(V^{n,j} - V^{n,i}\big)^2\Big)^{\frac{1}{2}}\\
&\le C\,h\,  \|\eta_w^\prime\|_{L^\infty(0,1)}\|\nabla\,\bar s\|_{L^\infty(\Omega \times ]0,T[)}\|\nabla\,U_{w,h,\tau}\|_{L^2(\Omega \times ]0,T[)} \|\nabla\,V_{h,\tau}\|_{L^2(\Omega \times ]0,T[)}\\
&\le C\,h\,  \|\eta_w^\prime\|_{L^\infty(0,1)} \|\nabla\,\bar s\|_{L^\infty(\Omega \times ]0,T[)} |v|_{H^1(0,T;H^2(\Omega))}.
\end{align*}
With \eqref{eq:Pw+pg2}, this implies
\begin{equation}
\label{eq:Pw+pg3}
\lim_{(h,\tau) \to (0,0)} \frac{1}{2}\sum_{n=1}^{N} \tau\sum_{i,j=1}^M    c_{ij} \eta_w(\bar s_{w,\tau}^{n,ij})\big(U_w^{n,j}- 
U_w^{n,i}\big)\big(V^{n,j} - V^{n,i}\big)
 = \int_0^T\int_\Omega \eta_w(\bar s) \nabla\,\bar W_{\Rd w \Bk} \cdot \nabla\,v.
\end{equation}
To recover $\int_0^T\big[P_{w,h, \tau},I_h(\eta_w(S_{h,\tau})); U_{w,h, \tau} ,V_{h,\tau}\big]_h$, we write
$$\eta_w(\bar s_{w,\tau}^{n,ij}) = \eta_w(\bar s_{w,\tau}^{n,ij}) - \eta_w(S_w^{n,ij}) 
+ \eta_w(S_w^{n,ij}),
$$
and we must examine the convergence of
$$
X:=\frac{1}{2}\sum_{n=1}^{N} \tau\sum_{i,j=1}^M   c_{ij} \big(\eta_w(\bar s_{w,\tau}^{n,ij}) - \eta_w(S_w^{n,ij})\big)\big(U_w^{n,j}- 
U_w^{n,i}\big)\big(V^{n,j} - V^{n,i}\big).
$$
On the one hand, owing to the smoothness of $v$, we have
\begin{equation}
\label{eq:deltavij}
|V^{n,j} - V^{n,i}| \le C h_i\|\nabla\,v\|_{L^\infty (\Omega \times ]0,T[)},
\end{equation}
where $h_i$ is the length of the edge whose endpoints are the vertices $i$ and $j$. On the other hand,
$$ |\eta_w(\bar s_{w,\tau}^{n,ij}) - \eta_w(S_w^{n,ij})| \le C \|\eta_w^\prime\|_{L^\infty(0,1)} |\bar s_{w,\tau}^{n,ij} - S_w^{n,ij}|.
$$
Hence
$$
|X| \le C  \|\nabla\,v\|_{L^\infty (\Omega \times ]0,T[)}
 \|\nabla\,U_{w,h,\tau}\|_{L^2(\Omega \times ]0,T[)}\Big(\sum_{n=1}^{N} \tau
 \sum_{i,j=1}^M   c_{ij} h_i^2|\bar s_{w,\tau}^{n,ij} - S_w^{n,ij}|^2\Big)^{\frac{1}{2}}.
$$
It is easy to check that
$$\sum_{i,j=1}^M   c_{ij} h_i^2|\bar s_{w,\tau}^{n,ij} - S_{w}^{n,ij}|^2 \le C \sum_{i=1}^M m_i |\bar s_{\tau}^{n,i} - S^{n,i}|^2.
$$
Therefore
\begin{align*}
|X|& \le  C  \|\nabla\,v\|_{L^\infty (\Omega \times ]0,T[)} \|\nabla\,U_{w,h,\tau}\|_{L^2(\Omega \times ]0,T[)}\Big(\sum_{n=1}^{N} \tau\big\|I_h(\bar s_{\tau}^{n}) - 
S_{h,\tau}^{n}\big\|_{L^2(\Omega)}^2\Big)^{\frac{1}{2}}\\
&= C  \|\nabla\,v\|_{L^\infty (\Omega \times ]0,T[)} \|\nabla\,U_{w,h,\tau}\|_{L^2(\Omega \times ]0,T[)}\|I_h(\bar s_{\tau}) - S_{h,\tau}\|_{L^2(\Omega\times ]0,T[)},
\end{align*}
where we have used the equivalence \eqref{eq:equivnorm}. Then, we write
$$\|I_h(\bar s_{\tau}) - S_{h,\tau}\|_{L^2(\Omega\times ]0,T[)} \le \|I_h(\bar s_{\tau}) - \bar s_\tau\|_{L^2(\Omega\times ]0,T[)} + \|\bar s_\tau - \bar s\|_{L^2(\Omega\times ]0,T[)} + \| \bar s-S_{h,\tau}\|_{L^2(\Omega\times ]0,T[)},$$  
and the approximation properties of $I_h$, the strong convergence of $\bar s_\tau$ to $\bar s$ and of $S_{h,\tau}$ to $\bar s$, all in $L^2(\Omega \times ]0,T[)$ imply that
\begin{equation}
\label{eq:Pw+pg4}
\lim_{(h,\tau) \to (0,0)}
\sum_{n=1}^{N} \tau\sum_{i,j=1}^M    c_{ij} \big(\eta_w(\bar s_{w,\tau}^{n,ij}) - \eta_w(S_w^{n,ij})\big)\big(U_w^{n,j}- 
U_w^{n,i}\big)\big(V^{n,j} - V^{n,i}\big) =0.
 \end{equation} 
A combination of \eqref{eq:Pw+pg4} and \eqref{eq:Pw+pg3} yields the intermediate convergence result when the limit function $\bar s$ is smooth,
\begin{equation}
\label{eq:Pw+pg5}
\lim_{(h,\tau) \to (0,0)}  -\sum_{n=1}^{N} \tau\big[P_{w,h, \tau},I_h(\eta_w(S_{h,\tau})); U_{w,h,\tau} ,V_{h,\tau}\big]_h 
 = \int_0^T\int_\Omega \eta_w(\bar s) \nabla\,\bar W_{\Rd w \Bk} \cdot \nabla\,v.
\end{equation}
It remains to lift the regularity restriction on $\bar s$. 
Let $(S_m)_{m\ge 1}$ be a sequence of smooth functions that tend to $\bar s$ in $L^2(\Omega \times ]0,T[)$. Then for each $\varepsilon >0$, there exists an integer $M_0$ such that 
\begin{equation}
\label{eq:M0}
\|S_{M_0} - \bar s\|_{L^2(\Omega \times ]0,T[)} \le \varepsilon.
\end{equation}
From \eqref{eq:M0}, the projection properties, and the fact that $M_0$ is fixed, we infer
\begin{equation}
\label{eq:projerror}
\begin{split}
\|\rho_\tau(\eta_w(\bar s)) - \eta_w(\bar s)&\|_{L^2(\Omega \times ]0,T[)} \le \|\rho_\tau(\eta_w(\bar s)- \eta_w(S_{M_0}))\|_{L^2(\Omega \times ]0,T[)} \\
&+ \|\rho_\tau(\eta_w(S_{M_0}))- \eta_w(S_{M_0})\|_{L^2(\Omega \times ]0,T[)} + \|\eta_w(S_{M_0})- \eta_w(\bar s)\|_{L^2(\Omega \times ]0,T[)}\\
& \le (2\,\varepsilon + C\, \tau) \|\eta_w^\prime\|_{L^\infty(0,1)}.
\end{split}
\end{equation}
 Now, we replace \eqref{eq:meanetaSw0} by 
\begin{equation}
\label{eq:meanetaSw3}
\begin{split}
\int_0^T&\int_\Omega \eta_w(\bar s) \nabla\,U_{w,h,\tau} \cdot \nabla\,V_{h,\tau} 
= \int_0^T \int_\Omega \rho_\tau\big(\eta_w(\bar s)- \eta_w(S_{M_0})\big) 
\nabla\,U_{w,h,\tau} \cdot \nabla\,V_{h,\tau}\\
 &+ \int_0^T\int_\Omega \rho_\tau\big( \eta_w(S_{M_0})\big)
\nabla\,U_{w,h,\tau} \cdot \nabla\,V_{h,\tau} \\
& = \int_0^T \int_\Omega \rho_\tau\big(\eta_w(\bar s)- \eta_w(S_{M_0})\big)
\nabla\,U_{w,h,\tau} \cdot \nabla\,V_{h,\tau} 
 + \int_0^T\int_\Omega \eta_w(S_{M_0}) \nabla\,U_{w,h,\tau} \cdot \nabla\,V_{h,\tau}.
\end{split}
\end{equation}
For the first term, owing to \eqref{eq:M0}, the projection properties, and \eqref{eq:eta'w}, we have
\begin{align*}
\Big| \int_0^T \int_\Omega \rho_\tau\big(\eta_w(\bar s)- &\eta_w(S_{M_0})\big) 
\nabla\,U_{w,h,\tau} \cdot \nabla\,V_{h,\tau}  \Big| \le \|\nabla\, V_{h,\tau}\|_{L^\infty(\Omega \times ]0,T[)}\\
& \times \| \nabla\, U_{w,h, \tau}\|_{L^2(\Omega\times ]0,T[)}  \| \eta_w(\bar s)- \eta_w(S_{M_0})\|_{L^2(\Omega \times ]0,T[)}\\
& \le \varepsilon \,\|\eta_w^\prime\|_{L^\infty(0,1)}\|\nabla\, V_{h,\tau}\|_{L^\infty(\Omega \times ]0,T[)} \| \nabla\, U_{w,h,\tau}\|_{L^2(\Omega \times ]0,T[)}.
\end{align*}
Then the uniform boundedness of $U_{w,h,\tau}$ and $V_{h,\tau}$ yield
\begin{equation}
\label{eq:denst1}
\Big| \int_0^T \int_\Omega \rho_\tau\big(\eta_w(S)- \eta_w(S_{M_0})\big) 
\nabla\,U_{w,h,\tau} \cdot \nabla\,V_{h,\tau}  \Big| \le C\, \varepsilon,
\end{equation}
with a constant $C$ independent of $h$ and $\tau$. Thus, we must examine the limit of the second term.
Since $M_0$ is fixed and $S_{M_0}$ is smooth, by reproducing the previous steps, we derive the analogue of
\eqref{eq:Pw+pg3} for the function $S_{M_0}$,
\begin{equation}
\label{eq:Pw+pg3M0}
\begin{split}
\lim_{(h,\tau) \to (0,0)} \frac{1}{2}\sum_{n=1}^{N} \tau\sum_{i,j=1}^M &  c_{ij} \eta_w((S_{M_0})_{w,\tau}^{n,ij})\big(U_w^{n,j}- 
U_w^{n,i}\big)\big(V^{n,j} - V^{n,i}\big)\\
& = \int_0^T\int_\Omega \eta_w(S_{M_0}) \nabla\,\bar W_{\Rd w \Bk} \cdot \nabla\,v = \int_0^T\int_\Omega \eta_w(\bar s) \nabla\,\bar W_{\Rd w \Bk} \cdot \nabla\,v + R,
\end{split}
\end{equation}
where
\begin{equation}
\label{eq:Pw+pg3M0R}
\begin{split}
 |R| &= \big|\int_0^T\int_\Omega \big(\eta_w(S_{M_0}) - \eta_w(\bar s)\big)\nabla\,\bar W_{\Rd w \Bk} \cdot \nabla\,v\big| \\
&\le \|\eta_w^\prime\|_{L^\infty(0,1)} \|S_{M_0}-\bar s\|_{L^2(\Omega \times ]0,T[)}
\|\nabla\, \bar W_{\Rd w \Bk}\|_{L^2(\Omega \times ]0,T[)} \|\nabla\,v \|_{ L^\infty(\Omega \times ]0,T[)} \le C\, \varepsilon.
\end{split}
\end{equation}
To relate the left-hand side of \eqref{eq:Pw+pg3M0} to
$\big[P_{w,h, \tau},I_h(\eta_w(S_{h,\tau})); U_{w,h,\tau} ,V_h\big]_h$, 
we split
$$
\eta_w((S_{M_0})_{w,\tau}^{n,ij}) =  \eta_w(S_w^{n,ij}) + \eta_w((S_{M_0})_{w,\tau}^{n,ij}) - \eta_w(S_w^{n,ij}),
$$
and examine the convergence of
$$
Y: = \frac{1}{2}\sum_{n=1}^{N} \tau\sum_{i,j=1}^M    c_{ij} \big(\eta_w((S_{M_0})_{w,\tau}^{n,ij}) - \eta_w(S_w^{n,ij})\big)\big(U_w^{n,j}- 
U_w^{n,i}\big)\big(V^{n,j} - V^{n,i}\big).
$$
By arguing as above and using the interpolant $I_h$, we derive
$$
|Y|  \le  C \|\eta_w^\prime\|_{L^\infty(0,1)} \|\nabla\,v\|_{L^\infty (\Omega \times ]0,T[)}\|\nabla\,U_{w,h,\tau}\|_{L^2(\Omega \times ]0,T[)} \|I_h((S_{M_0})_{\tau}) - S_{h,\tau}\|_{L^2(\Omega\times ]0,T[)}.
$$
Finally, we write
\begin{align*}
\|I_h((S_{M_0})_{\tau}) &- S_{h,\tau}\|_{L^2(\Omega\times ]0,T[)} \le \|I_h((S_{M_0})_{\tau}) - (S_{M_0})_{\tau}\|_{L^2(\Omega\times ]0,T[)} \\
&+ \|(S_{M_0})_{\tau} -S_{M_0}\|_{L^2(\Omega\times ]0,T[)} + 
\|S_{M_0}-\bar s\|_{L^2(\Omega\times ]0,T[)} + \|\bar s-S_{h,\tau}\|_{L^2(\Omega\times ]0,T[)}\\
& \le C\,h \|S_{M_0}\|_{L^\infty(0,T;H^2(\Omega))} + C \tau \|S_{M_0}\|_{H^1(0,T;L^2(\Omega))} + \varepsilon + \|\bar s-S_{h,\tau}\|_{L^2(\Omega\times ]0,T[)},
\end{align*}
so that
\begin{equation}
\label{eq:Pw+pg4M0}
|Y| \le C(h + \tau + \varepsilon) + \|\bar s-S_{h,\tau}\|_{L^2(\Omega\times ]0,T[)}.
\end{equation}

In the next theorem, the limit \eqref{eq:Pw+pg5} when $\bar s$ is only in $L^2(\Omega\times ]0,T[)$ follows by combining  \eqref{eq:meanetaSw3}--\eqref{eq:Pw+pg4M0}.  The same argument holds when $w$ is replaced by $o$. 

\begin{theorem}
\label{thm:limp+pwg}
Let $v \in {\mathcal C}^1(\bar \Omega \times [0,T])$ be a smooth function and let $V_{h,\tau}= I_h(v)(t_{n})$ in  $]t_{n-1},t_n]$. Under the assumptions and notation on the mobility \eqref{eq:eta'w}--\eqref{eq:etao_1}, 
\begin{equation}
\label{eq:limPw+pg}
\lim_{(h,\tau) \to (0,0)}  -\int_0^T\big[P_{\alpha,h, \tau},I_h(\eta_\alpha(S_{h,\tau})); U_{\alpha,h, \tau} ,V_{h,\tau}\big]_h 
 = \int_0^T\int_\Omega \eta_\alpha(\bar s) \nabla\,\bar W_{\Rd \alpha \Bk} \cdot \nabla\,v
\end{equation}
where $\bar s$ is the strong limit of $S_{h,\tau}$ and $\bar W_{\Rd \alpha \Bk}$ the weak limit of $U_{\alpha,h,\tau}$, $\alpha = w,o$.
\end{theorem}

%---------------- 

\subsubsection{\Rd The additional term \Bk}

This paragraph is dedicated to the limit of 
$$\int_0^T\big[P_{\alpha, h, \tau},I_h(\eta_{\alpha}(S_{h,\tau})); I_h(p_{\alpha g}(S_{h,\tau})) ,V_{h,\tau}\big]_h,\  \alpha = w,o.$$
It shall be split below, as suggested by the following observation, derived from \eqref{eq:defpwgpog}  and \eqref{eq:g}:
\begin{align*}
\eta_{w}(S_{w}^{ij})p_{w g}(S^{j}) + g(S^{j}) =& \int_0^{S^{j}}  f_o(x) \Bk \big(\eta_w(S_{w}^{ij}) - \eta_w(x)\big)p_c^\prime(x)\,dx,\\
 \eta_{o}(S_{o}^{ij})p_{o g}(S^{j}) + g(S^{j}) =& \int_0^{S^{j}}  f_w(x)  \big(\eta_o(S_{o}^{ij}) - \eta_o(x)\big)p_c^\prime(x)\,dx.
\end{align*}
  Thus, we add and subtract $g$ \Bk and write by applying \eqref{eq:nablauv},
\begin{align*}
\int_0^T\big[&P_{\alpha, h, \tau},I_h(\eta_{\alpha}(S_{h,\tau})); I_h(p_{\alpha g}(S_{h,\tau})) ,V_{h,\tau}\big]_h  \\
&= 
\sum_{n=1}^{N} \tau\sum_{i,j=1}^M V^{n,i}   c_{ij}\Big[\eta_{\alpha \Bk }(S_\alpha^{n,ij})\big(p_{ \alpha \Bk g}(S^{n,j}) - p_{ \alpha \Bk g}(S^{n,i})\big) + g(S^{n,j}) - g(S^{n,i})\Big]\\
& + \int_0^T \int_\Omega \nabla \,g(S_{h,\tau}) \cdot \nabla \,V_{h,\tau} = T_1 + T_2.
\end{align*}
Since
\begin{equation}
\label{eq:pwg1}
\lim_{(h,\tau) \to (0,0)}  T_2 = \int_0^T \int_\Omega \nabla \,g(\bar s) \cdot \nabla \,v,
\end{equation}
we must prove that the first term tends to zero. When $\alpha =w$,  it  
has the form
\begin{equation}
\label{eq:defT1}
T_1 = - \frac{1}{2} \sum_{n=1}^{N} \tau \sum_{i,j=1}^M\  c_{ij}\Big(\int_{S^{n,i}}^{S^{n,j}} f_o(x)\big(\eta_w(S_{w}^{n,ij}) - \eta_w(x)\big)p_c^\prime(x)\,dx\Big) \big(V^{n,j}-V^{n,i}\big),
\end{equation}
with an analogous expression in the non-wetting phase. Then \eqref{eq:deltavij} yields,
\begin{equation}
\label{eq:T2h}
|T_1 | \le \frac{C}{2}\|\nabla\,v\|_{L^\infty(\Omega \times ]0,T[)} \sum_{n=1}^{N} \tau \sum_{i,j=1}^M h_i c_{ij} \big|\int_{S^{n,i}}^{S^{n,j}} f_o(x)\big(\eta_w(S_{w}^{n,ij}) - \eta_w(x)\big)p_c^\prime(x)\,dx\big|.
\end{equation}
Showing that $T_1$ is small requires a technical argument that we split into several steps.

\begin{proposition}
\label{pro:T2hpro1}
For the wetting phase, we have
\begin{equation}
\label{eq:T2h1}
|\int_{S^{i}}^{S^{j}} f_o(x)\big(\eta_w(S_{w}^{ij}) - \eta_w(x)\big)p_c^\prime(x)\,dx\big|
 \le - \big(\eta_w(S^{j}) - \eta_w(S^{i})\big)\big(p_{wg}(S^{j})- p_{wg}(S^{i})\big).
\end{equation}
For the non-wetting phase, the corresponding expression is bounded by 
\begin{equation}
\label{eq:T2h11}
|\int_{S^{i}}^{S^{j}} f_w(x)\big(\eta_o(S_{o}^{ij}) - \eta_o(x)\big)p_c^\prime(x)\,dx\big|
 \le  \big(\eta_o(S^{j}) - \eta_o(S^{i})\big)\big(p_{og}(S^{j})- p_{og}(S^{i})\big).
\end{equation}
\Bk
\end{proposition}

\begin{proof}
Let us prove \eqref{eq:T2h1}, the proof of \eqref{eq:T2h11} being similar. \Bk The discussion depends on the respective values of $S^{j}$ and $S^{i}$. There are two cases: $S^{i} < S^{j}$ or $S^{i} > S^{j}$. Of course $S^{i} = S^{j}$ brings nothing. 

1)\ If $S^{i} < S^{j}$ and $S_{w}^{ij} = S^{i}$, then
$\eta_w(S_{w}^{ij}) - \eta_w(x) = \eta_w(S^{i}) - \eta_w(x)$,
and, as $p_{wg}$ is decreasing,
$$ 
0\le \int_{S^{i}}^{S^{j}} f_o(x)(- p_c^\prime(x)) \big(\eta_w(x) - \eta_w(S^{i})\big)\,dx
\le - \big(\eta_w(S^{j}) - \eta_w(S^{i})\big)\big(p_{wg}(S^{j})- p_{wg}(S^{i})\big).
$$
If $S_{w}^{ij} = S^{j}$, then
$\eta_w(S_{w}^{ij}) - \eta_w(x) = \eta_w(S^{j}) - \eta_w(x)$, and
$$
0 \le \int_{S^{i}}^{S^{j}} f_o(x)(-p_c^\prime(x)) \big(\eta_w(S^{j}) -\eta_w(x)\big)\,dx
\le - \big(\eta_w(S^{j}) - \eta_w(S^{i})\big)\big(p_{wg}(S^{j})- p_{wg}(S^{i})\big).
$$

2)\ If $S^{i} > S^{j}$ and $S_{w}^{ij} = S^{i}$, then
$$
0 \le \int_{S^{j}}^{S^{i}} f_o(x)(-p_c^\prime(x)) \big(\eta_w(S^{i})-\eta_w(x)\big)\,dx
\le - \big(\eta_w(S^{i}) - \eta_w(S^{j})\big)\big(p_{wg}(S^{i})- p_{wg}(S^{j})\big).
$$
Finally, suppose that $S_{w}^{ij} = S^{j}$. Then
$$
0 \le \int_{S^{j}}^{S^{i}} f_o(x)\,p_c^\prime(x) \big(\eta_w(S^{j})-\eta_w(x)\big)\,dx
\le - \big(\eta_w(S^{i}) - \eta_w(S^{j})\big)\big(p_{wg}(S^{i})- p_{wg}(S^{j})\big).
$$
This proves \eqref{eq:T2h1}.
\end{proof}

By substituting \eqref{eq:T2h1} into \eqref{eq:T2h}, we arrive at
\begin{equation}
\label{eq:T2h2}
|T_1 | \le \frac{C}{2}\|\nabla\,v\|_{L^\infty(\Omega \times ]0,T[)} 
 \sum_{n=1}^{N} \tau \sum_{i,j=1}^M  h_i c_{ij}\Big(- \big(\eta_w(S^{n,j}) - \eta_w(S^{n,i})\big)\big(p_{wg}(S^{n,j})- p_{wg}(S^{n,i})\big)\Big),
\end{equation}
with an analogous bound in the non-wetting phase.  Up to the factor $h_i$, they behave  like $\int_0^T \int_\Omega \nabla(I_h(\eta_{ \alpha } (S_{h,\tau}))) \cdot \nabla(I_h( p_{ \alpha  g} (S_{h,\tau})))$, $\alpha = w,o.$ 
Thus $T_1$  tends to zero if this quantity is bounded or is of the order of a negative power of $h$ that is larger than $-1$. We have no direct bound for it, but as we do have a bound for 
$\int_0^T \int_\Omega \nabla(I_h(f_\alpha(S_{h,\tau}))) \cdot \nabla(I_h(g(S_{h,\tau})))$,
see \eqref{eq:Bddgrad.g1}, we can gain some insight by relating the two integrands. 
Again, we examine the wetting phase, the treatment of the non-wetting phase being the same. 
\Bk
The proposition below will be applied to $x_1 = S^{n,i}$ and $x_2 = S^{n,j}$. The condition $x_1 <x_2$ is not a restriction because if it does not hold, the indices $i$ and $j$ can be interchanged without changing the value of the two integrands.

\begin{proposition}
\label{pro:etawpwg1}
 Under the assumptions and notation on the mobility \eqref{eq:eta'w}--\eqref{eq:etao_1}, we have for all pairs $x_1$, $x_2$ with $0 \le x_1 < x_2 \le \frac{3}{4}$, 
\begin{equation}
\label{eq:etawpwg2}
\big(\eta_w(x_2) - \eta_w(x_1)\big) \big(p_{wg}(x_1) - p_{wg}(x_2)\big) \le C ( x_2^{\theta_w} -  x_1^{\theta_w}) ( x_2^{\beta_3} -  x_1^{\beta_3}),
\end{equation}
\begin{equation}
\label{eq:fwg2}
\big(f_w(x_2) - f_w(x_1)\big) \big(g(x_2) - g(x_1)\big) \ge C ( x_2^{\theta_w} -  x_1^{\theta_w}) ( x_2^{\theta_w + \beta_3} -  x_1^{\theta_w +\beta_3}).
\end{equation}
Similarly, we have for all pairs $x_1$, $x_2$ with $\frac{1}{4} \le x_1 < x_2 \le 1$, 
\begin{equation}
\label{eq:etawpwg3}
\big(\eta_w(x_2) - \eta_w(x_1)\big) \big(p_{wg}(x_1) - p_{wg}(x_2)\big) \le C (x_2-x_1)\big( (1-x_1)^{\theta_o + \beta_4} - (1-x_2)^{\theta_o + \beta_4}\big),
\end{equation}
\begin{equation}
\label{eq:fwg3}
\big(f_w(x_2) - f_w(x_1)\big) \big(g(x_2) - g(x_1)\big) \ge C \big( (1-x_1)^{\theta_o} - (1-x_2)^{\theta_o}  \big) \big( (1-x_1)^{\theta_o + \beta_4} - (1-x_2)^{\theta_o + \beta_4}  \big).
\end{equation}
Finally, we have for all pairs $x_1$, $x_2$ with $0\le x_1 \le \frac{1}{4}$ and $\frac{3}{4}\le x_2 \le 1$,
\begin{equation}
\label{eq:etawpwg4}
\big(\eta_w(x_2) - \eta_w(x_1)\big) \big(p_{wg}(x_1) - p_{wg}(x_2)\big) \le C\big(f_w(x_2) - f_w(x_1)\big) \big(g(x_2) - g(x_1)\big).
\end{equation}
All constants $C$ above are independent of $x_1$ and $x_2$.

\end{proposition}

\begin{proof}
According to \eqref{eq:eta'w},
$$\eta_w(x_2) - \eta_w(x_1) \le \frac{1}{\alpha_w \theta_w} (x_2^{\theta_w} -  x_1^{\theta_w}).
$$
Next, recalling that $p_{wg}^\prime(x) = f_o(x)p_c^\prime(x)$, we have, owing to \eqref{eq:pc'}, \eqref{eq:etao_1}, and \eqref{eq:lowerboundetas},
\begin{equation}
\label{eq:deltpwg}
\begin{split}
p_{wg}(x_1) - p_{wg}(x_2) &= \int_{x_1}^{x_2} f_o(x)(-p_c^\prime(x))\,dx \le \frac{1}{\eta_\ast}\frac{1}{\alpha_3}\frac{1}{\alpha_o\theta_o}\int_{x_1}^{x_2} x^{\beta_3 -1}(1-x)^{\theta_o+ \beta_4-1}\, dx\\
& \le \frac{1}{\eta_\ast}\frac{1}{\alpha_3}\frac{1}{\alpha_o\theta_o}\int_{x_1}^{x_2} x^{\beta_3 -1}\,dx \le \frac{1}{\eta_\ast}\frac{1}{\alpha_3\beta_3}\frac{1}{\alpha_o\theta_o}(x_2^{\beta_3}-x_1^{\beta_3}),
\end{split}
\end{equation}
and \eqref{eq:etawpwg2}, valid on $[0,1]$, follows from these two inequalities.

For \eqref{eq:fwg2}, we use \eqref{eq:fwprime1} that gives
\begin{equation}
\label{eq:deltfw}
f_w(x_2) - f_w(x_1) \ge \frac{\alpha_o \alpha_w}{C_{\rm max}^2}\frac{1}{ \theta_o\theta_w} \big(\frac{1}{4}\big)^{\theta_o} (x_2^{\theta_w}- x_1^{\theta_w}),
\end{equation}
and we use \eqref{eq:lbgprime} that gives
$$g(x_2)-g(x_1) \ge \frac{\alpha_3}{C_{\rm max}}\frac{\alpha_w}{\theta_w}  \frac{\alpha_o}{\theta_o}\frac{1}{\theta_w + \beta_3} \big(\frac{1}{4}\big)^{\theta_o+ \beta_4 -1}(x_2^{\theta_w+ \beta_3}- x_1^{\theta_w+ \beta_3}).
$$
The product of the two leads to \eqref{eq:fwg2}.

Regarding \eqref{eq:etawpwg3},  \eqref{eq:etawpwg2}, albeit valid for all $x \in [0,1]$, is not adequate for the comparison we have in mind, and instead we use that 
$$\eta_w^\prime(x) \le \frac{1}{\alpha_w},
$$
which implies that
$$\eta_w(x_2) - \eta_w(x_1) \le \frac{1}{\alpha_w}(x_2 - x_1).
$$
Similarly, we use
$$
-p_{wg}^\prime(x) \le \frac{1}{\eta_\ast}\frac{1}{\alpha_3}\frac{1}{\alpha_o\theta_o} 4^{1-\beta_3} (1-x)^{\theta_o +\beta_4 -1},
$$
so that 
$$ p_{wg}(x_1) - p_{wg}(x_2) \le  \frac{1}{\eta_\ast}\frac{1}{\alpha_3}\frac{1}{\alpha_o\theta_o} \frac{1}{\theta_o + \beta_4}4^{1-\beta_3}\big((1-x_1)^{\theta_o+\beta_4}-(1-x_2)^{\theta_o+\beta_4}\big)),
$$
thus proving \eqref{eq:etawpwg3}. Next, by applying \eqref{eq:fwprime2}, we have
$$f_w(x_2)- f_w(x_1) \ge \frac{1}{C_{\rm max}^2}\frac{\alpha_o \alpha_w}{ \theta_w\theta_o} \big(\frac{1}{4}\big)^{\theta_w} \big((1-x_1)^{\theta_o}- (1-x_2)^{\theta_o}\big).
$$
Likewise, by applying \eqref{eq:lbgprime}, we obtain
$$g(x_2) - g(x_1) \ge \frac{\alpha_3}{C_{\rm max}}\frac{\alpha_w}{\theta_w}  \frac{\alpha_o}{\theta_o}\frac{1}{\theta_o +\beta_4}\big(\frac{1}{4}\big)^{\theta_w-1+\beta_3} \big((1-x_1)^{\theta_o+\beta_4}- (1-x_2)^{\theta_o+\beta_4}\big).
$$
The product of the two yields \eqref{eq:fwg3}.

Finally, when $0\le x_1 \le \frac{1}{4}$ and $\frac{3}{4}\le x_2 \le 1$, since both $\eta_w$ and $-p_{wg}$ are both increasing, they satisfy
$$\big(\eta_w(x_2) - \eta_w(x_1)\big)\big(p_{wg}(x_1) -p_{wg}(x_2)\big) \le \eta_w(1)\big(-p_{wg}(1)\big) >0.
$$
Likewise, as both $f_w$ and $g$ are increasing, they satisfy
$$\big(f_w(x_2) - f_w(x_1)\big)\big(g(x_2) -g(x_1)\big) \ge \big(f_w(\frac{3}{4}) - f_w(\frac{1}{4})\big)\big(g(\frac{3}{4}) -g(\frac{1}{4})\big)= :D >0.
$$
Hence
$$\big(\eta_w(x_2) - \eta_w(x_1)\big)\big(p_{wg}(x_1) -p_{wg}(x_2)\big) \le -\frac{1}{D}
\big(\eta_w\,p_{wg}\big)(1) \big(f_w(x_2) - f_w(x_1)\big)\big(g(x_2) -g(x_1)\big),
$$
whence \eqref{eq:etawpwg4}. Clearly all constants involved are independent of $x_1$ and $x_2$.
\end{proof}

It stems from \eqref{eq:etawpwg2} and \eqref{eq:fwg2}, that the two left-hand sides cannot be compared when $x_1$ and $x_2$ are too small. The same observation applies to \eqref{eq:etawpwg3} and 
\eqref{eq:fwg3} when $1-x_1$ and $1-x_2$ are too small. But in this case, there is no need for comparison because the expression we want to bound is sufficiently small, as is shown in the 
next proposition where again, $x_1 = S^{n,i}$ and $x_2 = S^{n,j}$.

\begin{proposition}
\label{pro:x2small}
Suppose that $x_1<x_2 \le h_i^{\gamma_1}$ for some exponent $\gamma_1 >0$ such that
\begin{equation}
\label{eq:gamm1} 
\gamma_1 >\frac{1}{\theta_w + \beta_3}.
\end{equation}
Then
\begin{equation}
\label{eq:x2small1}
h_i \big(\eta_w(x_2) - \eta_w(x_1)\big) \big(p_{wg}(x_1) - p_{wg}(x_2)\big) \le C\,h_i^2 h_i^{\delta_1},
\end{equation}
where 
\begin{equation}
\label{eq:delta1}
0<\delta_1 \le \gamma_1(\theta_w + \beta_3)-1.
\end{equation}
Similarly, suppose that $1-x_2<1-x_1 \le h_i^{\gamma_2}$ for some exponent $\gamma_2 >0$ such that
\begin{equation}
\label{eq:gamm2} 
\gamma_2 >\frac{1}{1 + \theta_o + \beta_4}.
\end{equation}
Then
\begin{equation}
\label{eq:x2small2}
h_i \big(\eta_w(x_2) - \eta_w(x_1)\big) \big(p_{wg}(x_1) - p_{wg}(x_2)\big) \le C\,h_i^2 h_i^{\delta_2},
\end{equation}
where 
\begin{equation}
\label{eq:delta2}
0<\delta_2 \le \gamma_2(1+\theta_o + \beta_4)-1.
\end{equation}
In both cases, the constants $C$ are independent of $x_1$, $x_2$, and $h_i$.
\end{proposition}

\begin{proof}
In the first case, according to \eqref{eq:etawpwg2}, the choice \eqref{eq:delta1} and  \eqref{eq:gamm1} on $\gamma_1$, we have 
$$h_i \big(\eta_w(x_2) - \eta_w(x_1)\big) \big(p_{wg}(x_1) - p_{wg}(x_2)\big) \le C\,h_i^{1 + \gamma_1(\theta_w + \beta_3)} = h_i^2 h_i^{\gamma_1(\theta_w + \beta_3)-1},
$$
with the constant $C$ of \eqref{eq:etawpwg2}, which gives  \eqref{eq:x2small1}. In the second case, the same argument leads to
$$h_i \big(\eta_w(x_2) - \eta_w(x_1)\big) \big(p_{wg}(x_1) - p_{wg}(x_2)\big) \le h_i^{1 + \gamma_2(1+\theta_o + \beta_4)} = h_i^2 h_i^{\gamma_2(1+\theta_o + \beta_4)-1},
$$
with the constant $C$ of \eqref{eq:etawpwg3}, thus implying \eqref{eq:x2small2} with the choice \eqref{eq:delta2} for $\delta_2$ and the condition \eqref{eq:gamm2} on $\gamma_2$.
\end{proof}

Now, we turn to the case when $x_2$ is not too small.

\begin{proposition}
\label{pro:x2notsmall}
In addition to \eqref{eq:gamm1}, suppose that the exponent $\gamma_1$ of Proposition \ref{pro:x2small} satisfies
\begin{equation}
\label{eq:gamm11} 
\gamma_1 <\frac{1}{\theta_w}.
\end{equation}
Suppose that $x_1 < x_2$ and $ \frac{3}{4} \ge x_2 > h_i^{\gamma_1}$. 
Then
\begin{equation}
\label{eq:x2notsmall1}
h_i \big(\eta_w(x_2) - \eta_w(x_1)\big) \big(p_{wg}(x_1) - p_{wg}(x_2)\big) 
\le C h_i^{\delta_1^\prime} \big(f_w(x_2) - f_w(x_1)\big)\big(g(x_2)- g(x_1)\big),
\end{equation}
 where
\begin{equation}
\label{eq:delta11prime}
0<\delta_1^\prime = {\rm min}\big( 1- \gamma_1\theta_w, \delta_1\big).
\end{equation}
Again, the constant $C$ is independent of $x_1$, $x_2$, and $h_i$.
\end{proposition}

\begin{proof}
Either $x_1 \le \frac{1}{2}  x_2$ or $x_1 > \frac{1}{2}  x_2$, and we examine each case.

1) When $x_1 \le \frac{1}{2}  x_2$, formula \eqref{eq:fwg2} leads to
$$\big(f_w(x_2) - f_w(x_1)\big)\big(g(x_2)- g(x_1)\big) \ge 
\big(1-(\frac{1}{2})^{\theta_w}\big) \big(1-(\frac{1}{2})^{\theta_w+ \beta_3}\big) 
C x_2^{2 \theta_w + \beta_3},
$$
with the constant $C$ of \eqref{eq:fwg2},
whereas
$$\big(\eta_w(x_2) - \eta_w(x_1)\big) \big(p_{wg}(x_1) - p_{wg}(x_2)\big) \le  C x_2^{\theta_w + \beta_3}, 
$$
with the constant $C$ of \eqref{eq:etawpwg2}. Hence
$$
 h_i \big(\eta_w(x_2) - \eta_w(x_1)\big)\big(p_{wg}(x_1) - p_{wg}(x_2)\big) \le C \frac{h_i}{x_2^{\theta_w}} 
 \big(f_w(x_2) - f_w(x_1)\big)\big(g(x_2)- g(x_1)\big),
$$
with another constant $C$ independent of $x_1$, $x_2$, and $h_i$. Now, we use the assumption that $x_2 > h_i^{\gamma_1}$. Then, owing to \eqref{eq:delta11prime},
$$\frac{h_i}{x_2^{\theta_w}} \le h_i^{1- \gamma_1 \theta_w} \le h_i^{\delta_1^\prime}, 
$$
and we recover \eqref{eq:x2notsmall1}.

2) When $x_1 > \frac{1}{2}  x_2$, we infer from the next to last inequality in  \eqref{eq:deltpwg} that 
$$
p_{wg}(x_1) - p_{wg}(x_2) \le \frac{1}{\eta_\ast}\frac{1}{\alpha_3}\frac{1}{\alpha_o\theta_o}  \big(x_2 - x_1\big)x_1^{\beta_3-1}
\le \frac{1}{\eta_\ast}\frac{1}{\alpha_3}\frac{1}{\alpha_o\theta_o} 2^{1-\beta_3}\frac{1}{x_2^{1-\beta_3}} \big(x_2 - x_1\big).
$$
Thus, on the one hand,
\begin{equation}
\label{eq:lhsx2notsmall}
\big(\eta_w(x_2) - \eta_w(x_1)\big) \big(p_{wg}(x_1) - p_{wg}(x_2)\big) \le C\,\frac{1}{x_2^{1-\beta_3}}
\big(x_2 - x_1\big) \big(x_2^{\theta_w} - x_1^{\theta_w}\big),
\end{equation} 
where $C$ is the above constant divided by $\alpha_w$. On the other hand, we use the lower bound \eqref{eq:deltfw} for the difference in $f_w$ and we need a lower bound for the difference in $g$. It is derived from \eqref{eq:lbgprime},
\begin{equation}
\label{eq:lowbddg}
\begin{split}
g(x_2) - g(x_1) &\ge \frac{\alpha_3}{C_{\rm max}}\frac{\alpha_w}{\theta_w}  \frac{\alpha_o}{\theta_o}\big(\frac{1}{4}\big)^{\theta_o+\beta_4-1}x_1^{\theta_w+ \beta_3-1}(x_2 - x_1)\\
& \ge \frac{\alpha_3}{C_{\rm max}}\frac{\alpha_w}{\theta_w}  \frac{\alpha_o}{\theta_o}\big(\frac{1}{4}\big)^{\theta_o+\beta_4-1}\big(\frac{1}{2}\big)^{\theta_w+\beta_3-1}x_2^{\theta_w+ \beta_3-1} (x_2 - x_1).
\end{split}
\end{equation}
Hence \eqref{eq:deltfw} and \eqref{eq:lowbddg} yield
\begin{equation}
\label{eq:rhsx2notsmall}
\big(f_w(x_2) - f_w(x_1)\big)\big(g(x_2) - g(x_1)\big) \ge C x_2^{\theta_w+ \beta_3-1}
 \big(x_2^{\theta_w} - x_1^{\theta_w}\big) (x_2 - x_1),
\end{equation} 
with the product of the constants of \eqref{eq:deltfw} and \eqref{eq:lowbddg}. Then by combining \eqref{eq:lhsx2notsmall} and \eqref{eq:rhsx2notsmall}, we deduce that
$$
\big(\eta_w(x_2) - \eta_w(x_1)\big) \big(p_{wg}(x_1) - p_{wg}(x_2)\big) \le \frac{C}{h_i^{\gamma_1 \theta_w}}
\big(f_w(x_2) - f_w(x_1)\big)\big(g(x_2) - g(x_1)\big);
$$
which is \eqref{eq:x2notsmall1} when $\delta_1^\prime$ satisfies \eqref{eq:delta11prime}. 
\end{proof}
%Of course, Proposition \ref{pro:x2small} holds with $\delta_1^\prime $ instead of $\delta_1$.

The case when $1-x_1$ is not too small is handled by the next proposition.

\begin{proposition}
\label{pro:1-x1notsmall}
In addition to \eqref{eq:gamm2}, suppose that the exponent $\gamma_2$ of Proposition \ref{pro:x2small} satisfies
\begin{equation}
\label{eq:gamm21} 
\gamma_2 <\frac{1}{\theta_o-1}.
\end{equation}
Suppose that $\frac{1}{4} < x_1 < x_2 \le 1$ and $1-x_1 > h_i^{\gamma_2}$.  
Then
\begin{equation}
\label{eq:1-x1notsmall1}
h_i \big(\eta_w(x_2) - \eta_w(x_1)\big) \big(p_{wg}(x_1) - p_{wg}(x_2)\big) 
\le C h_i^{\delta_2^\prime} \big(f_w(x_2) - f_w(x_1)\big)\big(g(x_2)- g(x_1)\big),
\end{equation}
 where
\begin{equation}
\label{eq:delta12prime}
0<\delta_2^\prime = {\rm min}\big(\delta_2, 1-\gamma_2(\theta_o-1)\big).
\end{equation}
Again, the constant $C$ is independent of $x_1$, $x_2$, and $h_i$.
\end{proposition}

\begin{proof}
The proof is analogous to that of Proposition \ref{pro:x2notsmall}, but we sketch the steps for the reader's convenience. We skip the constants' details, but stress that they are independent of $x_1$, $x_2$, and $h_i$. Again, there are two possibilities, 
either $1- x_2 \le \frac{1}{2} (1- x_1)$ or $1- x_2 > \frac{1}{2}  (1- x_1)$, and we examine each case.

1) In the first case,
$$\big(\eta_w(x_2) - \eta_w(x_1)\big) \big(p_{wg}(x_1) - p_{wg}(x_2)\big) \le C (1- x_1)^{1+\theta_0 + \beta_4},
$$ 
and 
$$\big(f_w(x_2) - f_w(x_1)\big)\big(g(x_2)- g(x_1)\big) \ge C (1- x_1)^{2\theta_0 + \beta_4}.
$$
Hence
\begin{align*}
 \big(\eta_w(x_2) - \eta_w(x_1)\big) \big(p_{wg}(x_1) - p_{wg}(x_2)\big) &\le C \frac{1}{(1-x_1)^{\theta_o-1}}  \big(f_w(x_2) - f_w(x_1)\big)\big(g(x_2)- g(x_1)\big)\\
 &\le C\frac{1}{h_i^{\gamma_2(\theta_o-1)}}\big(f_w(x_2) - f_w(x_1)\big)\big(g(x_2)- g(x_1)\big).
\end{align*}
With  \eqref{eq:gamm21} and  \eqref{eq:delta12prime}, this implies \eqref{eq:1-x1notsmall1}.

2) In the second case, we have on the one hand,
$$p_{wg}(x_1) - p_{wg}(x_2) \le C (x_2- x_1) (1- x_1)^{\theta_o + \beta_4 -1},
$$
so that
$$
\big(\eta_w(x_2) - \eta_w(x_1)\big) \big(p_{wg}(x_1) - p_{wg}(x_2)\big) \le C (x_2- x_1)^2(1- x_1)^{\theta_o + \beta_4 -1}.
$$
On the other hand,
$$
f_w(x_2) - f_w(x_1) \ge C (x_2- x_1) (1- x_1)^{\theta_o-1},
$$
and
$$
g(x_2)- g(x_1) \ge C (x_2- x_1)(1- x_1)^{\theta_o + \beta_4-1},
$$
and thus
\begin{align*}
 \big(\eta_w(x_2) - \eta_w(x_1)\big) \big(p_{wg}(x_1) - p_{wg}(x_2)\big) &\le C \frac{1}{(1-x_1)^{\theta_o-1}} 
 \big(f_w(x_2) - f_w(x_1)\big)\big(g(x_2)- g(x_1)\big)\\
 &\le C\frac{1}{h_i^{\gamma_2(\theta_o-1)}}\big(f_w(x_2) - f_w(x_1)\big)\big(g(x_2)- g(x_1)\big),
\end{align*}
whence \eqref{eq:1-x1notsmall1}.
\end{proof}
%Again, Proposition \eqref{pro:1-x1notsmall} holds with $\delta_2^\prime $ instead of $\delta_2$.

In view of  \eqref{eq:gamm1}, \eqref{eq:delta1}, \eqref{eq:gamm11}, and  
\eqref{eq:delta11prime}, let us choose
\begin{equation}
\label{eq:choicedelt1}
\delta_1 = \delta_1^\prime= \frac{\beta_3}{2 \theta_w + \beta_3},\quad \gamma_1 = \frac{2}{2 \theta_w + \beta_3}.
\end{equation}
Then \eqref{eq:gamm1} and \eqref{eq:delta1} are satisfied, as well as \eqref{eq:gamm11} and 
\eqref{eq:delta11prime}. Likewise, in view of \eqref{eq:gamm2}, \eqref{eq:delta2}, \eqref{eq:gamm21}, 
and \eqref{eq:delta12prime}, the choice
\begin{equation}
\label{eq:choicedelt2}
\delta_2 = \delta_2^\prime= \frac{2+\beta_4}{2 \theta_o + \beta_4},\quad \gamma_2 = \frac{2}{2 \theta_o + \beta_4},
\end{equation}
satisfies \eqref{eq:gamm2}, \eqref{eq:delta2}, \eqref{eq:gamm21}, \eqref{eq:delta12prime}. Then the desired limit follows by collecting these results. 

\begin{lemma}
\label{lem:T1to0}
Under the assumptions and notation on the mobility \eqref{eq:eta'w}--\eqref{eq:etao_1}, the term $T_1$   defined in \eqref{eq:defT1}  tends to zero, with a similar limit in the non-wetting phase, i.e.,
\begin{equation}
\label{eq:limT1}
\begin{split}
&\lim_{(h,\tau) \to (0,0)}  \sum_{n=1}^{N} \tau \sum_{i,j=1}^M  c_{ij}\Big(\int_{S^{n,i}}^{S^{n,j}} f_o(x)\big(\eta_w(S_{w}^{n,ij}) - \eta_w(x)\big)p_c^\prime(x)\,dx\Big)
  \big(V^{n,j}-V^{n,i}\big) = 0,\\
&\lim_{(h,\tau) \to (0,0)}  \sum_{n=1}^{N} \tau \sum_{i,j=1}^M  c_{ij}\Big(\int_{S^{n,i}}^{S^{n,j}} f_w(x)\big(\eta_o(S_{o}^{n,ij}) - \eta_o(x)\big)p_c^\prime(x)\,dx\Big)
 \big(V^{n,j}-V^{n,i}\big) = 0.
\end{split}
\end{equation}
\end{lemma}

\begin{proof}
 We prove the first limit. \Bk Here the parameters of Propositions \ref{pro:x2small} and \ref{pro:x2notsmall}  are chosen by \eqref{eq:choicedelt1} and \eqref{eq:choicedelt2}.
It stems from the above considerations that, for each index $n$, the set of all indices $(i,j)$ from $1$ to $M$ can be partitioned into three subsets,
$${\mathcal O}_1 = \{(i,j)\,;\, 0 \le S^{n,i} < S^{n,j}  \le \frac{3}{4} \},\quad 
{\mathcal O}_2 = \{(i,j)\,;\,  \frac{1}{4} \le S^{n,i} < S^{n,j}\le 1\},
$$
$${\mathcal O}_3 = \{(i,j)\,;\, 0 \le S^{n,i} \le \frac{1}{4} \ \mbox{and}\ \frac{3}{4} \le S^{n,j} \le 1\}.$$
In turn, ${\mathcal O}_1$ and ${\mathcal O}_2$ can each be partitioned into two subsets
$$
{\mathcal O}_{1,1} = \{(i,j) \in {\mathcal O}_1\,;\, S^{n,j} \le h_T^{\gamma_1} \}, \quad 
{\mathcal O}_{1,2} = \{(i,j) \in {\mathcal O}_1\,;\, S^{n,j} > h_T^{\gamma_1} \},
$$
$$
{\mathcal O}_{2,1} = \{(i,j) \in {\mathcal O}_2\,;\, 1- S^{n,i} \le h_T^{\gamma_2} \},\quad
{\mathcal O}_{2,2} = \{(i,j) \in {\mathcal O}_2\,;\, 1- S^{n,i} > h_T^{\gamma_2} \}.
$$
To simplify, let 
$$A_{i,j} =  c_{ij}\Big(\int_{S^{n,i}}^{S^{n,j}} f_o(x)\big(\eta_w(S_{w}^{n,ij}) - \eta_w(x)\big)p_c^\prime(x)\,dx\Big) \big(V^{n,j}-V^{n,i}\big).
$$
In view of \eqref{eq:x2small1} and \eqref{eq:x2small2},  for all pairs $(i,j)$ in ${\mathcal O}_{\ell,1}$, $\ell = 1,2$, $A_{i,j}$ satisfies
$$|A_{i,j}| \le C \|\nabla\,v\|_{L^\infty(\Omega \times ]0,T[)} h_i^{2+ \delta_\ell}c_{ij}.$$
Owing to \eqref{eq:x2notsmall1} and \eqref{eq:1-x1notsmall1}, 
for all pairs $(i,j)$ in ${\mathcal O}_{\ell,2}$, $\ell = 1,2$, we have
$$|A_{i,j}| \le C \|\nabla\,v\|_{L^\infty(\Omega \times ]0,T[)} h_i^{\delta_\ell}c_{ij}\big(f_w(S^{n,j}) - f_w(S^{n,i})\big)\big(g(S^{n,j})- g(S^{n,i})\big).$$
Finally, for all pairs $(i,j)$ in ${\mathcal O}_3$,
$$|A_{i,j}| \le C \|\nabla\,v\|_{L^\infty(\Omega \times ]0,T[)} h_i c_{ij}\big(f_w(S^{n,j}) - f_w(S^{n,i})\big)\big(g(S^{n,j})- g(S^{n,i})\big).$$
According to \eqref{eq:Bddgrad.g1}, the sum of the terms over all $(i,j)$ in ${\mathcal O}_{\ell,2}$ and 
${\mathcal O}_3$ tends to zero. For the remaining terms,
 observe that by definition, 
$$h_i^2 c_{ij} \le C |\Delta_i \cap \Delta_j|,
$$
so that the sum over all $(i,j)$ in ${\mathcal O}_{\ell,1}$ is bounded by $C h_i^{\delta_\ell}$ that also tends to zero, whence the  first part of the \Bk limit \eqref{eq:limT1}.  The same limit to zero holds for the non-wetting phase. \Bk
\end{proof}

With \eqref{eq:pwg1}, this lemma leads  to the desired limit of the term with the auxiliary pressures.

\begin{theorem}
\label{thm:limpwg}
Let $v \in {\mathcal C}^1(\bar \Omega \times [0,T])$ be a smooth function and let $V_{h,\tau}(t) = I_h(v)(t_{n})$ in  $]t_{n-1},t_n]$. Under the assumptions and notation on the mobility \eqref{eq:eta'w}--\eqref{eq:etao_1}, 
\begin{equation}
\lim_{(h,\tau) \to (0,0)} \int_0^T\big[P_{ \alpha \Bk ,h, \tau},I_h(\eta_{ \alpha \Bk}(S_{h,\tau})); I_h(p_{ \alpha \Bk g}(S_{h,\tau})) ,V_{h,\tau}\big]_h = \int_0^T \int_\Omega \nabla \,g(\bar s) \cdot \nabla \,v,\  \alpha = w,o,
\end{equation}
where $\bar s$ is the limit of $S_{h,\tau}$.
\end{theorem}

Finally, Theorems \ref{thm:limp+pwg} and \ref{thm:limpwg}, together with \eqref{eq:pwgprime+gprime} and  \eqref{eq:splitW}, give the desired convergence of the upwind diffusion terms. 

\begin{theorem}
\label{thm:limupwd}
With the notation and assumptions of Theorem \ref{thm:limp+pwg}, we have for all functions $v \in {\mathcal C}^1(\bar \Omega \times [0,T])$,
\begin{equation}
\label{eq:limupwind}
\begin{split}
\lim_{(h,\tau) \to (0,0)}& -\int_0^T \big[P_{ \alpha  ,h, \tau},I_h(\eta_{ \alpha }(S_{h,\tau})); P_{ \alpha,h, \tau},V_{h,\tau}\big]_h \\
& = \int_0^T\int_\Omega \big(\eta_{w}(\bar s) \nabla(\bar p_w + p_{wg}(\bar s)) + \nabla\, g(\bar s) \big) \cdot \nabla\,v \quad \mbox{if }\  \alpha = w,\\
& = \int_0^T\int_\Omega \big(\eta_{o}(\bar s) \nabla(\bar p_o - p_{og}(\bar s)) - \nabla\, g(\bar s) \big) \cdot \nabla\,v\quad \mbox{if }\  \alpha = o.
\end{split}
\end{equation}
\end{theorem}

%==========================

\subsection{Convergence of the right-hand sides}

\label{sec:right-handside}

In order to pass to the limit in the right-hand sides of \eqref{eq:vars1}--\eqref{eq:vars2} \Bk it is convenient to replace the quadrature formulas by integrals. Since the quadrature formulas are exact for polynomials of degree one, this is achieved by approximating some functions with the operator $\rho_h$, see \eqref{eq:rhoh}.  
As $s_{\mathrm{in}}$ belongs to $L^\infty(\Omega \times ]0,T[)$, standard approximation properties of $\rho_\tau$ and $r_h$ and a density argument imply 
%\begin{equation}
%\label{eq:lim-sin1}
%\lim_{(h,\tau) \to (0,0)} \rho_\tau(r_h(s_{\mathrm{in}})) = s_{\mathrm{in}} \ \mbox{in}\; L^\infty(\Omega \times ]0,T[),
%\end{equation}
\begin{equation}
\label{eq:lim-sin2}
\lim_{(h,\tau) \to (0,0)} \rho_\tau(\rho_K(s_{\mathrm{in}})) = s_{\mathrm{in}} \ \mbox{in}\; L^\infty(K \times ]0,T[).
\end{equation}
Then the continuity of $f_{\Rd \alpha \Bk}$,  for $\alpha = w,o$, \Bk yields
%\begin{equation}
%\label{eq:lim-fosin1}
%\lim_{(h,\tau) \to (0,0)} f_{\Rd \alpha \Bk}(\rho_\tau(r_h(s_{\mathrm{in}}))) = f_{\Rd \alpha \Bk}(s_{\mathrm{in}})  \ \mbox{in}\; L^\infty(\Omega \times ]0,T[),
%\end{equation}
\begin{equation}
\label{eq:lim-fosin2}
\lim_{(h,\tau) \to (0,0)} f_{ \alpha \Bk}(\rho_\tau(\rho_K(s_{\mathrm{in}}))) = f_{ \alpha \Bk}(s_{\mathrm{in}}) \ \mbox{in}\; L^\infty(K \times ]0,T[).
\end{equation}
Similarly, since $\bar{q}$ belongs to $L^2(\Omega \times ]0,T[)$,
$$
\lim_{(h,\tau) \to (0,0)} \rho_\tau(\rho_K(\bar{q})) = \bar q \ \mbox{in}\; L^2(K \times ]0,T[).
$$
Also the (constant in space) correction added to $\rho_\tau(r_h(\bar{q}))$ satisfies
$$\lim_{(h,\tau) \to (0,0)} \rho_\tau\big(\frac{1}{|\Omega|}\int_\Omega (r_h(\bar q) - \bar q)\big) = 0 \ \mbox{in}\; L^2(\Omega \times ]0,T[).
$$
Therefore
\begin{equation}
\label{eq:lim-qht}
\lim_{(h,\tau) \to (0,0)} \bar q_{h,\tau} = \bar q \ \mbox{in}\; L^2(\Omega \times ]0,T[).
 \end{equation}
 With the same function $V_{h,\tau}$, consider the first term in the right-hand sides of 
 \eqref{eq:vars1}--\eqref{eq:vars2} \Bk
$$X := \sum_{n=1}^{N} \tau \big(I_h(f_{\alpha \Bk}(s_{\mathrm{in},h}^{n}))\bar{q}_h^{n},V_h^{n} \big)_h
= \int_0^T \big(I_h(f_{ \alpha \Bk}(s_{\mathrm{in},h, \tau}))\bar{q}_{h,\tau},V_{h,\tau} \big)_h.
$$
By definition of the quadrature formula, $X$ has the following expression:
$$X = \sum_{n=1}^{N} \tau \sum_{K \in \bar \Omega} \frac{|K|}{d+1} \sum_{\ell =1}^{d+1}f_{ \alpha \Bk}(s_{\mathrm{in,h,\tau}}^{n,\ell_i}) \bar q_{h,\tau}^{n,\ell_i}V_{h,\tau}^{n,\ell_i}.
$$
By inserting $f_{ \alpha \Bk}(\rho_\tau(\rho_K(s_{\mathrm{in}})))$ and $\rho_\tau(\rho_K(\bar q))$, this becomes
\begin{align*}
X =& \sum_{n=1}^{N} \tau \sum_{K \in \bar \Omega} \frac{|K|}{d+1} \sum_{\ell =1}^{d+1} \big(f_{ \alpha \Bk}(s_{\mathrm{in,h,\tau}}^{n,\ell_i}) - f_{ \alpha \Bk}(\rho_\tau(\rho_K(s_{\mathrm{in}})))\big)\bar q_{h,\tau}^{n,\ell_i}V_{h,\tau}^{n,\ell_i}\\
& + \sum_{n=1}^{N} \tau \sum_{K \in \bar \Omega} \frac{|K|}{d+1} \sum_{\ell =1}^{d+1} f_{ \alpha \Bk}(\rho_\tau(\rho_K(s_{\mathrm{in}})))\big(\bar q_{h,\tau}^{n,\ell_i}- \rho_\tau(\rho_K (\bar q))\big)V_{h,\tau}^{n,\ell_i}\\
&+ \int_0^T \int_\Omega f_{ \alpha \Bk}(\rho_\tau(\rho_K(s_{\mathrm{in}})))\rho_\tau (\rho_K (\bar q))V_{h,\tau} = X_1+X_2+X_3,
\end{align*}
since the last summand is a polynomial of degree one. We have
$$\lim_{(h,\tau) \to (0,0)} X_3 =\int_0^T \int_\Omega f_o(s_{\mathrm{in}})\,\bar q\, v.
$$
It remains to show that $X_1$ and $X_2$ tend to zero. For $X_1$, since $f_o$ and $f_w$ have the same derivative (up to the sign), we deduce from \eqref{eq:fwprime}, \eqref{eq:eta'w}, 
\eqref{eq:eta'o}, \eqref{eq:etaw_1}, \eqref{eq:etao_1}, and
\eqref{eq:bddsh} that $f_\alpha^\prime$ is bounded in $[0,1]$; hence
$$| f_{\alpha \Bk}(s_{\mathrm{in,h,\tau}}^{n,\ell_i}) - f_{\alpha \Bk}(\rho_\tau(\rho_K(s_{\mathrm{in}}))) |\le C |s_{\mathrm{in,h,\tau}}^{n,\ell_i} - \rho_\tau(\rho_K(s_{\mathrm{in}}))|.
$$
Thus, the summand is bounded by polynomials and the equivalence of norms yields
$$|X_1| \le C \|v\|_{L^\infty( \Omega \times ]0,T[)}\|s_{\mathrm{in,h,\tau}}- \rho_\tau(\rho_K(s_{\mathrm{in}}))\|_{L^2(\Omega \times ]0,T[)} \|\bar q_{h,\tau}\|_{L^2(\Omega \times ]0,T[)},
$$
that tends to zero with $(h,\tau)$. It is easy to check that the same holds for $X_2$.
Hence
\begin{equation}
\label{eq:lim-X}
\lim_{(h,\tau) \to (0,0)} \int_0^T \big(I_h(f_{ \alpha \Bk}(s_{\mathrm{in},h, \tau}))\bar{q}_{h,\tau},V_{h,\tau} \big)_h = \int_0^T \int_\Omega f_{\Rd \alpha \Bk}(s_{\mathrm{in}})\,\bar q\, v.
 \end{equation}

 The argument for the second term in the right-hand side of \eqref{eq:vars1} is much the same; we insert $\rho_\tau (\rho_K (\bar s))$ and we use the fact that
 $$ \lim_{(h,\tau) \to (0,0)} \|S_{h,\tau}- \rho_\tau (\rho_K (\bar s))\|_{L^2(\Omega \times ]0,T[)} =0.
 $$
 Then the argument used for the first term readily gives
 \begin{equation}
\label{eq:lim-XX}
\lim_{(h,\tau) \to (0,0)} \int_0^T ( I_h(f_{ \alpha \Bk}(S_{h,\tau}))\underline{q}_{h_\tau},V_{h,\tau}\big)_h = \int_0^T \int_\Omega f_{ \alpha \Bk}(\bar s)\,\bar q\, v.
 \end{equation}
By combining \eqref{eq:lim-X} and \eqref{eq:lim-XX}, we obtain convergence of the right-hand sides, \begin{equation}
\label{eq:lim-XXX}
\lim_{(h,\tau) \to (0,0)} \Big( \int_0^T \big(I_h(f_{\Rd \alpha \Bk}(s_{\mathrm{in},h, \tau}))\bar{q}_{h,\tau} - I_h(f_{ \alpha \Bk}(S_{h,\tau}))\underline{q}_{h_\tau},V_{h,\tau} \big)_h  =  \int_0^T \int_\Omega \big(f_{\alpha \Bk}(s_{\mathrm{in}})\,\bar q- f_{ \alpha \Bk}(\bar s)\,\bar q\big) v.
 \end{equation}
 
%==========================

\subsection{The full scheme}
\label{sec:full scheme}

It remains to pass to the limit in the time derivative, say in \eqref{eq:vars1}, summed over $n$, and tested with the same $V_{h,\tau}$ as previously, except that here we take $v(T) = 0$. After summation by parts, this term reads
\begin{equation}
\label{eq:sumparts}
\sum_{n=1}^N (S_h^n-S_h^{n-1},V_h^n)_h^\varphi = - \sum_{n=1}^{N-1} (V_h^{n+1} - V_h^n, S_h^n)_h^\varphi -(V_h^1, S_h^0)_h^\varphi.
 \end{equation}
By definition,
$$(V_h^{n+1} - V_h^n, S_h^n)_h^\varphi = \sum_{K \in \bar \Omega} \frac{|K|}{d+1} \varphi|_K \sum_{\ell = 1}^{d+1}(V^{n+1,i_\ell}- V^{n,i_\ell}) S^{n,i_\ell}.
$$
By inserting $\rho_K(V^{n+1,i_\ell}- V^{n,i_\ell})$ in each element, this becomes
$$
(V_h^{n+1} - V_h^n, S_h^n)_h^\varphi = (V_h^{n+1} - V_h^n-\rho_h(V^{n+1}-V^n), S_h^n)_h^\varphi + \int_\Omega \varphi\rho_h(V^{n+1}- V^{n}) S^{n}_h.
$$
The first term has the bound
$$\big|(V_h^{n+1} - V_h^n-\rho_h(V^{n+1}-V^n), S_h^n)_h^\varphi\big| \le \|\varphi\|_{L^\infty(\Omega)}\|
V_h^{n+1} - V_h^n-\rho_h(V^{n+1}-V^n)\|_h \|S_h^n\|_h.
$$
Since the functions are piecewise polynomials, the equivalence of norms yields
\begin{align*}
\big|\sum_{n=1}^{N-1} &(V_h^{n+1} - V_h^n-\rho_h(V^{n+1}-V^n), S_h^n)_h^\varphi\big| \le C\,\|\varphi\|_{L^\infty(\Omega)}\\
& \times\big(\sum_{n=1}^{N-1} \tau \|\frac{1}{\tau}\big(I_h( v^{n+1}-v^n) - \rho_h(V^{n+1}-V^n)\big)\|^2_{L^2(\Omega)}\big)^{\frac{1}{2}}
\big(\sum_{n=1}^{N-1} \tau \|S_h^n\|^2_{L^2(\Omega)}\big)^{\frac{1}{2}}.
\end{align*}
Then the regularity of $v$, the approximation properties of $I_h$ and $\rho_h$ and the boundedness of $S_{h,\tau}$ imply that
$$\lim_{(h,\tau) \to (0,0)}  \big|\sum_{n=1}^{N-1} (V_h^{n+1} - V_h^n-\rho_h(V^{n+1}-V^n), S_h^n)_h^\varphi\big| =0.
$$
Similarly, it is easy to check from the convergence of $S_{h,\tau}$ that
$$-\lim_{(h,\tau) \to (0,0)}\sum_{n=1}^{N-1}\int_\Omega \varphi\rho_h(V^{n+1}- V^{n}) S^{n}_h = -\int_0^T \int_\Omega \varphi (\partial_t v) \bar s.
$$
The treatment of the initial term is the same. Hence
\begin{equation}
\label{eq:limdtS}
\lim_{(h,\tau) \to (0,0)} \sum_{n=1}^N (S_h^n-S_h^{n-1},V_h^n)_h^\varphi = -\int_0^T \int_\Omega \varphi (\partial_t v) \bar s  - \int_\Omega \varphi s^0 v.
\end{equation}
By combining \eqref{eq:limdtS}, with \eqref{thm:limupwd} and \eqref{eq:lim-XXX}, we readily see that the limit functions $\bar s$, $\bar p_\alpha$ and $p_{\alpha g}(\bar s)$ satisfy the weak formulation 
\eqref{eq:Var}. This proves Theorem \ref{thm:mainresult}.

\section{Numerical validation}
\label{sec:Valid}

This section proposes a numerical validation of our algorithm
with a two dimensional finite difference code. Details on the algorithm implemented
are given. A problem with manufactured solutions is then considered to study the convergence
properties of our algorithm.

\subsection{Implementation of the model}
To avoid dealing with nonlinear terms, we implement a modified version
of the algorithm proposed in section~\ref{subsec:time-space}. 
The main difference consists of approximating the terms 
$S_w^{n+1,ij}$, $S_o^{n+1,ij}$ and $P_o^{n+1}$ with first time order extrapolation. 
For each node $1\leq i \leq M$, the unknowns $(S^{n+1,i}, P_w^{n+1,i})$
are computed as the solution of the following problem:
\begin{eqnarray}
\frac{\tilde{m}_i}{\Delta t} (S^{n+1,i} - S^{n,i}) 
-\sum_{j\neq i, j\in N(i)} c_{ij} \eta_w(S^{*,n+1,ij}_w)  (P_w^{n+1,j}-P_w^{n+1,i})
\nonumber\\
= m_i f_1^{n+1,i}, \quad 1\leq i\leq M,
\label{eq:num_scheme1}
\\
-\frac{\tilde{m}_i}{\Delta t} (S^{n+1,i} - S^{n,i}) 
-\sum_{j\neq i, j\in N(i)} c_{ij} \eta_o(S^{*, n+1,ij}_o) (P_w^{n+1,j}-P_w^{n+1,i})
\nonumber\\
-\sum_{j\neq i, j\in N(i)} c_{ij} \eta_o(S^{*, n+1,ij}_o) (p_c^{*,n+1,j}-p_c^{*,n+1,i})= m_i f_2^{n+1,i}, \quad 1\leq i\leq M,\label{eq:num_scheme2}
\end{eqnarray}
where the pressure $P_o$ has been substituted with $P_w + p_c$ 
with respect to \eqref{eq:scheme3}. The solution $P_w$ is 
enforced to satisfy \eqref{eq:scheme4} a posteriori by 
subtracting its integral $\sum_{i=1}^M  m_i P_w^i$ after solving the above problem.
The terms $S_w^{*, n+1,ij}$ and $S_o^{*, n+1,ij}$ are approximated 
at time iteration n by setting them to $S_w^{n,ij}$ and $S_o^{n,ij}$.
Eventually, the capillary pressure $p_c^{*,n+1}$ is approximated 
with a first order Taylor expansion with respect to the saturation $S$, it reads:
\begin{equation}
p_c^{*,n+1}= p_c^n + \left(\frac{\partial p_c}{\partial S}\right)^n (S^{n+1} - S^n).
\end{equation}
We note that to facilitate the implementation of this algorithm in 
a two dimensional finite difference code, the source terms of 
the equations~\eqref{eq:scheme1}-\eqref{eq:scheme2} have been 
replaced by functions denoted by $f_1$ and $f_2$.

\subsection{Numerical test with a manufactured solution}
The numerical validation of the algorithm is done by approximating the analytical
solutions defined by
\begin{equation}
\label{eq:num_test_solu_Pw}
 P_{w}(t,x,y)= 2 + x^2y - y^2 + x^2\sin(t+y),
\end{equation}
\begin{equation}
\label{eq:num_test_solu_S}
 S(t,x,y)= 0.2(2 + 2xy + \cos(t+x)),
\end{equation}
on the computational domain $\Omega=[0,1]^2$.
Dirichlet boundary conditions are applied on $\partial \Omega$
on both unknowns $P_w$ and $S$. The initial conditions of the problem
satisfy~\eqref{eq:num_test_solu_Pw}-\eqref{eq:num_test_solu_S}.
The porosity of the domain is set to:
\begin{equation}
 \phi(t,x,y)=  0.2(1+xy).
\end{equation}
The mobilities $\eta_w$ and $\eta_o$, introduced in section~\ref{sec:problem}, 
are defined as follows:
\begin{equation}
\eta_w = 4 S^2, \qquad \eta_o = 0.4  (1-S)^2.
 %K = 2,\qquad  k_{rw}=S^2, 
 %\qquad k_{ro}=(1-S)^2, \qquad \mu_w =0.5, \qquad\mu_o=5.
\end{equation}
The capillary pressure is based on the Brooks-Corey model, it reads:
\begin{equation}
   P_{c} =
  \begin{cases}
  A S^{-0.5} & \text{if $S>0.05$},\\
  A (1.5 - 10 S)\times 0.05^{-0.5} & \text{otherwise}.
  \end{cases}
\end{equation}
where $A$ is a constant set to $50$.
The term sources $f_1$ and $f_2$ are computed accordingly.
The convergence tests are performed on a set of six uniform grids 
with respective mesh size 
$h\in\{0.2, 0.1, 0.05, 0.025, 0.0125, 0.00625\}$. 
The convergence properties are evaluated by using a time step 
$\tau$ set to the mesh size $h$ with a final time $T=1$.
As the time derivatives and the saturations $S_w^{n+1, ij}, S_o^{n+1,ij}$
are computed with first order time approximation, 
we expect the convergence rate in the $L^2$ norm to be of order one.
%OLD ERROR COMPUTED WITH QUADRATURE FORMULA of PAPER (sum_i m_i f_i) so projection on P1
% \begin{table}[ht]
% \centering
% \begin{tabular}{|c|c|c|c|c|c|} \hline  
%  \multicolumn{2}{|c|}{$L^2$-norm of error}   & \multicolumn{2}{|c|}{Water pressure $P_w$}
%  & \multicolumn{2}{|c|}{Water saturation $S$}  \\ \hline \hline
%   $h$  & $n_{df}$ & Error  & Rate  & Error  & Rate     \\ \hline
%   $ 0.2$& 25 & 9.40E-3  & - & 4.56E-3  &  -      \\ \hline
% $0.1$ & 100 & 4.31E-3 & 1.12  &  2.34E-3 & 0.96  \\  \hline
% $0.05$ & 400 &2.10E-3 & 1.04 &  1.14E-4 &  1.02 \\  \hline
% $0.025$ & 1600 & 1.05E-3& 1.00  &  5.57E-4 & 1.03 \\  \hline
% $0.0125$ & 6400& 5.23E-4  & 1.01 &  2.75E-4 & 1.02   \\
%  \hline
% % \vspace{0.05cm}
% \end{tabular}
% \caption{Results of convergence tests where the mesh size is denoted by $h$ 
% and the number of degrees of freedom per unknown by $n_{df}$. 
% The time step $\tau$ is set to h and errors are computed at final time $T=1$.}
% \label{tab:test_num_2D}
% \end{table}
%NEW ERROR COMPUTED WITH GAUSSIAN QUADRATURE DEGREE 2
\begin{table}[ht]
\centering
\begin{tabular}{|c|c|c|c|c|c|} \hline  
 \multicolumn{2}{|c|}{$L^2$-norm of error}   & \multicolumn{2}{|c|}{Water pressure $P_w$}
 & \multicolumn{2}{|c|}{Water saturation $S$}  \\ \hline \hline
  $h$  & $n_{df}$ & Error  & Rate  & Error  & Rate     \\ \hline
  $ 0.2$& 25 & 8.50E-3  & - &  4.21E-3 &  -      \\ \hline
$0.1$ & 100 & 4.15E-3 & 1.03  & 2.30E-3  & 0.87  \\  \hline
$0.05$ & 400 & 2.08E-3 & 1.00 &  1.14E-4 &  1.01 \\  \hline
$0.025$ & 1600 & 1.04-3 & 1.00  & 5.57E-4  & 1.03 \\  \hline
$0.0125$ & 6400& 5.23E-4  & 0.99 & 2.75E-4  &  1.02  \\
 \hline
\end{tabular}
\caption{Results of convergence tests where the mesh size is denoted by $h$ 
and the number of degrees of freedom per unknown by $n_{df}$. 
The time step $\tau$ is set to h and errors are computed at final time $T=1$.}
\label{tab:test_num_2D}
\end{table}
The results of the convergence tests are presented in 
Table~\ref{tab:test_num_2D}. The theoretical order of convergence, 
equal to one, is recovered for both unknowns which confirms the 
correct behavior of the algorithm.

\section{Conclusions}

This paper formulates a $\mathbb{P}_1$ finite element method to solve the immiscible two-phase flow problem in porous media.  The unknowns are the phase pressure and saturation, which are the preferred unknowns in industrial reservoir simulators. The numerical method employs mass lumping for integration and an upwind flux technique. As a consequence, the saturation is shown to be bounded between zero and one. The discrete approximations of pressure and saturation converge to the weak solution as the time step and mesh sizes tend to zero.

\bibliographystyle{siam}
\bibliography{ref}
\end{document}